\documentclass[10pt, reqno,amsmath,amsthm,amssymb,amscd]{amsart}
\usepackage{mathrsfs,amssymb, amscd,amsmath,amsthm}
\usepackage[enableskew,vcentermath]{youngtab}
\usepackage{multicol}\multicolsep=0pt
\usepackage{tikz}

\newcommand{\clr}{rgb:black,1;blue,4;red,1}
\newcommand{\wdot}{ node[circle, draw, fill=white, thick, inner sep=0pt, minimum width=4pt]{}}
\newcommand{\bdot}{ node[circle, draw, fill=\clr, thick, inner sep=0pt, minimum width=4pt]{}}

\newcommand{\up}{\uparrow}
\newcommand{\down}{\downarrow}

\newcommand{\wrd}{\langle\up,\down\rangle}

\newcommand{\p}[1]{|{#1}|}

\newcommand{\lcap}{
\begin{tikzpicture}[baseline = 3pt, scale=0.5, color=\clr]
        \draw[-,thick] (1,0) to[out=up, in=right] (0.53,0.5) to[out=left, in=right] (0.47,0.5);
        \draw[->,thick] (0.49,0.5) to[out=left,in=up] (0,0);
\end{tikzpicture}
}
\newcommand{\lcup}{
\begin{tikzpicture}[baseline = 6pt, scale=0.5, color=\clr]
        \draw[-,thick] (1,1) to[out=down, in=right] (0.53,0.5) to[out=left, in=right] (0.47,0.5);
        \draw[->,thick] (0.49,0.5) to[out=left,in=down] (0,1);
\end{tikzpicture}
}

\newcommand{\swap}{
\begin{tikzpicture}[baseline = 3pt, scale=0.5, color=\clr]
        \draw[->,thick] (0,0) to[out=up, in=down] (1,1);
        \draw[->,thick] (1,0) to[out=up, in=down] (0,1);
\end{tikzpicture}
}

\newcommand{\swapr}{
\begin{tikzpicture}[baseline = 3pt, scale=0.5, color=\clr]
        \draw[<-,thick] (0,0) to[out=up, in=down] (1,1);
        \draw[<-,thick] (1,0) to[out=up, in=down] (0,1);
\end{tikzpicture}
}
\newcommand{\rswap}{
\begin{tikzpicture}[baseline = 3pt, scale=0.5, color=\clr]
        \draw[->,thick] (0,0) to[out=up, in=down] (1,1);
        \draw[<-,thick] (1,0) to[out=up, in=down] (0,1);
\end{tikzpicture}
}

\newcommand{\caup}{
\begin{tikzpicture}[baseline = 5pt, scale=0.75, color=\clr]
                 \draw[<-,thick] (0.75,1) to[out=down,in=left] (1,0.6) to[out=right,in=down] (1.25,1);
          \draw[<-,thick] (1.25,0) to[out=up,in=right] (1,0.4) to[out=left,in=up] (0.75,0);\end{tikzpicture}}

\newcommand{\xdot}{
\begin{tikzpicture}[baseline = 3pt, scale=0.5, color=\clr]
        \draw[->,thick] (0,0) to[out=up, in=down] (0,1);
        \draw (0,0.4)\bdot;
\end{tikzpicture}
}
\newcommand{\xdotr}{
\begin{tikzpicture}[baseline = 3pt, scale=0.5, color=\clr]
        \draw[<-,thick] (0,0) to[out=up, in=down] (0,1);
        \draw (0,0.4)\bdot;
\end{tikzpicture}
}

\newcommand{\cldotr}{
\begin{tikzpicture}[baseline = 3pt, scale=0.5, color=\clr]
        \draw[<-,thick] (0,0) to[out=up, in=down] (0,1);
        \draw (0,0.4)\wdot;
\end{tikzpicture}
}
\newcommand{\cldot}{
\begin{tikzpicture}[baseline = 3pt, scale=0.5, color=\clr]
        \draw[->,thick] (0,0) to[out=up, in=down] (0,1);
        \draw (0,0.4)\wdot;
\end{tikzpicture}
}

\def\SUM#1#2{{\mbox{$\sum\limits_{#1}^{#2}$}}}
 \providecommand{\og}{``}
\providecommand{\fg}{''} \providecommand{\smfandname}{and}

\def\sc{\scriptstyle}
\def\ssc{\scriptscriptstyle}

\usepackage{amssymb}
\usepackage{mathrsfs,amssymb}
\usepackage{multicol}\multicolsep=0pt
\usepackage[enableskew,vcentermath]{youngtab}

\usepackage[sort]{cite}
\usepackage{xcolor,graphicx}

\def\crulefill{\leavevmode\leaders\hrule height 1pt\hfill\kern 0pt}
\long\def\QUERY#1{%
\leavevmode\newline%
\noindent$\star\star\star$\thinspace\textsf{Comment/Query}\crulefill\newline%
   \space #1\newline\hbox to 120mm{\crulefill}$\star\star\star$\newline}
\newtheorem{Theorem}{Theorem}[section]%[chapter] theorem number will %continue
\newtheorem{Lemma}[Theorem]{Lemma}
\newtheorem{Cor}[Theorem]{Corollary}
\newtheorem{Prop}[Theorem]{Proposition}

 \theoremstyle{definition}

\newtheorem{Defn}[Theorem]{Definition}

\newtheorem{Assumption}[Theorem]{Assumption}

\numberwithin{equation}{section}
\theoremstyle{definition}
\makeatletter
\def\enumerate{\begingroup\ifnum\@enumdepth>3\@toodeep\else
      \advance\@enumdepth\@ne
      \edef\@enumctr{enum\romannumeral\the\@enumdepth}%
      \topsep\z@\parskip\z@
      \list{\csname label\@enumctr\endcsname}
        {\@nmbrlisttrue\let\@listctr\@enumctr
         \parsep\z@\itemsep\z@\topsep\z@
         \setcounter{\@enumctr}{0}
         \def\makelabel##1{\hss\llap{\rm ##1}}
       }\fi}

\makeatother

\let\bar=\overline
\let\epsilon=\varepsilon
\def\({\big(}
\def\){\big)}

\def\C{\mathbb C}

\def\Z{\mathbb Z}

\def\0{\underline{0}}

\DeclareMathOperator{\End}{End}

{\catcode`\|=\active
  \gdef\set#1{\mathinner{\lbrace\,{\mathcode`\|"8000%
                                   \let|\midvert #1}\,\rbrace}}
  \gdef\seT#1{\mathinner{\Big\lbrace\,{\mathcode`\|"8000%
                                   \let|\midverT #1}\,\Big\rbrace}}
}
\def\midvert{\egroup\mid\bgroup}
\def\midverT{\egroup\,\Big|\,\bgroup}

\def\Set[#1]#2|#3|{\Big\{\ #2\ \Big| \
           \vcenter{\hsize #1mm\centering #3}\Big\}}

\def\up{{\boldsymbol\upsilon}}

\def\th{\theta}
\def\b{\beta}
\def\d{\delta}

\def\g{\gamma}

\def\l{\lambda}

\def\si{\sigma}

\def\sc{\scriptstyle}
\def\ssc{\scriptscriptstyle}
\def\dis{\displaystyle}

\def\bs{\backslash}

\def\Z{\mathbb{Z}}

\def\C{\mathbb{C}}

\def\es{\varepsilon}
\def\OT#1{{\raisebox{-8.5pt}{$\stackrel{{\mbox{\Large$\dis\otimes$}}}{\sc#1}$}}}

\def\mfg{{\mathfrak g}}
\def\fh{{\mathfrak h}}

\def\Set{{\rm Set}}

\def\OTIMES{{{\sc\!}\otimes{\sc\!}}}
\def\b{\beta}
\def\equa#1#2{
\begin{equation}\label{#1}#2\end{equation}}

\def\up{{\upsilon}}

\def\textsf#1{{\textit{#1}}}%

\begin{document}
\baselineskip15pt
\title{Affine walled Brauer-Clifford superalgebras}
\author{Mengmeng Gao, Hebing Rui, Linliang Song, Yucai Su}
\address{M.G.  Department  of Mathematics, East China Normal University,  Shanghai, 2000240, China}\email{g19920119@163.com}
\address{H.R.  School of Mathematical Science, Tongji University,  Shanghai, 200092, China}\email{hbrui@tongji.edu.cn}
\address{L.S.  School of Mathematical Science, Tongji University,  Shanghai, 200092, China}\email{llsong@tongji.edu.cn}
\address{Y.S.  School of Mathematical Science, Tongji University,  Shanghai, 200092, China}\email{ycsu@tongji.edu.cn}
\thanks{H. Rui is supported  partially by NSFC (grant No.  11571108).  L. Song is supported  partially by NSFC (grant No.  11501368).  Y. Su is supported partially by NSFC (grant No.  11371278 and  11431010)}

\date{\today}
\sloppy \maketitle

\begin{abstract}
In this paper, a notion of  affine walled Brauer-Clifford superalgebras $BC_{r, t}^{\rm aff} $ is introduced over an arbitrary integral domain $R$ containing $2^{-1}$.
These superalgebras can be considered as  affinization of  walled Brauer superalgebras in \cite{JK}.  By constructing infinite many homomorphisms from $BC_{r, t}^{\rm aff}$ to a class of level two walled Brauer-Clifford superagebras over  $\mathbb C$, we prove  that  $BC_{r, t}^{\rm aff} $ is free over $R$ with infinite rank. We explain  that any finite dimensional irreducible  $BC_{r, t}^{\rm aff} $-module over an algebraically closed field $F$ of characteristic not $2$   factors through a cyclotomic quotient of  $BC_{r, t}^{\rm aff} $, called  a
 cyclotomic  (or level $k$)  walled Brauer-Clifford superalgebra   $ BC_{k, r, t}$.
 Using a previous method on cyclotomic walled Brauer algebras in \cite{RSu1}, we prove   that
   $BC_{k, r, t}$   is  free over $R$ with  super rank
$(k^{r+t}2^{r+t-1} (r+t)!, k^{r+t}2^{r+t-1} (r+t)!)$ if and only if  it is admissible in the sense of   Definition~\ref{condi-kcycl}.
Finally, we prove that the degenerate affine (resp., cyclotomic)  walled
Brauer-Clifford superalgebras defined by Comes-Kujawa in \cite{CK} are  isomorphic to our affine (resp., cyclotomic)  walled Brauer-Clifford superalgebras.
\end{abstract}

\tableofcontents

\section{Introduction}

In his pioneer's work, Schur considered $V^{\otimes r}$,   the $r$-th tensor product of the natural module $V$  of the general linear group $GL_n(\mathbb C)$. This is a left $GL_n(\mathbb C)$-module such that $GL_n(\mathbb C)$ acts on $V^{\otimes r}$ diagonally.  There is a right action of the symmetric group $\Sigma_r$ on $V^{\otimes r}$, and two such actions commute with each other.  This enabled Schur to
establish a duality  between the polynomial representations of  $GL_n(\mathbb C)$ and   the representations of  symmetric groups over $\mathbb C$.
Later on, such a result was generalized by  Brauer~\cite{B},  Sergeev ~\cite{Ser} and Lehrer-Zhang\cite{GZ} and so on. In these cases, the $r$-th tensor product $V^{\otimes r}$  are considered where $V$ is the natural module of  a symplectic group  or an  orthogonal group or a queer Lie superalgebra $\mathfrak{q}(n)$ or  an orthosymplectic supergroup and so on.  The Brauer algebras and the Hecke-Clifford superalgebras etc  naturally appear as  endomorphism algebras of  $V^{\otimes r}$.

Koike~\cite{Koi}  considered the  mixed tensor  modules $V^{\otimes r}\OTIMES (V^*)^{\otimes t}$ of
the $r$-th power of the natural module $V$ and the $t$-th power of the dual natural module $V^*$ of the general linear group $GL_n(\mathbb C)$ for various $r,t\in\Z^{\ge0}$. This led him to introduce the notion of  walled Brauer algebras in \cite{Koi} (see also \cite{Tur}).
  Shader and Moon \cite{SM}  set up super  Schur-Weyl dualities between walled Brauer algebras and general linear Lie superalgebras, by studying mixed tensor modules of general linear Lie superalgebras ${\frak{gl}}_{m|n}$.  Brundan  and Stroppel  \cite{BS4}
established super Schur-Weyl dualities between level two
Hecke algebras $H_r^{p,q}$ and  ${\frak{gl}}_{m|n}$, by studying tensor modules $K_\l\OTIMES V^{\otimes r}$ of Kac modules $K_\l$ with the
$r$-th power $V^{\otimes r}$ of the natural module $V$ of ${\frak{gl}}_{m|n}$. This led
them to obtain various results including the celebrated 
one on Morita equivalences between blocks of categories of finite dimensional ${\frak{gl}}_{m|n}$-modules
and categories of
finite dimensional left modules over some generalized Khovanov's diagram algebras~\cite{BS}.
By studying tensor modules $M^{r,t}:=V^{\otimes r}\OTIMES K_\l\OTIMES  (V^*)^{\otimes t}$ of
Kac modules $K_\l$ with the $r$-th power of the natural module $V$ and the $t$-th power of the dual natural module $V^*$ of $\frak{gl}_{m|n}$, two of authors~\cite{RSu, RSu1} introduced  a new class of associative algebras, referred to  affine walled Brauer algebras, over a   commutative ring containing~$1$.\footnote{See \cite{Sa} (resp.,  \cite{BCNR}) where the affine walled Brauer algebra is defined over $\mathbb C$ (resp., over $R$  in terms of affine oriented  Brauer category.}
They established super Schur-Weyl dualities between level two walled Brauer algebras $B_{2,r,t}$ and general linear Lie superalgebras,
which enables them to
classify  highest weight vectors of ${\frak{gl}}_{m|n}$-modules $M^{r,t}$, and to determine
decomposition numbers of $B_{2,r,t}$ arising from super Schur-Weyl duality. In order to further study representation theory of queer Lie superalgebras and to establish
higher level mixed Schur-Weyl duality between queer Lie superalgebras and some associative algebras, a
 natural question is, what kind of algebras may come into the play if one replaces
general linear Lie superalgebras ${\frak{gl}}_{m|n}$ by queer Lie superalgebras $\mathfrak{q}(n)$.
This is one of the motivations of the present paper to introduce the notion of affine walled Brauer-Clifford superalgebras.
Another motivation comes from two of authors' work on the Jucys-Murphy elements of walled Brauer algebras in \cite{RSu}.

In 2014, Jung and Kang~\cite{JK}  introduced the notion of  walled Brauer superalgebras  or walled Brauer-Clifford superalgebras $BC_{r, t}$  so as to establish the mixed Schur-Weyl-Sergeev duality for queer Lie superalgebras $\mathfrak{q}(n)$. The superalgebra $BC_{r, t}$ can be considered as a generalization of a Hecke-Clifford superalgebra and a walled Brauer algebra. In the present paper, we construct Jucys-Murphy elements for $BC_{r, t}$ and study its properties in detail. Through these elements, we can  introduce the notion of affine Brauer-Clifford superalgebras $BC_{r, t}^{\rm aff}$ in a ring theoretical way.  Using arguments similar to those in \cite{RSu},
we construct  infinite many homomorphisms between the affine Brauer-Clifford superalgebra $BC_{r, t}^{\rm aff}$ and walled Brauer-Clifford superalgebras, and thus we are able to prove that the affine walled Brauer-Clifford superalgebra is free over $R$ with infinite rank if the defining parameter $\omega_1$ is zero. However, many affine walled Brauer-Clifford superalgebras which appear in the higher level mixed Schur-Weyl-Sergeev duality have non-zero defining parameter $\omega_1$. In order to overpass this, we consider level two mixed Schur-Weyl-Sergeev duality for $\mathfrak{q}(n)$ and prove that a class of level two walled Brauer-Clifford superalgebras over $\mathbb C$  with non-zero $\omega_1$ have required  super-dimensions. Using these level two walled Brauer-Clifford superalgebras instead of  walled Brauer-Clifford superalgebras used before, we can establish infinite many superalgebra homomorphisms and hence prove the freeness of the  affine  walled Brauer-Clifford superalgebra over $R$ no matter whether $\omega_1$ is zero or not.   This  is one of the points which is different from the work in \cite{RSu}.

It is a natural problem to give a  classification of finite dimensional irreducible $BC_{r, t}^{\rm aff} $-modules over an algebraically closed field of characteristic not $2$. By introducing cyclotomic quotients of $BC_{r, t}^{\rm aff}$, called cyclotomic walled Brauer-Clifford superalgebras, we are able to prove that
 any finite dimensional irreducible $BC_{r, t}^{\rm aff} $-module factors through  a cyclotomic walled Brauer-Clifford superalgebra.
  We define this superalgebra over $R$ and prove that it is free over $R$ with required rank if and only if it is admissible in the sense of  Definition~\ref{condi-kcycl}.
In a sequel, we will classify finite dimensional irreducible modules for affine and cyclotomic walled Brauer-Clifford superalgebras over an arbitrary (algebraically closed) field with characteristic not $2$.

We notice that degenerate affine walled Brauer-Clifford algebras and their cyclotomic quotients are also introduced by  Comes and Kujawa   via affine oriented  Brauer-Clifford supercategory  \cite{CK} in the early of July, 2017 (at that time we had obtained  our affine and cyclotomic walled Brauer-Clifford algebras), and they prove that their degenerate affine walled Brauer-Clifford superalgebras are free over $R$ with infinite rank.
We introduce affine walled Brauer-Clifford algebras or their cyclotomic quotients in terms of generators and defining relations in order to study  relationship between finite dimensional modules of affine walled Brauer-Clifford algebras or their cyclotomic quotients and those  of queer Lie superalgebras.
Motivated by \cite{BCNR}, we prove that a  degenerate   affine walled Brauer-Clifford superalgebra in \cite{CK}  is isomorphic to one of our affine walled Brauer-Clifford superalgebras. We also prove that their cyclotomic walled Brauer-Clifford superalgebra is isomorphic to one of  ours if both of these two superalgebras are admissible in the sense of Definition~\ref{condi-kcycl}. This also gives a proof of the freeness of Comes-Kujawa's cyclotomic walled Brauer-Clifford superalgebra. Such a result is not available  in \cite{CK}.

We organize our paper as follows. In section~2, we recall the notion of walled Brauer-Clifford superalgebras $BC_{r, t}$  in \cite{JK}. Several properties on the Jucys-Murphy elements of $BC_{r, t}$
are given.  This leads us to introduce the notion of affine walled Brauer-Clifford superalgebras in section~3.  We give infinite many homomorphisms between  $BC_{r, t}^{\rm aff} $ and  walled Brauer-Clifford superalgebras.   We also define cyclotomic walled Brauer-Clifford superalgebras. In section~4, we use higher level mixed Schur-Weyl-Sergeev dualities to prove that a class of level two walled Brauer-Clifford algebras with non-zero parameter $\omega_1$ have required super-dimensions over $\mathbb C$. In section~5,  we  construct infinite many homomorphisms between  $BC_{r, t}^{\rm aff} $  and   level two walled Brauer-Clifford superalgebras which appear in the higher level mixed Schur-Weyl-Sergeev dualities in section~4.  This in turn enables us  to  mimic  arguments in \cite{RSu}  to prove the freeness of   $BC_{r, t}^{\rm aff} $ over $R$. In particular,   $BC_{r, t}^{\rm aff} $ is of infinite super-rank.  In section~6, we prove that a cyclotomic walled Brauer-Clifford superalgebra is free over $R$ with required super rank if and only if it is admissible. Finally,  in section~7, we prove that the degenerate affine (resp.,  cyclotomic) walled Brauer-Clifford superalgebra in \cite{CK} is isomorphic to ours. In the later case, we need to assume that both superalgebras are admissible.

\section{Walled Brauer-Clifford superalgebras}

Throughout, we assume that  $R$ is  an integral  domain containing  $2^{-1}$.
Let $\Sigma_r$ be the {\it symmetric group} in $r$ letters.
 Then $\Sigma_r$ is generated by $s_1, \ldots, s_{r-1}$,  subject to the  relations (for all admissible $i$ and $j$):
\begin{equation}\label{symm} s_i^2=1, \ \ s_is_{i+1}s_i=s_{i+1}s_is_{i+1}, \ \ \text{ $s_is_j=s_js_i$, if  $|i-j|>1$.}\end{equation}
Each $s_i$ can be identified with  the simple reflection  $(i, i+1)$, where $(i, j)\in  \Sigma_r$, which switches $i , j$ and fixes others. In this paper, we always assume that
$\Sigma_r$ acts on the right of the set $\{1, 2, \ldots, r\}$.

The {\it  Hecke-Clifford algebra} $HC_{r}$ was introduced by Sergeev~\cite{Ser}  in order to study the $r$-th  tensor product of the natural module for the queer Lie superalgebras $\mathfrak{q}(n)$.  It is  the  associative $R$-superalgebra generated by even elements $s_1, \ldots, s_{r-1}$ and odd elements $c_1, \ldots, c_r$   subject to \eqref{symm} together with  the   following defining relations (for all admissible  $i, j$):
\begin{equation}\label{sera} c_i^2=-1, \ \ c_ic_j=-c_jc_i, \ \ \text{$w^{-1} c_i w=c_{(i)w}, \forall w\in \Sigma_r$}. \end{equation}

 In this paper, we denote $\Z_i=\{0, 1, \ldots, i-1\}$.   We always use $\alpha_j$ to denote the $j$-th coordinate of
$\alpha\in\Z_i^r$  for $1\le j\le r$. Let  $|\alpha|=\sum_{j=1}^r \alpha_j$.
 The Hecke-Clifford algebra $HC_r$ is free over $R$ with basis $\{c^{\alpha} w\mid w\in \Sigma_r, \alpha\in \Z_2^r\}$, where $c^\alpha=c_1^{\alpha_1}\cdots c_r^{\alpha_r}$  (see \cite{K}).  Since  $s_1, \ldots, s_{r-1}$ (resp., $c_1, \ldots, c_r$) are   even (resp.,~odd),
the even (resp., odd) subspace of  $HC_r$ is spanned by $\{c^{\alpha} w\mid w\in \Sigma_r, \alpha\in \mathbb Z_2^r, |\alpha|\in 2\mathbb Z\}$ (resp., $\{c^{\alpha} w\mid w\in \Sigma_r,  \alpha\in \mathbb Z_2^r, |\alpha|\not\in 2\mathbb Z\}$).  In particular, the super rank of $HC_r$ is $(2^{r-1} r!, 2^{r-1} r!)$.

We need $\overline{HC}_r$  as follows. As the $R$-superalgebra, it is generated by the even elements $\bar s_1, \ldots, \bar s_{r-1}$ and odd  elements $\bar c_1, \ldots, \bar c_r$ subject to the relations (for all admissible $i$ and $j$):  \begin{equation}\label{aseradual}\begin{aligned} &  \bar s_i^2=1, \ \ \bar s_i\bar s_{i+1}\bar s_i=\bar s_{i+1}\bar s_i\bar s_{i+1}, \ \ \text{and  $\bar s_i\bar s_j=\bar s_j\bar s_i$, if  $|i-j|>1$, }\\
 & \bar c_i^2=1, \ \ \bar c_i\bar c_j=-\bar c_j\bar c_i, \ \ \text{and $w^{-1} \bar c_i w=\bar c_{(i)w}, \forall w\in \Sigma_r$. }\\ \end{aligned}  \end{equation}
In this case, we identify $\bar s_i$  with $s_i$. If $\sqrt{-1}\in R$£¬ then $HC_r$ is the $\overline{HC}_r$ by setting  $\bar c_i=\sqrt{-1}c_i$ and $\bar s_i=s_i$.
Let   \begin{equation}\label{mmpy} L_1=0,  \text{ and  $ L_i=\mathfrak  L_i+c_i \mathfrak L_i c_i $, $2\le i\le r$,}\end{equation} where  $\mathfrak  L_i=\sum_{j=1}^{i-1} (j, i)$.
These elements, which  are known as Jucys-Murphy elements of $HC_r$, satisfy  the following relations for all admissible $i, j, k$:
\begin{equation}\label{jm-wba1}\begin{aligned} & L_i L_j=L_jL_i, \ \   s_i L_k=L_k s_i\ \  \text{ if $k\neq i, i+1$},\\
& s_iL_is_i=L_{i+1}-(1-c_ic_{i+1})s_i,\ \  c_iL_k=(-1)^{\delta_{i, k}} L_k c_i,\\
\end{aligned}\end{equation}
where $\delta_{i, k}=1$ if $i=k$, and $0$ otherwise. If we denote by $\bar L_1, \ldots, \bar L_r$ the Jucys-Murphy elements of $\overline {HC}_{r}$, then $\bar L_1=0$ and
$\bar  L_i= \bar {\mathfrak L}_i -\bar c_i \bar  {\mathfrak L}_i \bar  c_i$,
 where  $\bar {\mathfrak L}_j=\sum_{k=1}^{j-1} (\bar k, \bar j)$. In this case, we identify $\bar i$ with $i$ for all $1\le i\le r$.
 So, $\Sigma_r$ can be identified with the symmetric group on the set $\{\bar 1, \ldots, \bar r\}$, and  \eqref{jm-wba1} turns out to be
 \begin{equation}\label{jm-wba2}\begin{aligned} &\bar  L_i \bar L_j=\bar L_j\bar L_i, \ \   \bar s_i \bar L_k=\bar L_k \bar s_i\ \  \text{ if $k\neq i, i+1$},\\
& \bar s_i\bar L_i\bar s_i=\bar L_{i+1}-(1+
\bar c_i\bar c_{i+1})\bar s_i,\ \   {\bar c_i\bar L_k=(-1)^{\delta_{i, k}}} \bar L_k \bar c_i.\\
\end{aligned}\end{equation}
Considering  $-L_i$ (resp., $-\bar L_i$)  as  abstract generators $x_i$ (resp., $\bar x_i$) yields  the notion of the affine Hecke-Clifford algebra $HC_{r}^{\rm aff}$ (resp.,
${\overline{HC}}{\ssc\,}_{r}^{\rm aff}$) defined as follows.

The {\it affine Hecke-Clifford} algebra
$HC_{r}^{\rm aff}$ is the associative $R$-superalgebra generated by even elements  $s_1, \ldots, s_{r-1}, x_1$ and odd elements $c_1, \ldots, c_r$   subject to
 \eqref{symm}--\eqref{sera}, together with the following defining relations (for all admissible $i$ and $j$):
 \begin{equation} \label{asera} x_1x_2=x_2x_1, \ \ x_1c_i=(-1)^{\delta_{i, 1}} c_ix_1, \ \ s_jx_1=x_1s_j, \text{if $j\neq 1$, }\end{equation}
 where $x_2=s_1x_1s_1-(1-c_1c_2)s_1$. Later on, we need  $\overline{HC}_r^{\rm aff}$ as follows. As  the $R$-superalgebra, it is generated  by even elements  $\bar s_1, \ldots, \bar s_{r-1}, \bar x_1$ and odd elements $\bar c_1, \ldots, \bar c_r$   subject to \eqref{aseradual} together with the following defining relations (for all admissible $i$ and $j$):
 \begin{equation}\label{aseradual2}
 \bar x_1\bar x_2=\bar x_2\bar x_1, \ \  \bar x_1\bar c_i=(-1)^{\delta_{i, 1}} \bar c_i\bar x_1, \ \ \bar s_j\bar x_1=\bar x_1\bar s_j, \text{if $j\neq 1$, }
 \end{equation}
 where $\bar x_2=\bar s_1\bar x_1\bar s_1-(1+\bar c_1\bar c_2)\bar s_1$. Certainly,  $\overline {HC}_r^{\rm aff}$  is $HC_r^{\rm aff}$ if $\sqrt{-1}\in R$. For $1\le i\le r$, define  \begin{equation} \label{relsmurp} x_i=x_i'-L_i, \text{ and $\bar x_i=\bar x_i'-\bar L_i$,} \end{equation}
  where $x_{i}'=s_{i-1}\cdots s_1x_1 s_1\cdots s_{i-1}$, $x_1'=x_1$ and $\bar x_i'=\bar s_{i-1}\cdots \bar s_1 \bar x_1 \bar s_1\cdots \bar s_{i-1}$, $\bar x_1'=\bar x_1$.
  Then we have the following relations for all admissible $i$ and $j$:
 \begin{equation} \label{comm-hc} \begin{aligned} & x_{i+1}=s_{i}x_{i}s_{i}-(1-c_ic_{i+1}) s_i, \text{ and $x_ix_j=x_jx_i$,}\\
 & \bar x_{i+1}=\bar s_{i}\bar x_{i}\bar s_{i}-(1+\bar c_i\bar c_{i+1}) \bar s_i, \text{ and $\bar x_i\bar x_j=\bar x_j\bar x_i$.}
 \\ \end{aligned}
 \end{equation}
   For all  $\alpha\in \mathbb N^r$, define $x^\alpha= x_1^{\alpha_1}\cdots x_r^{\alpha_r}$ and $\bar x^\alpha=\bar x_1^{\alpha_1}\cdots \bar x_r^{\alpha_r}$.  It is proved in \cite{K} that $HC_r^{\rm aff}$ has basis $\{x^{\alpha}c^{\beta} w \mid w\in \mathfrak S_r, \alpha\in \mathbb N^r,
 \beta\in \Z_2^r\}$. The even (resp., odd ) subspace of $HC_r^{\rm aff}$ is spanned by all $ x^\alpha c^{\beta} w$ such that
$|\beta|\in 2\mathbb Z$ (resp., $|\beta|\not\in 2\mathbb Z $). Similar results hold for $\overline{HC}_r^{\rm aff}$.

We are going to recall the definition of  the  {\it walled Brauer-Clifford superalgebra} $BC_{r, t}$. This superalgebra was introduced by Jung and Kang in \cite{JK} so as to study the mixed tensor product of the natural module and its linear dual for the queer Lie superalgebra $\mathfrak{q}(n)$. The original $BC_{r, t}$  is defined via
{\it $(r, t)$-superdiagrams} in \cite{JK}. In this paper, we use its equivalent definition.

\begin{Defn}\cite[Theorem~5.1]{JK} \label{wsera}  The walled  Brauer-Clifford superalgebra $ BC_{r,t}$ is the associative $R$-superalgebra  generated by even generators $e_1$,  $s_1, \ldots, s_{r-1}$, $\bar s_{1}, \ldots, \bar s_{t-1}$, and odd generators  $c_1, \ldots, c_r, \bar c_1, \ldots,  \bar c_{t}$ subject to  \eqref{symm}--\eqref{aseradual} together with the following defining relations for all admissible $i, j$:
 \begin{multicols}{2}
\begin{enumerate}
\item [(1)] $e_1 c_1=e_1\bar c_1 $, $c_1 e_1=\bar c_1 e_1 $,
\item  [(2)]  $\bar s_j c_i=c_i \bar s_j$, $ s_i \bar c_j=\bar c_j  s_i$,
\item [(3)] $c_i \bar c_j=-\bar c_j c_i$, $s_i\bar s_j=\bar s_j s_i$,
\item [(4)]  $e_1^2=0$,
\item [(5)] $e_1 s_1 e_1=e_1=e_1\bar s_1 e_1$,
\item [(6)] $s_i e_1=e_1 s_i$,   $\bar s_i e_1=e_1 \bar s_i$, if $i\neq 1$,
\item [(7)]$e_1s_1\bar s_1 e_1 s_1=e_1s_1\bar s_1 e_1 \bar s_1$,
\item [(8)] $s_1 e_1s_1\bar s_1 e_1 =\bar s_1 e_1s_1\bar s_1 e_1$,
\item [(9)] $c_i e_1=e_1 c_i$ and  $\bar c_i e_1=e_1\bar c_i$, if $i\neq 1$,
\item [(10)] $e_1c_1e_1=0=e_1\bar c_1 e_1$.
\end{enumerate}
\end{multicols}
\end{Defn}

\begin{Lemma}\label{tau}  There is a unique $R$-linear anti-involution $\tau: BC_{r, t}\rightarrow BC_{r, t}$, which fixes    all of its generators.\end{Lemma}
\begin{proof} It follows from Definition~\ref{wsera}, immediately.\end{proof}

It is known that  the subalgebra of $BC_{r, t}$  generated by even generators  $s_1, \ldots, s_{r-1}$, $\bar s_1, \ldots, \bar s_{t-1}$ and $e_1$ is isomorphic to the walled Brauer algebra $ B_{r, t}(0)$ in \cite{Koi, Tur}. This enables us  to  use freely
the results on $B_{r, t}(0)$ in  \cite{RSu} so as to simplify our presentation.
Write   $s_{i, j}=s_i s_{i+1, j}$ if $i<j$ and $s_{i, i}=1$ and $s_{i, j}=s_{i, j+1} s_{ j}$ if $i>j$. Similarly, we have $\bar s_{i, j}$'s, etc.
Following \cite{RSong}, define $D_{r, t}^f=\{1\}$ if $f=0$ and
 \begin{eqnarray} &\!\!\!\!\!\!\!\!\!\!\!\!\!\!\!\!\!\!\!\!\!\!\!\!\!\!\!\label{drt} {D}_{r,t}^f=\Big\{ s_{f,i_f} \bar s_{f, j_f} \cdots s_{1,i_1}\bar s_{1,{j_1}}\,\Big|&\!\!\! k \le {j_k}\le t, 1\le k\le f,\nonumber\\
 &\!\!\!\!\!\!\!\!\!\!\!\!\!\!\!\!\!\!\!\!\!\!\!\!\!\!\!&\!\!\!{1 \le i_1< i_{2} < \ldots <i_f\le r} \Big\} \text{ if $0<f\le \min\{r, t\}$.}\end{eqnarray}

\begin{Defn}Define  $e^0=1$ and $e^f=e_1e_2\cdots e_f$ if $0<f\le \min\{r, t\}$, where  $e_i=e_{i, i}$  and    $e_{i, j}=(\bar 1, \bar j) (1, i)  e_1  (1, i)(\bar 1, \bar j)$
for all admissible  $i, j$.\end{Defn}

 \begin{Theorem}\cite[Theorem~5.1]{JK}\label{wbhsa321} The walled Brauer-Clifford superalgebra $BC_{r, t}$ has  $R$-basis
 \begin{equation}\label{bcabasis} \left\{ c^{\alpha} d_1^{-1}e^f w d_2 \bar c^{\beta}\mid 0\le f\le \min\{r, t\}, w\in \Sigma_{r-f}\times { {\Sigma}}_{\overline{t-f}}, d_1, d_2\in  D_{r, t}^f, (\alpha, \beta)\in \Z_2^r\times \Z_2^t \right \}.\end{equation} In particular, the super rank of $BC_{r, t}$ is $(2^{r+t-1} (r+t)!, 2^{r+t-1} (r+t)!)$.
 \end{Theorem}
\begin{proof} The basis of $BC_{r, t}$ given in \eqref{bcabasis} is a refinement of  $X$  given in the proof of \cite[Theorem~5.1]{JK}. We remark that  each $d_1^{-1}e^f w d_2$ corresponds to a unique walled Brauer diagram in \cite{RSu}.
  \end{proof}

\begin{Cor}\label{wsba-1}  For any positive integer $k$,    the subalgebra of $BC_{k+r, k+t}$ generated by even elements $e_{k+1}, s_{k+1}, \ldots$, $s_{k+r-1}$, $\bar s_{k+1}, \ldots, \bar s_{k+t-1}$ and odd elements  $c_{k+1}, \bar c_{k+1}$ is isomorphic to $BC_{r, t}$.    \end{Cor}
\begin{proof} Easy exercise using Theorem~\ref{wbhsa321} and Definition~\ref{wsera}. \end{proof}

\begin{Lemma}\label{spannwbc} Let $BC_{k-1, k-1}$ be the  subalgebra of $BC_{k, k}$ generated by $e_1, s_1, \ldots, s_{k-2}$, $\bar s_1,\ldots,  \bar s_{k-2}$ and  $c_1,  \bar c_1$.  Then $e_k BC_{k, k}$ is a left  $BC_{k-1, k-1}$-module spanned by all $ e_k c_k^\sigma s_{k, j}\bar s_{k, l}$ such that
 $\sigma\in \Z_2$ and $1\le  j, l\le k$.
  \end{Lemma}
\begin{proof} It is enough to prove that the left $BC_{k-1, k-1}$-module   $V_k$  spanned all $ e_k c_k^\sigma s_{k, j}\bar s_{k, l}$
is a right $BC_{k, k}$-module. If so, then $V_k=e_k BC_{k, k}$ by the fact that   $e_k\in V_k$. We have  $V_k s_i\subset V_k$ and $V_k c_1\subset V_k$
 since $$e_k c_k^\sigma s_{k, j}\bar s_{k, l}s_i=\begin{cases}  {s_{i} e_k c_k^\sigma s_{k, j}\bar s_{k, l}} &\text{if $i<j$,}\\
 {e_k c_k^\sigma s_{k, j+1}\bar s_{k, l}} &\text{if $i=j$,}\\
{s_{i-1}e_k c_k^\sigma s_{k, j}\bar s_{k, l}} &\text{if $i>j$, }\\ \end{cases} \ \text{ and } \
e_k c_k^\sigma s_{k, j}\bar s_{k, l} c_1=\begin{cases} \epsilon c_1  e_k c_k^\sigma s_{k, j}\bar s_{k, l} &\text{if $j>1$,}\\
  e_k c_k^{\sigma+1} s_{k, j}\bar s_{k, l} &\text{if $j=1$, }\\
  \end{cases}$$
where $\epsilon=1$ (resp., $-1$)  if $\sigma=0$ (resp., $1$).  Similarly,  $V_k \bar s_i\subset V_k$ and $V_k \bar c_1\subset V_k$. Finally,  $V_k e_1\subset V_k$ since
\begin{equation}\label{step1} e_k c_k^\sigma s_{k, j}\bar s_{k, l}e_k =\begin{cases} 0 &\text{if $j=k=l$,}\\
c_{k-1}^\sigma s_{k-1, j} e_k &\text{if $l=k>j$,}\\
 \bar c_{k-1}^\sigma \bar s_{k-1, l} e_k &\text{if $j=k>l$,}\\
e_{k-1}  c_{k-1}^\sigma s_{k-1, j}\bar s_{k-1, l}e_k &\text{if $j, l<k$.}\\
\end{cases}
\end{equation}
\end{proof}

\begin{Prop}\label{ese1} We have  $e_k BC_{k, k} e_k=e_{k} BC_{k-1, k-1}$ for all $k\ge 2$ and $e_1 BC_{1, 1} e_1=0$. \end{Prop}
\begin{proof} We have $e_k BC_{k, k} e_k\subseteq e_{k} BC_{k-1, k-1}$  by   Lemma~\ref{spannwbc} and \eqref{step1}. When $k\ge 2$,  the inverse inclusion follows from the  equation   $e_k=e_ks_{k-1}e_k$ and $e_k x=x e_k$ for any $x\in BC_{k-1,k-1}$.
\end{proof}

\begin{Defn}\label{RSu-GRSS} For all admissible  $i, j$, let $  y_i=\sum_{j=1}^{i-1} (e_{i, j}+\bar e_{i, j})-L_i$, and $\bar y_i=\sum_{j=1}^{i-1} (e_{j, i}-\bar e_{j, i})-\bar L_i$, where $\bar e_{i, j}=c_i e_{i, j} c_i$. Then
$y_i={\mathfrak y_i}+c_i {\mathfrak y_i} c_i$ and $\bar y_j= {\bar {\mathfrak y}_j}-\bar c_j {\bar {\mathfrak y}_j} \bar c_j$,
where $\mathfrak y_i$ (resp.,  ${\bar {\mathfrak y}}_j$)  is   $y_i$ (resp.,  $\bar y_j$)  in \cite[(3.5)]{RSu} in the case $\delta_1=0$.
So, \begin{equation}\label{fraky} \mathfrak y_i=\mbox{$\sum\limits_{j=1}^{i-1}$} (e_{i, j}-(j, i)), \text{ and } \bar {\mathfrak y}_i=
\mbox{$\sum\limits_{j=1}^{i-1}$} (e_{j, i}-(\bar j, \bar i)).\end{equation}
\end{Defn}

\begin{Lemma}\label{usef} With the notations above, the following results hold in   $ BC_{r, t}$ $($for all admissible  $i, j)$,
\begin{multicols}{2}
\begin{enumerate}
\item [(1)] $s_j y_i=y_i s_j$, $\bar s_j \bar y_i=\bar y_i\bar s_j$ if $j\neq i-1, i$,
\item  [(2)] $s_j\bar y_i=\bar y_i s_j$, $\bar s_j  y_i=y_i\bar s_j$ if $j\neq i-1$,
\item [(3)] $y_i c_i=-c_i y_i$, $\bar y_i \bar c_i=-\bar c_i \bar y_i$,
\item [(4)] $y_i c_j=c_j y_i$, $\bar y_i \bar c_j=\bar c_j\bar y_i$ if $i\neq j$,
\item [(5)] $y_i\bar c_j=\bar c_j y_i$, $\bar y_i c_j=c_j\bar y_i$ if $j\geq i$,
\item [(6)] $y_{i} y_{i+1}=y_{i+1}y_i$, $\bar y_i \bar y_{i+1}=\bar y_{i+1}\bar y_i$,
\item [(7)] $y_i(e_i+\bar y_i-\bar e_i)=(e_i+\bar y_i-\bar e_i)y_i $,
\item [(8)] $e_{i} \bar y_i=e_i( L_i-\bar L_i)$, $e_{i} y_i=e_i(\bar L_i-L_i)$,
\item [(9)] $e_i s_i y_is_i=s_i y_is_i e_i$,  $e_j \bar s_j \bar y_j\bar s_j=\bar s_j \bar y_j\bar s_j e_j$,
\item [(10)] $y_i\tilde y_i =\tilde y_i y_i$, $\bar y_i {\tilde {\bar y}}_i  ={\tilde {\bar y}}_i  \bar y_i$,
\item [(11)] $ e_i y_i^k c_i e_i=0$,  $\forall k\in \mathbb N$,
\item [(12)] $e_i y_i^{2n} e_i=0 $, $e_i \bar y_i^{2n} e_i=0$,  $\forall n\in \mathbb N$,
\item [(13)] $e_i y_i e_i=e_i \bar y_i e_i=0$,
 \end{enumerate}\end{multicols}\noindent
 where $\tilde y_i=s_iy_is_i-(1-c_ic_{i+1})s_i$ and ${\tilde {\bar y}}_i=\bar s_i\bar y_i\bar s_i-(1+\bar c_i\bar c_{i+1})\bar s_i$.\end{Lemma}
\begin{proof} We assume $\sqrt{-1}\in R$. Then  $BC_{r, t}\cong BC_{t, r}$. It is reasonable since we can embed $R$ into a larger integral domain containing $\sqrt{-1}$.
 The required isomorphism sends (a): $\sqrt{-1}  c_i $ (resp.,  $\sqrt{-1}\bar c_j$) in  $BC_{r, t}$ to $\bar c_i$ (resp., $c_j$) in $BC_{t, r}$; (b): $e_1$ to $e_1$; (c): $s_i$ (resp., $\bar s_j$) in $BC_{r, t}$ to $\bar s_i$ (resp., $s_j$) in $BC_{t, r}$.
 So, it suffices  to verify one of equations  in (1)--(6), (8)--(13) except (11).

(1) If $j\neq i, i-1$, then $s_j c_i=c_is_j$  and $s_j\mathfrak y_i=\mathfrak y_i s_j$  by \eqref{sera} and  \cite[Lemma~3.3(6)]{RSu}.    So, $$s_j y_i=s_j(\mathfrak y_i+c_i\mathfrak y_i c_i)= (\mathfrak y_i+c_i\mathfrak y_i c_i)s_j=y_i s_j.$$

 (2) If  $j\neq i-1$, then  $s_j\bar {\mathfrak y}_i=\bar {\mathfrak y}_i s_j$ by \cite[Lemma~3.3(7)]{RSu}.  By Definition~\ref{wsera}(2), $$ s_j\bar y_i=s_j(\bar{\mathfrak y}_i-\bar c_i\bar{\mathfrak y}_i\bar c_i)=s_j\bar y_i.$$

 (3) Since $c_i^2=-1$, we have   $c_iy_i=c_i(\mathfrak y_i+c_i\mathfrak y_i c_i)=-\mathfrak y_i c_i+c_i \mathfrak y_i=-y_i c_i$.

 (4) By \eqref{sera}, \eqref{jm-wba1} and Definition~\ref{wsera}(9),  $c_ic_j=-c_jc_i$,  $c_j L_i=L_ic_j$  and $e_{i, k}c_j=c_je_{i, k}$ if $i\neq j$.
By  Definition~\ref{RSu-GRSS}, we have  $y_ic_j=c_jy_i$.

  (5)  If $j>k$, then $e_{i, k} \bar c_j=\bar c_j e_{i,k}$. By Definition~\ref{wsera}(2)--(3) and Definition~\ref{RSu-GRSS}, we have $y_i \bar c_j=\bar c_j y_i$.

  (6) Since  $\mathfrak y_i=\sum_{j=1}^{i-1} (e_{i, j}-(j, i))$ (see \eqref{fraky}), we have  $\mathfrak y_i c_{i+1}=c_{i+1}\mathfrak  y_i$. By \cite[Lemma~3.3(9)]{RSu},  $\mathfrak y_i\mathfrak y_{i+1}=\mathfrak y_{i+1}\mathfrak y_{i}$. Thus,
    $$\begin{aligned}  y_i y_{i+1} &=(\mathfrak y_i+c_i\mathfrak y_ic_i) y_{i+1}\overset{(4)}=\mathfrak y_i y_{i+1}+c_i\mathfrak y_i y_{i+1}c_i\\
&=  \mathfrak y_i\mathfrak y_{i+1}+c_{i+1}\mathfrak y_i\mathfrak y_{i+1}c_{i+1}+c_i \mathfrak y_i \mathfrak y_{i+1} c_i+c_ic_{i+1} \mathfrak y_i\mathfrak y_{i+1} c_{i+1} c_i.\\ \end{aligned} $$
 Applying the anti-involution $\tau$ on the above equation yields $ y_i y_{i+1}=y_{i+1} y_i$.

(7) We have
$$ \begin{aligned} y_i (e_i-\bar e_i+\bar y_i)&=(\mathfrak y_i+c_i\mathfrak y_i c_i) (e_i+{\bar {\mathfrak y}}_i-\bar c_i {\bar {\mathfrak y}}_i  \bar c_i-c_i e_ic_i)\\
& =\mathfrak y_i(e_i+{\bar {\mathfrak y}}_i)-\bar c_i \mathfrak y_i (e_i+{\bar {\mathfrak y}}_i)\bar c_i +c_i \mathfrak y_i (e_i+{\bar {\mathfrak y}}_i) c_i-
\bar c_i c_i \mathfrak y_i (e_i+{\bar {\mathfrak y}}_i) c_i\bar c_i.  \\
\end{aligned}$$
Applying the anti-involution $\tau$ on the above equation and using  $(e_i+{\bar {\mathfrak y}}_i )\mathfrak y_i = \mathfrak y_i (e_i+{\bar {\mathfrak y}}_i )$ (see   \cite[Lemma~3.3(4)]{RSu}) yields (7).

(8) By  \cite[Lemma~3.3(1)]{RSu},  $e_i\mathfrak y_i=e_{i} (-\mathfrak L_i+{\bar {\mathfrak L}}_i)$. So,
$$\begin{aligned} e_iy_i &=e_i\mathfrak y_i+e_i c_i\mathfrak y_i c_i=e_i\mathfrak y_i(1+\bar c_i c_i)=-e_i(\mathfrak L_i-\bar{\mathfrak L}_i)(1+\bar c_i c_i)=-e_i\mathfrak L_i+e_i{\bar {\mathfrak L}}_i
-e_i(\mathfrak L_i-{\bar {\mathfrak L}}_i)\bar c_i c_i\\ & =-e_i (\mathfrak L_i+ c_i \mathfrak L_i c_i)+e_i(\bar {\mathfrak L}_i-
\bar c_i \bar {\mathfrak L}_i \bar c_i )=-e_i
(L_i-\bar L_i).\\ \end{aligned}
$$

   (9) By \cite[Lamma~3.3(5)]{RSu}, $e_i s_i\mathfrak y_i s_i=  s_i\mathfrak y_i s_i e_i$. So,
 $$ e_i s_i y_i s_i=e_i s_i c_i\mathfrak y_i c_i s_i+  e_i s_i \mathfrak y_i  s_i
=c_{i+1} e_i s_i\mathfrak y_i s_i c_{i+1}+ s_i \mathfrak y_i s_i  e_i=s_i y_i s_i e_i.$$

(10) We define $\textbf{m}_i=s_i\mathfrak y_i s_i-s_i$. By \cite[Lemma~3.3(5)]{RSu}), $\textbf{m}_i \mathfrak y_i=\mathfrak y_i \textbf{m}_i $.
So,
\begin{equation}\label{yy123} \begin{aligned}&  y_i\tilde y_i =(\mathfrak y_i+c_i\mathfrak y_ic_i) (s_i(\mathfrak y_i+c_i\mathfrak y_ic_i)s_i-(1-c_ic_{i+1})s_i)\\
=&\mathfrak y_i \textbf{m}_i + \mathfrak y_i (s_ic_i \mathfrak y_i c_i s_i+c_ic_{i+1}s_i)+c_i\mathfrak y_ic_i (s_i (\mathfrak y_i+c_i\mathfrak y_i c_i) s_i-(1-c_ic_{i+1})s_i)\\
=&\mathfrak y_i  \textbf{m}_i  +c_{i+1} \mathfrak y_i  \textbf{m}_i c_{i+1}+c_i\mathfrak y_i c_i \textbf{m}_i +
c_i\mathfrak y_i c_i(s_ic_i\mathfrak y_ic_is_i +c_ic_{i+1}s_i )\\
=& \mathfrak y_i  \textbf{m}_i +c_{i+1}\mathfrak y_i \textbf{m}_i c_{i+1}+c_i\mathfrak y_ic_is_i\mathfrak y_is_i-
c_i\mathfrak y_i c_i s_i -c_i\mathfrak y_ic_{i+1}s_i+c_i\mathfrak y_i c_i s_i c_i \mathfrak y_i c_i s_i\\
=& \mathfrak y_i \textbf{m}_i  +c_{i+1}\mathfrak y_i  \textbf{m}_i c_{i+1}+ c_i\mathfrak y_i \textbf{m}_i c_i -c_i\mathfrak y_ic_{i}s_i-c_ic_{i+1}
\mathfrak y_is_i\mathfrak y_i s_i  c_ic_{i+1}\\
=&\mathfrak y_i \textbf{m}_i +c_{i+1}\mathfrak y_i  \textbf{m}_i c_{i+1}+ c_i\mathfrak y_i \textbf{m}_i c_i+c_{i+1}c_i\mathfrak y_i \textbf{m}_i c_ic_{i+1}.
\end{aligned}\end{equation}
Applying the anti-involution $\tau$ on \eqref{yy123}  yields $\tilde y_i y_i= y_i\tilde y_i$.

   (11) If  $k$ is odd,  then
$$e_i y_i^k c_i e_i=e_i y_i^k \bar c_i e_i\overset{(5)}=e_i\bar c_i y_i^k  e_i\overset{(3)} =-e_i y_i^k c_i e_i, $$
 forcing   $ e_i y_i^k c_i e_i=0$.  If $k$ is even, then
 \begin{equation}\label{eybary} e_{i} y_i^{k-1} \bar y_i=e_{i} y_{i}^{k-1}(\bar y_i+e_i-\bar e_i)-e_iy_i^{k-1} e_i+e_iy_i^{k-1}\bar  e_i
 \overset{(7)}=e_{i}(\bar y_i+e_i-\bar e_i)y_{i}^{k-1}-e_iy_i^{k-1} e_i+e_iy_i^{k-1}\bar  e_i.\end{equation}
So, $e_{i} y_i^{k-1} \bar y_i c_ie_i=e_i \bar y_i y_i^{k-1} c_i e_i$.
On the other hand,  $e_{i} y_i^{k-1} \bar y_i c_ie_i=-e_i c_i
y_i^{k-1}\bar y_i e_i=e_{i} c_i y_i^k e_i $ and $e_i \bar y_i y_i^{k-1} c_i e_i=-e_i y_i^k c_ie_i=-e_{i} c_i y_i^k e_i$. So, $ e_i y_i^k c_i e_i=0$ for even $k$.

 (12) We  have $e_k y_k^{2n} e_k=0$ since
$$e_k y_k^{2n} e_k =e_k y_k^{2n} \bar c_k^2 e_k\overset{(5)}=e_k \bar c_k  y_k^{2n} c_k e_k
=  e_k c_k y_k^{2n}c_k  e_k\overset{(3)}= e_k y_k^{2n} c_k^2 e_k=-e_k y_k^{2n} e_k. $$

(13) If $j<i$, then  $e_i c_i e_{i, j} c_i e_i=e_i \bar c_i e_{i, j} \bar c_i e_i=e_{i} e_{i, j} e_i=e_i$ and $e_ic_i(j, i)c_ie_i=e_i \bar c_i (j, i) \bar c_i e_i=e_i$.
So, $$e_i y_i e_i=2 e_i\mathfrak y_i e_i=2\mbox{$\sum\limits_{j=1}^{i-1}$} e_i(e_{i, j}-(i, j))e_i=0.$$
\end{proof}

\begin{Cor}\label{equal123} There is a unique element $\omega_{a, k}\in BC_{k-1, k-1}$ such that $e_k y_k^a e_k=\omega_{a, k} e_k$. Similarly, there is a unique element $\bar \omega_{a, k}\in BC_{k-1, k-1}$ such that $e_k {\bar y_k}^a e_k=\bar \omega_{a, k} e_k$. Moreover, $\omega_{2n,k}=\bar \omega_{2n,k}=0$.
\end{Cor}
\begin{proof} The existence of an  $\omega_{a, k}$ follows from  Proposition~\ref{ese1} and the uniqueness of such an element follows from Theorem~\ref{wbhsa321}.
 Finally, we have $\omega_{2n,k}=\bar \omega_{2n,k}=0$ by  Lemma~\ref{usef}(12).\end{proof}

\begin{Lemma}\label{y1} For   $n\in \mathbb N$,   $e_i \bar y_i^{2n+1}\!=\!\sum_{j=0}^n a^{(i)}_{2n+1, j} e_i y_i^{2j+1}$ for some $a_{2n+1,j}^{(i)}\in$ $
R[\omega_{3, i},  \ldots,\omega_{2n-1, i} ]$ such that
\begin{enumerate}\item [(1)] $a_{2n+1, n}^{(i)}=-1$,\vskip2pt
 \item [(2)] $a_{2n+1, j}^{(i)}=a_{2n-1, j-1}^{(i)}$, $1\le j\le n$,
 \item [(3)]  $a^{(i)}_{2n+1, 0}=\sum_{j=1}^{n-1} a^{(i)}_{2n-1, j} \omega_{2j+1, i}$.\end{enumerate}
\end{Lemma}

\begin{proof} When $n=0$, we have $e_i(y_i+\bar y_i)=0$ by   Lemma~\ref{usef}(8).
So, $a_{1, 0}^{(i)}=-1$.
In general,
 we have
$$\begin{aligned}  e_i y_i^{2j-1} \bar y_i^2& = e_i y_i^{2j-1} ( \bar y_i+e_i-\bar e_i) \bar y_i-e_i y_i^{2j-1} e_i \bar y_i\ \  (\text{by Lemma~\ref{usef}(11)})\\
&  =e_i (\bar y_i+e_i-\bar e_i) y_i^{2j-1} \bar y_i +\omega_{2j-1,i} e_i  y_i \ \  (\text{by Lemma~\ref{usef}(7)})
\\ & =-e_i y_i^{2j} \bar y_i+\omega_{2j-1,i} e_i  y_i\ \ (\text{ by Lemma~\ref{usef}(8)}).  \end{aligned} $$ Similarly, using $\omega_{2j, i}=0$ yields $e_i y_i^{2j} \bar y_i = -e_i y_i^{2j+1}$.  So, $e_i  y_i^{2j-1} \bar y_i^2=e_i y_i^{2j+1}+w_{2j-1,i} e_i  y_i$. By inductive assumption on $n$ and $\omega_{1, i}=0$, we have the result, immediately.
\end{proof}

\begin{Lemma}\label{y2} For positive integers  $n$,  $e_i \bar y_i^{2n}\!=\!\sum_{j=0}^n a^{(i)}_{2n, j} e_i y_i^{2j}$ for some  $a_{2n,j}^{(i)}\!\in\!
R[\omega_{3, i},  \ldots,\omega_{2n-1, i} ]$ such that
\begin{enumerate}\item [(1)] $a_{2n, n}^{(i)}=1$,
 \item [(2)] $a_{2n, j}^{(i)}=a_{2n-2, j-1}^{(i)}$, $1\le j\le n$,
 \item  [(3)] $a^{(i)}_{2n, 0}=\sum_{j=1}^{n-1} a^{(i)}_{2n-2, j} \omega_{2j+1, i}$.\end{enumerate}
\end{Lemma}
\begin{proof} We have $e_i y_i^{2j} \bar y_i^2=-e_iy^{2j+1} \bar y_i=e_iy_i^{2j+2}+\omega_{2j+1, i} e_1$. By inductive assumption  on $n$ and $\omega _{1,i}=0$, we immediately have the result. \end{proof}

 We can assume $k \ge 2$ (resp., $n\ge 2$)  in the Lemma~\ref{free1}  since  $y_1=\bar y_1=0$ (resp., $\omega_{1, k}=\bar \omega_{1, k}=0$ by Lemma~\ref{usef}(13)).
\begin{Lemma}\label{free1} We have $\bar \omega_{2n-1, k}\in $ $ R[\omega_{3, k}, \ldots, \omega_{2n-1, k}]$ if  $k, n\in\Z^{\ge2}$. Furthermore, both  $\omega_{2n-1, k}$
and $\bar\omega_{2n-1, k}$ are central in $BC_{k-1, k-1}$.\end{Lemma}

\begin{proof}
By Lemma~\ref{y1} and inductive assumption on $k$, we have the first statement.
To
prove the second, note 
that any $h\in \{e_1, s_1, \ldots, s_{k-2}, c_1\}$ commutes with $e_k, y_k$. So, $e_k (h \omega_{2n-1, k})=e_k(\omega_{2n-1, k} h)$. By Theorem~\ref{wbhsa321}, $h \omega_{2n-1, k}=\omega_{2n-1, k} h$.
Finally, we need to check that $e_k (h \omega_{2n-1, k})=e_k(\omega_{2n-1, k} h)$ for any $h\in \{\bar s_1, \ldots, \bar s_{k-2}, \bar c_1\}$. In this case, we use Lemma~\ref{y1}.
More explicitly, we can use $\bar y_k$ instead of $y_k$
in $e_k y_k^{2n-1} e_k$. Thus  $h \omega_{2n-1, k}=\omega_{2n-1, k} h$, as required.
\end{proof}

In the following, we define
\begin{equation} \label{hhk} h_k =y_k+e_k+\bar e_k, \text{ and } \bar h_k=\bar y_k+e_k-\bar e_k, \text{ for all admissible $k$.}\end{equation}
\begin{Lemma}\label{syk} For $k,a\in\Z^{\ge 1}$, we have $$s_k y_{k+1}^a=h_k^a s_k-\mbox{$\sum\limits_{b=0}^{a-1}$} h_k ^{a-1-b}y_{k+1}^b+\mbox{$\sum\limits_{b=0}^{a-1}$} (-1)^{a-b} c_kc_{k+1} h_k^{a-b-1} y_{k+1}^b,$$
where $h_k$ is given in \eqref{hhk}.
\end{Lemma}
\begin{proof} It is easy to verify the result by induction on $a$.
\end{proof}

\begin{Lemma}\label{dzjk} Suppose $1\le j\le k-1$. Define
$ z_{j, k}=s_{j, k-1} h_{k-1}  s_{k-1, j}$,
 and $ \bar z_{j, k}=\bar s_{j, k-1} \bar h_{k-1} \bar s_{k-1, j}$, where $h_{k-1}$ and $\bar h_{k-1}$ are given in \eqref{hhk}.  Then
 \begin{enumerate} \item [(1)]  $ z_{j, k}=\sum_{\ell=1}^{k-1} e_{j, \ell}-\sum_{1\le s\le k-1, s\neq j}  (s, j)+c_j \left( \sum_{\ell=1}^{k-1} e_{j, \ell}-\sum_{1\le s\le k-1, s\neq j}  (s, j)\right) c_j$,
 \item [(2)] $\bar z_{j, k}=\sum_{\ell=1}^{k-1} e_{ \ell, j}-\sum_{1\le  s\le k-1,  s\neq j}  (\bar s, \bar j)-
 \bar c_j \left(\sum_{\ell=1}^{k-1} e_{ \ell, j}-\sum_{1\le  s\le k-1,  s\neq j}  (\bar s, \bar j)\right) \bar c_j$.
\end{enumerate} \end{Lemma}
\begin{proof} Easy exercise. \end{proof}
Note that $\omega_{0,k}=0$ and $\omega_{1, k}=\bar \omega_{1, k}=0$ (see Lemma~\ref{usef}(13)), and $e_k h=0$ for $h\in BC_{k-1, k-1}$ if and only if $h=0$. We will use these facts  frequently in the proof of
the following lemma, where we use the terminology that
a monomial in  $z_{j, k+1}$'s and $\bar z_{j, k+1}$'s
is a {\it leading term} in an expression if it has
the highest degree by defining ${\rm deg\,}z_{i,j}={\rm deg\,}\bar z_{i,j}=1$.

\begin{Lemma}\label{omed} For any positive integer  $n$,   $\omega_{2n+1, k+1}$ can be written as an $R$-linear combination of
monomials in  $ z_{j, k+1}$'s and $\bar z_{j, k+1}$'s for $1\le j\le k$ such that the leading terms
of  $\omega_{2n+1, k+1}$ are $2\sum_{j=1}^{k}(-z_{j, k+1}^{2n}+\bar z_{ j, k+1}^{2n})$.
\end{Lemma}

\begin{proof}  By  Corollary \ref{equal123} and Lemma~\ref{usef}(8),  we have
\begin{equation}\label{ommmm}
\omega_{2n+1,k+1}e_{k+1}=e_{k+1} y_{k+1}^{2n+1} e_{k+1}=e_{k+1} (\bar L_{k+1}- L_{k+1}) y_{k+1}^{2n}e_{k+1}.\end{equation}
Note that $(j,k+1)=s_{j,k}s_ks_{k,j}$ and  $s_{j,k}, s_{k,j}$ commute with $y_{k+1},e_{k+1}, c_{k+1}$  (see Lemma~\ref{usef}(1) and \eqref{sera}).
Considering the right-hand side of \eqref{ommmm} and expressing $L_{k+1}$ by \eqref{mmpy},  we see that a term  of $-e_{k+1}  L_{k+1} y_{k+1}^{2n}e_{k+1}$
 becomes
$$\begin{aligned}{\sc\!}&-s_{j, k} e_{k+1} (s_{k} y_{k+1}^{2n}\!+
c_{k+1} s_{k} c_{k+1}y_{k+1}^{2n}) e_{k+1} s_{ k, j}=-2s_{j, k} e_{k+1} (s_k y_{k+1}^{2n}) e_{k+1}s_{k, j}\\
& =-2s_{j, k} e_{k+1}\left \{h_k^{2n} s_k -\mbox{$\sum\limits_{b=0}^{2n-1}$} h_k ^{2n-b-1} y_{k+1}^b
+ \mbox{$\sum\limits_{b=0}^{2n-1}$}(-1)^{2n-b} c_k c_{k+1} h_k^{2n-b-1}  y_{k+1}^b\right\} e_{k+1} s_{k, j} \\
&=-2\!s_{j, k} h_k^{2n}e_{k+1}s_ke_{k+1}s_{k, j}+2s_{j, k}\mbox{$\sum\limits_{b=0}^{2n-1}$} h_k^{2n-b-1} e_{k+1} y_{k+1}^b e_{k+1} s_{k, j} \ \ (\text{by Lemma~\ref{usef}(11)})
\\ &=-2\!s_{j, k}e_{k+1}\Big( h_k^{2n}-\mbox{$\sum\limits_{b=0}^{2n-1}$} h_k^{2n-b-1}  \omega_{b, k+1}\Big)s_{k, j}.\\  \end{aligned}$$
 By inductive assumption, the right-hand side of  the above equation can be written as an $R$-linear combination of monomials with the required form such that the leading term is $-2z_{j, k+1}^{2n}$. Finally,  we consider terms in \eqref{ommmm} concerning $\bar L_{k+1}$, namely we need to deal with
 $$e_{k+1} (\bar j, \overline{ k\!+\!1}) y_{k+1}^{2n}  e_{k+1}- e_{k+1} \bar c_{k+1} (\bar j, \overline{ k\!+\!1})\bar c_{k+1} y_{k+1}^{2n}  e_{k+1}=2 e_{k+1}(\bar j, \overline{ k\!+\!1}) y_{k+1}^{2n}  e_{k+1} .$$
 Applying $\tau$ on $ e_{k+1}\bar y_{k+1}^{2n}$ and using Lemma~\ref{y2} and inductive assumption on $n$, we can use
   $ \bar y_{k+1}^{2n} e_{k+1}$ to replace
$y_{k+1}^{2n} e_{k+1}$ in $e_{k+1}  (\bar j, \overline{k\!+\!1})  y_{k+1}^{2n} e_{k+1}$ (by forgetting lower terms). This enables us to consider  $ e_{k+1} (\overline j,
 \overline {k\!+\!1})\bar y_{k+1}^{2n}  e_{k+1}$ instead. As above, this term
can be written as the required form with leading term $ 2\bar z_{j, k+1}^{2n}$. The proof is completed.
\end{proof}

\begin{Lemma}\label{commomega}For~$a\!\in\!\Z^{\ge0},{\ssc\,}k\!\in\!\Z^{\ge 1}$, both $\omega_{a, k+1}$ and
$\bar \omega_{a, k+1}$ commute with $y_{k+1}$, $\bar y_{k+1}, c_{l}, \bar c_{l}$, $l\ge k+1$.\end{Lemma}
 \begin{proof} Since $\omega_{a, 1}=\bar \omega_{a, 1}=\omega_{1, k}=\bar \omega_{1, k}=\omega_{2a, k}=\bar\omega_{2a, k}=0$ for all admissible $a, k$, we can assume $a, k\in \mathbb Z^{\ge 2}$ and $2\nmid a$. In order to verify that $\omega_{a, k+1}$ and
$\bar \omega_{a, k+1}$ commute with $y_{k+1}$,  by Lemmas~\ref{free1},~\ref{dzjk}--\ref{omed},  it suffices to prove that   $y_{k+1}$ commutes with both  $z_{j, k+1}$ and $\bar z_{j, k+1}$ for $1\le j\le k$. By Lemma~\ref{dzjk},    $z_{j, k+1}=\mathfrak z_{j, k+1}+c_j \mathfrak z_{j, k+1} c_j$, where
 $$\mathfrak z_{j, k+1}=\mbox{$\sum\limits_{\ell=1}^k$} e_{j, \ell}-\mbox{$\sum\limits_{1\le s\le k, s\neq j}$}  (s, j),$$
 which  is $z_{j, k+1} $  in \cite[Lemma~3.9]{RSu}.
  Obviously,  $c_{k+1} \mathfrak z_{j, k+1}=\mathfrak z_{j, k+1} c_{k+1}$.
    By Lemma~\ref{usef}(4), we have  $y_{k+1}c_j=c_j y_{k+1}$.
   So,
  \begin{equation}\label{zy123}  \begin{aligned}  z_{j,k+1}y_{k+1}&
  =\mathfrak z_{j, k+1} (c_{k+1} \mathfrak y_{k+1}c_{k+1}+\mathfrak y_{k+1})+c_j \mathfrak z_{j, k+1}(c_{k+1} \mathfrak y_{k+1}c_{k+1}+\mathfrak y_{k+1})c_j\\
  &=c_{k+1} \mathfrak  z_{j, k+1}\mathfrak y_{k+1} c_{k+1}+\mathfrak z_{j, k+1}\mathfrak y_{k+1}+c_j  \mathfrak z_{j, k+1} \mathfrak y_{k+1} c_j+c_jc_{k+1} \mathfrak z_{j, k+1} \mathfrak y_{k+1}c_{k+1}c_j. \\
  \end{aligned}\end{equation}
  Recall that $\tau$ is the $R$-linear anti-involution in Lemma~\ref{tau}.
   By Definition~\ref{RSu-GRSS} and Lemma~\ref{dzjk}, $\tau$ fixes both   $z_{j, k+1}$ and $y_{k+1}$.
  So, $y_{k+1} z_{j,k+1}=\tau(z_{j,k+1}y_{k+1})$. Since $\mathfrak y_{k+1} \mathfrak z_{j, k+1} =\mathfrak z_{j, k+1} \mathfrak y_{k+1} $ (see \cite[Lemma~3.11]{RSu}), we have    $z_{j,k+1}y_{k+1}=y_{k+1} z_{j,k+1}$ by \eqref{zy123}. One can check $\bar z_{j, k+1} y_{k+1}=y_{k+1} \bar z_{j, k+1}$ similarly via Definition~\ref{RSu-GRSS} and the equation
   $\bar z_{j,k+1}=\bar {\mathfrak z}_{j, k+1}-\bar c_j \bar {\mathfrak z}_{j, k+1} \bar c_j$.
   This proves that $y_{k+1}$ commutes with $\omega_{a, k+1}$ and
$\bar \omega_{a, k+1}$. We remark that one can check both $\omega_{a, k+1}$ and
$\bar \omega_{a, k+1}$ commute with  $\bar y_{k+1}$, similarly.
    By Lemma~\ref{dzjk}, one can easily check  that $z_{j, k+1}$ and $\bar z_{j, k+1}$ commute with $c_l$ and $\bar c_l$ for all $l\ge k+1$. \end{proof}

\section{Affine walled Brauer-Clifford superalgebras}
In this section, we assume that  $R$ is  an integral  domain  containing $\omega_1$ and $2^{-1}$. Motivated by  Definition~\ref{wsera} and Lemma~\ref{usef}, we introduce the notion of affine walled Brauer-Clifford superalgebra over $R$ as follows.
\begin{Defn}\label{awbsa}
 The {\it affine  walled Brauer-Clifford superalgebra}
$ BC_{r,t}^{\text{\rm aff}}$ is the associative $R$-superalgebra generated by odd elements $c_1, \ldots, c_r$, $\bar c_1, \ldots, \bar c_t$ and  even elements  $e_1, x_1, \bar x_1$, $s_1, \ldots, s_{r-1}$, $\bar s_1, \ldots, \bar s_{t-1}$,  and two families of even central elements  $\omega_{2k+1}, \bar \omega_{k}$, $ k\in\Z^{\ge1}$ subject to \eqref{symm}--\eqref{aseradual}, \eqref{asera}, \eqref{aseradual2} and Definition~\ref{wsera}(1)--(10) together with the following defining relations for all admissible $i$:
\begin{multicols}{2}
\begin{enumerate}
\item [(1)]$e_1(x_1+\bar x_1)=(x_1+\bar x_1)e_1=0$,
 \item [(2)] $e_1s_1x_1s_1=s_1x_1s_1e_1$,
\item [(3)] $x_1(e_1+\bar x_1-\bar e_1 )=(e_1-\bar e_1+\bar x_1)x_1$,
\item [(4)] $e_1\bar s_1\bar x_1\bar s_1=\bar s_1\bar x_1\bar s_1 e_1$,
\item [(5)] $e_1x_1^{2k+1} e_1=\omega_{2k+1} e_1$,  $\forall k\in \mathbb N$,
\item [(6)] $e_1x_1^{2k}e_1=0$,  $\forall k\in \mathbb N$,
 \item  [(7)] $e_1\bar x_1^{k}e_1=\bar \omega_{k}e_1$, $\forall k\in \mathbb Z^{>0}$,
 \item [(8)] $x_1 \bar c_{i}=\bar c_{i} x_1$,
 \item [(9)] $\bar x_1 c_i= c_i \bar x_1$,
 \item [(10)] $ x_1 \bar s_i=\bar s_i x_1$,
 \item [(11)] $\bar x_1 s_i=s_i \bar x_1$.
 \end{enumerate}
\end{multicols}
 \end{Defn}

For the simplification of presentation, we set $\omega_{2k}=0$, $\forall k\in \mathbb N$. The  following  result follows from  Definition~\ref{awbsa}, immediately.
\begin{Lemma}\label{antiaff} There is an $R$-linear anti-involution $\sigma:  BC_{r,t}^{\rm aff}\rightarrow  BC_{r,t}^{\rm aff}$, which fixes all generators of $BC_{r, t}^{\rm aff}$ in Definition~$\ref{awbsa}$.  \end{Lemma}
Lemmas~\ref{assum} and \ref{assum1} can be proven by arguments similar to those for
  Lemmas~\ref{y1} and \ref{y2}.
\begin{Lemma}\label{assum} For any $n\in \mathbb N$, $e_1\bar x_1^{2n+1}=\sum_{j=0}^n a_{2n+1, j} e_1 x_1^{2j+1}$ for  some $a_{2n+1, j}\in BC_{r, t}^{\rm aff}$  such that
\begin{enumerate}\item [(1)]  $a_{2n+1, n}=-1$,
\item [(2)] $a_{2n+1, j}=a_{2n-1, j-1}$ for all $1\le j\le n-1$,
\item [(3)] $a_{2n+1, 0}=\sum_{j=0}^{n-1} a_{2n-1, j}\omega_{2j+1}$.  \end{enumerate}\noindent In particular, $a_{2n+1,j}\in R[ \omega_3, \ldots, \omega_{2n-1}]$, for all $0\le j\le n$. 
\end{Lemma}

\begin{Lemma}\label{assum1} For any positive integer $n$, $e_1\bar x_1^{2n}=\sum_{j=0}^n a_{2n, j} e_1 x_1^{2j}$ for some $a_{2n,j}\in BC_{r,t}^{\rm aff }$ such that
\begin{enumerate}\item[(1)] $a_{2n, n}=1$,
\item [(2)] $a_{2n, j}=a_{2n-1, j-1}$ for all $1\le j\le n-1$,
\item [(3)] $a_{2n, 0}=\sum_{j=0}^{n-1} a_{2n-2, j}\omega_{2j+1}$. \end{enumerate}\noindent In particular, $a_{2n, j}\in R[ \omega_3, \ldots, \omega_{2n-1}]$, for all $0\le j\le n$. 
\end{Lemma}

\begin{Cor}\label{assump1} If  $e_1$ is $R[\omega_3, \omega_5, \ldots, \bar \omega_1, \bar \omega_2,\ldots]$-torsion-free, then    $\bar \omega_{2n+1} =\sum_{i=0}^{n} a_{2n+1, i} \omega_{2i+1}$ and $\bar \omega_{2n}=0$  for all $n\in \mathbb N $. In particular, $\bar \omega_1=-\omega_1$. \end{Cor}
\begin{proof} By  Definition~\ref{awbsa}(1), $(\omega_1+\bar \omega_1)e_1=0$. If   $e_1$ is $R[\omega_3, \omega_5, \ldots, \bar \omega_1, \bar \omega_2,\ldots]$-torsion-free, $\bar \omega_1=- \omega_1$. In general,  by Lemma~\ref{assum}, $e_1\bar x_1^{2n+1}e_1=\sum_{j=0}^n a_{2n+1, j} e_1 x_1^{2j+1}e_1$. So, $\bar \omega_{2n+1}=\sum_{j=0}^n a_{2n+1, j} \omega_{2j+1}$.
Similarly, by Lemma~\ref{assum1} and Definition~\ref{awbsa}(6), $\bar \omega_{2n}=\sum_{j=0}^n a_{2n, j} \omega_{2j}=0$.
\end{proof}

\begin{Assumption}\label{admis-para}From here onwards, we always assume that $\bar \omega_{2n}=0$ and  $\bar \omega_{2n+1}$'s are given in Corollary~\ref{assump1}. Otherwise, we would have $e_1=0$ provided that $R$ is a field, in which case,   $BC_{r, t}^{\rm aff}$ turns out to be  $HC_{r}^{\rm aff}\boxtimes {\bar {HC}_t}^{\rm aff}$, the outer tensor product
of two affine Hecke-Clifford superalgebras! We remark that the tensor product is that for superalgebra in the sense that $$(x\boxtimes y ) (x_1\boxtimes y_1)=(-1)^{[y][x_1]} xx_1\boxtimes  y y_1,$$
for any homogenous elements $x, x_1\in HC_{r}^{\rm aff} $ and $y, y_1\in \overline{HC}_{t}^{\rm aff} $, where $[x]$, called the parity of $x$, is $1$ (resp., $0$)  if $x$ is odd (resp., even).
\end{Assumption}

\begin{Theorem}\label{level-1}For any $k\in \mathbb Z^{>0}$, there is a superalgebra  homomorphism $\Phi_k: BC_{r, t}^{\rm aff} \rightarrow BC_{r+k, t+k}$ sending
$s_i, \bar s_j, e_1, x_1, \bar x_1, c_l, \bar c_m, \omega_a, \bar \omega_a$
to
$s_{k+i}, \bar s_{k+j}, e_{k+1}, y_{k+1}, \bar y_{k+1}, c_{k+l}, \bar c_{k+m}, \omega_{a, k+1}, \bar \omega_{a, k+1}$ respectively  for all admissible  $a, i, j, l, m$'s.
\end{Theorem}
\begin{proof} It is enough to verify the images of generators of $BC_{r, t}^{\rm aff}$ satisfy the defining relations for $BC_{r, t}^{\rm aff}$ in Definition~\ref{awbsa}.
We say $\Phi_k$ satisfies the relation if the images of generators satisfy this relation.

By Lemmas~\ref{dzjk}--\ref{commomega}, the images of $\omega_a$ and $\bar \omega_a$  commute with the images of other generators.
 By Corollary~\ref{wsba-1}, $\Phi_k$
satisfies  \eqref{symm}--\eqref{aseradual}  and Definition~\ref{wsera}(1)--(10).
 $\Phi_k$ satisfies
 \eqref{asera} and \eqref{aseradual2} by  Lemma~\ref{usef}(1), (3), (4), (10). Further, $\Phi_k$ satisfies Definition~\ref{awbsa}(1)--(4) by Lemma~\ref{usef}(7)--(9). In this case, we need $(y_{k+1}+\bar y_{k+1})e_{k+1}=0$, which can be obtained by applying the anti-involution $\tau$ on Lemma~\ref{usef}(8).
$\Phi_k$ satisfies  Definition~\ref{awbsa}(5)--(7) by Corollary~\ref{equal123} and Lemma~\ref{usef}(11)--(13). Finally, $\Phi_k$ satisfies
 Definition~\ref{awbsa}(8)--(11) by Lemma~\ref{usef}(2), (5).
\end{proof}

In \cite{RSu},  two of the authors proved the freeness of the  affine walled Brauer algebra via  bases of infinite many walled Brauer algebras. The key point is the existence of infinite many homomorphisms between the affine walled Brauer algebra and  walled Brauer algebras \cite[Theorem~3.12]{RSu}.  In the current case,
  Theorem~\ref{level-1} is the counterpart of \cite[Theorem~3.12]{RSu}.
  Since $\omega_{1,  k}=0$ for all $k$, what we can do is to use  Theorem~\ref{level-1} to prove the freeness of affine walled Brauer-Clifford superalgebras with parameters $\omega_1=0$. However, many affine walled Brauer-Clifford algebras  which appear in the higher  mixed Schur-Weyl-Sergeev dualities  have non-zero parameter $\omega_1$. For details, see section~4. For this reason, we  use  level two walled Brauer-Clifford superalgebras  (with special parameters) instead of walled Brauer-Clifford superalgebras later on.
 This  is one of the points which is different from the work in \cite{RSu}.

In $BC_{r, t}^{\rm aff}$, we define $x_i, x_i', \bar x_j$ and $\bar x_j'$ as in \eqref{relsmurp} for all admissible $i$ and $j$.
\begin{Lemma}\label{hecrel} We have the  following  results for admissible  $i$ and $j$:
\begin{multicols}{2}
 \begin{enumerate}
\item [(1)] $x_i c_i=-c_i x_i$ and $\bar x_i \bar c_i=-\bar c_i \bar x_i$,  \item [(2)]  $x_i c_j=c_j x_i$, $\bar x_i \bar c_j=\bar c_j \bar x_i$ if $i\neq j$,
\item [(3)] $x_i\bar c_j=\bar c_j x_i$ and $\bar x_i c_j=c_j \bar x_i$.
  \end{enumerate}\end{multicols} \end{Lemma}
\begin{proof} (1) and (2) follow  from \eqref{sera}, \eqref{aseradual},
\eqref{asera} and \eqref{aseradual2}.  (3) follows from  Definition~\ref{awbsa}(8)--(11).  \end{proof}

\begin{Lemma}\label{hecrel2}   We have the following results  for all admissible  $i, j$:
\begin{multicols}{2}
\begin{enumerate}\item [(1)] $x_i'(\bar x_j'+e_{i, j}-\bar e_{i, j})=(\bar x_j'+e_{i, j}-\bar e_{i, j}) x_i'$,
\item [(2)] $\bar x_j'( x_i'+e_{i, j}+\bar e_{i, j})=( x_i'+e_{i, j}+\bar e_{i, j})\bar x_j'$,
\item [(3)] $e_{i, j}(x_i'+\bar x_j')=0$,
\item [(4)] $(x_i'+\bar x_j')e_{i, j}=0$,
\item [(5)]$e_i{x_i'}^kc_ie_i=0$,$\forall k\in \mathbb N$,
\item [(6)]$e_i{\bar x_i'}^kc_ie_i=0$,$\forall k\in \mathbb N$.
 \end{enumerate}\end{multicols}
\end{Lemma}

 \begin{proof}  Multiplying $(1, i) (\bar 1, \bar j)$ on both sides of
Definition~\ref{awbsa}(3)  yields (1).  By Definition~\ref{awbsa}(1),(3), we know  that $\bar x_1(x_1+e_1+\bar e_1)=(x_1+e_1+\bar e_1)\bar x_1$. So (2) can be  proved similarly. Multiplying
 $(1, i)(\bar 1, \bar j)$ on both side of  Definition~\ref{awbsa}(1) yields (3) and (4).
We have $c_i x_i'=-x_i'c_i$ (resp., $\bar c_i x_i'=x_i' \bar c_i$) by \eqref{jm-wba1} and Lemma~\ref{hecrel}(1) (resp., Definition~\ref{awbsa}(8)--(11)). Also, (1) is the counterpart of Lemma~\ref{usef}(7).
So, (5) and (6)  can be verified by arguments similar to those  in the proof of Lemma~\ref{usef}(11). We leave the details to the reader. \end{proof}

 \begin{Lemma}\label{relxpri} We have the following results  for all admissible  $i,j, k, l$:
  \begin{multicols}{2}
 \begin{enumerate} \item [(1)]  $e_{i, k}x_j'=x_j'e_{i, k}$, if $i\neq j$,
 \item [(2)] $e_{i, k}\bar x_l'=\bar x_l' e_{i, k}$, if  $k\neq l$,
 \item [(3)] $e_{i, j}(x_i')^a e_{i, j}=\omega_a e_{i, j}$, $\forall a\in \mathbb  N$,
 \item [(4)] $e_{i, j}(\bar x_j')^a e_{i, j}=\bar \omega_a e_{i, j}, \forall a\in \mathbb  N$.
 \end{enumerate} \end{multicols}
 \end{Lemma}
 \begin{proof} We have $e_1x_2'=x_2'e_1$ by Definition~\ref{awbsa}(2). Multiplying $(2, j)$ on both sides of the equation yields $e_1 x_j'=x_j' e_1$. Since $i\neq j$,
  multiplying  $(1,i)(\bar 1, \bar k)$ on both sides of  $e_1 x_j'=x_j' e_1$ yields (1). (2) can be verified similarly.  (3) and (4) follow from Definition~\ref{awbsa}(5)--(7).
  \end{proof}

We consider $\ BC_{r,t}^{\text{\rm aff}}$ as a filtrated superalgebra  by setting
$$ \text{deg}{ s_i}=\text{deg}{ \bar s_j}=\text{deg}{ e_1}=\text{deg}{ c_n}=\text{deg}{ \bar c_m}=\text{deg}{  \omega_a}=\text{deg} {\bar \omega_a}=0 \text{ and } \
\text{deg} {x_k}=\text{deg}{ \bar x_\ell}=1,$$
for all admissible  $a, i, j, k, \ell, m, n$. Let $(BC_{r,t}^{\rm aff})^{(k)}$ be the super  $R$-submodule
spanned by monomials with degrees  less than or equal to $k$ for  $k\in{\mathbb{Z}}^{\ge0}$. Then we have the following filtration
   \begin{equation}\label{filtr}
BC_{r,t}^{\text{\rm aff}}\supset\ldots\supset (BC_{r,t}^{\text{\rm aff}})^{(1)}\supset(BC_{r,t}^{\text{\rm aff}})^{(0)}\supset (BC_{r,t}^{\text{\rm aff}})^{(-1)}=0.\end{equation}
Let ${\rm gr} ( BC_{r,t}^{\text{\rm aff}})\!=\!\oplus_{i\ge0}( BC_{r,t}^{\text{\rm aff}})^{[i]}$, where
$( BC_{r,t}^{\text{\rm aff}})^{[i]}\!=\!( BC_{r,t}^{\text{\rm aff}})^{(i)}/( BC_{r,t}^{\text{\rm aff}})^{(i-1)}$. Then
  ${\rm gr} ( BC_{r,t}^{\text{\rm aff}})$ is a $\mathbb Z$-graded superalgebra associated to $BC_{r,t}^{\text{\rm aff}}$.~We use the same symbols to denote elements in
${\rm gr} ( BC_{r,t}^{\text{\rm aff}})$. In particular, $x_i=x_i'$ and $\bar x_j'=\bar x_j$ in ${\rm gr} ( BC_{r,t}^{\text{\rm aff}})$.

\begin{Defn}\label{regm}  We say that $\textbf{m}$ is a regular monomial of  $BC_{r,t}^{\rm aff}$ if
 $\textbf{m}= x^{\alpha} d   {\bar x}^{\beta} \prod_{n\in \mathbb Z^{>0}} \omega_{2n+1}^{a_{2n+1}}$,  for some $d\in S$, $a_{2n+1}\in \mathbb N$ and $ (\alpha, \beta)\in \mathbb N^r\times \mathbb N^t$, where $S$ is given  in \eqref{bcabasis}, $x^\alpha=\prod_{i=1}^r x_i^{\alpha_i}$ and $\bar x^\beta=\prod_{i=1}^t \bar x_i^{\beta_i}$. \end{Defn}

\begin{Prop}\label{spanned}   As an $R$-module, $BC_{r,t}^{\rm aff}$ is spanned by all regular monomials in Definition~$\ref{regm}$.  \end{Prop}
\begin{proof} Let $M$ be the $R$-submodule of $ BC_{r,t}^{\rm aff}$ spanned by all regular monomials $\textit{\textbf{m}} $ in Definition~\ref{regm}. We want to prove
   \begin{equation} \label{veri}h\textit{\textbf{m}} \in M\mbox{ \ \ for any generator $h$ of $ BC_{r,t}^{\rm aff}$}.\end{equation}
   If so,  we have $M=BC_{r,t}^{\rm aff}$ since
 $1\in M$.

We  prove (\ref{veri})  by induction on  $|\alpha|$. If  $|\alpha|=0$, i.e.,  $\alpha_i=0$ for all possible $i$'s,
then  (\ref{veri}) follows from  Theorem~\ref{wbhsa321}  unless  $h= \bar x_1$. In the later case,
by \eqref{asera},\eqref{aseradual2} and  Lemma~\ref{hecrel},  we need to compute $\bar x_k e^f $ when $1\le k\le t$ and $f>0$.
If $k\in \{1,2, \ldots, f\}$, by Lemma~\ref{hecrel2}(4),  we use $-x_k$ instead of $\bar x_k$ since we work on the graded superalgebra ${\rm gr} ( BC_{r,t}^{\text{\rm aff}})$.
So, $h{\textit{\textbf{m}}}\in M$.   Otherwise, $k>f$.  By Lemma~\ref{relxpri}(2), we can use  $e^f \bar x_k $
 instead of $\bar x_k e^f $.  So, (\ref{veri}) follows from Lemma~\ref{hecrel} and Theorem~\ref{wbhsa321}.

Suppose $|\alpha|>0$. By \eqref{asera},\eqref{aseradual2},  Lemma~\ref{hecrel} and Theorem~\ref{wbhsa321}, we see that (\ref{veri}) holds
 unless $h\in \{\bar x_1, e_1\}$.
  Suppose  $h=\bar x_1$. By Lemma~\ref{hecrel2}(1) and (2), $\bar x_1 x_i=x_i\bar x_1$ in ${\rm gr} ( BC_{r,t}^{\text{\rm aff}})$. So, we need to deal with  $\bar x_k e^f $ when   $1\le k\le t$. They are  the cases that we have dealt with. So, $\bar x_1 \textbf{m}\in M$.

   Finally, we assume  $h=e_1$. If   $\alpha_i\neq 0$ for some  $i$ with $ 2\le i\le r$, then $e_1 x_i=x_i e_1$ in ${\rm gr} ( BC_{r,t}^{\text{\rm aff}})$ (see Lemma ~\ref{relxpri}(1)). By inductive assumption on $|\alpha|$, we have (\ref{veri}). In order to finish the proof, it remains to consider the case that
     $x^{\alpha}=x_1^{\alpha_1}$ such that   $\alpha_1>0$. In this case,
     \begin{equation}\label{veri1}\textbf{m} = x_1^{\alpha_1} {c}^{\gamma } d_1^{-1} e^f w d_2  {\bar c}^{\delta} \bar x^{\beta}\in M,\end{equation}
     where $d_1, d_2\in \mathscr D_{r, t}^f$ and $\beta\in \mathbb N^t$ and $(\gamma, \delta)\in \Z_2^r \times \Z_2^t $.
     Write  $d_1 e_1 d_1^{-1}=e_{i, j}$ for some $i,j$. By \eqref{asera} and inductive assumption on $|\alpha|$, we can use  $d_1^{-1} x_i^{\alpha_1}$
 to replace  $x_1^{\alpha_1} d_1^{-1}$ in \eqref{veri1}. So, we need to verify
  \begin{equation} \label{v1} e_{i, j} x_i^{\alpha_1} c^{\gamma} e^f w d_2 {\bar c}^{\delta }  \bar x^{\beta}\in M.\end{equation}
  By Lemma~\ref{hecrel2}(3) and inductive assumption, it is enough to  verify
   \begin{equation}\label{v2} e_{i, j} \bar x_j^{\alpha_1}c^{\gamma} e^f w d_2 {\bar c}^{\delta }  \bar x^{\beta}\in M.\end{equation}
  If $j\ge f+1$, then  \eqref{v2} follows from Lemma~\ref{relxpri}(2) and Theorem~\ref{wbhsa321}. Otherwise,  $j\le  f$.
If $i=j$, by inductive assumption, we use $( \bar x_i+\bar L_i)^{\alpha_1}$ instead of  $\bar x_i^{\alpha_1} $ in  $e_{i} \bar x_i^{\alpha_1} c^\gamma  e_i$.  If $\gamma_i=0$, then  $e_{i} \bar x_i^{\alpha_1}   e_i=\bar \omega_{\alpha_1} e_1$ in ${\rm gr} ( BC_{r,t}^{\text{\rm aff}})$ (see Lemma~\ref{relxpri}(4)). If $\gamma_i\neq 0$, then $e_{i} \bar x_i^{\alpha_1}c_i   e_i=0$ in ${\rm gr} ( BC_{r,t}^{\text{\rm aff}})$ (see Lemma~\ref{hecrel2}(5)). In any case,  \eqref{v2} follows from inductive assumption  on $|\alpha|$.
Finally, we assume  $i\neq j$.  If $i\neq k$, then  $e_{i,j}c_k=c_ke_{i,j}$. By inductive assumption, we  need to consider
  $e_{i,j} x_i^{\alpha_1} c_i^{\gamma_i} e_j=e_{i,j} x_i^{\alpha_1}  e_j c_i^{\gamma_i}$. Since    $$e_{i,j} x_i^{\alpha_1} e_j= e_{i,j} e_j x_i^{\alpha_1} = (i, j)x_i^{\alpha_1}e_j=x_j^{\alpha_1} (i, j)e_j$$  in
   ${\rm gr}(BC_{r,t}^{\text{\rm aff}})$,  by inductive assumption and our previous results on  $h\in$ $\{s_1, \ldots, s_{r-1}, c_1,\ldots, c_r, x_1\}$, we have \eqref{v2}. So  \eqref{v1} is true.  This completes the proof.
   \end{proof}

\begin{Defn}\label{affinewbcsa} Let $I$ be the two-sided ideal of $BC_{r, t}^{\rm aff}$ generated by $\omega_{2k+1}-{\tilde \omega}_{2k+1}$, where
$\tilde {\omega}_{2k+1}\in R$ for all $k\in \mathbb Z^{>0}$.  Let $\widetilde {BC}_{r, t}=BC_{r, t}^{\rm aff}/I$. \end{Defn}

\begin{Defn}\label{cwbcsa1} Let $I$ be the two-sided ideal of $\widetilde{BC}_{r, t}$ generated by $f(x_1)$ and $g(\bar x_1)$, where
 \begin{equation}\label{fg} f(x_1)= x_1^k\mbox{$\prod\limits  _{i=1}^m$} (x_1^2-u_i^2), \text{ and  }g(\bar x_1)= \bar x_1^{k_1}\mbox{$\prod\limits _{j=1}^{m_1}$}(\bar x_1^2-\bar u_j^2),\end{equation}  for some non-zero  $u_1, \ldots, u_m, \bar u_1, \ldots,\bar u_{m_1}\in R$ such that $\ell=k+2m=k_1+2m_1$ and \begin{equation} \label{efg3211} e_1 f(x_1)=(-1)^k e_1 g(\bar x_1).\end{equation}  The {\it level $\ell$  or cyclotomic walled Brauer-Clifford superalgebra $BC_{\ell, r, t}$ } is the quotient algebra
 $\widetilde{BC}_{r, t}/I$.\end{Defn}
In section~6, we will explain the reason why $f(x_1)$ and $g(\bar x_1)$ have to satisfy   \eqref{fg} and \eqref{efg3211}.

\begin{Defn}\label{cregm}  We say that $\textbf {m}$ is a regular monomial of $\widetilde{BC}_{r, t}$ (resp.,
$ BC_{\ell, r, t}$)  if it is of form  $ x^{\alpha} d   {\bar x}^{\beta} $,  for some $d\in S$, and $ (\alpha, \beta)\in  \mathbb N^r\times \mathbb N^t$ (resp.
$\mathbb Z_\ell^r\times \mathbb Z_\ell^t$), where $S$ is given   in \eqref{bcabasis}.\end{Defn}

\begin{Cor}\label{level-l-span}
As  $R$-modules, both $\widetilde{BC}_{r, t}$ and  $BC_{\ell, r,t}$ are  spanned by their  regular monomials.  \end{Cor}
\begin{proof} By Proposition~\ref{spanned}, $\widetilde{BC}_{r, t}$ is spanned by all its regular monomials.
Let $\phi_\ell: \widetilde{BC}_{r, t}\twoheadrightarrow BC_{\ell, r,t}$ be the canonical epimorphism.
It is enough to verify the image of a regular monomial $\textbf{m}$ of $\widetilde{BC}_{r, t}$   can be expressed as a  linear combination of
regular monomials of $BC_{\ell, r,t}$. If both $\alpha\in  \Z_\ell^r$ and $\beta\in \Z_\ell^t$,   then the images of $\textbf{m}$ is a regular monomial of $ BC_{\ell, r,t}$. Otherwise, either $\alpha_i\ge \ell$ or  $\beta_j\ge \ell$ for some possible $i$ or $j$. Since $\widetilde {BC}_{r, t}$ inherits the graded structure of   $BC_{r, t}^{\rm aff}$,  it results in a graded structure on $BC_{\ell, r, t}$.
 So,  either $x_i^{\alpha_i}$ or $\bar x_j^{\beta_j}$ can be expressed as a linear combination of elements in $BC_{\ell, r, t}$ with lower degrees. Using these elements to replace either $x_i^{\alpha_i}$ or $\bar x_j^{\beta_j}$ in the image of $\textbf{m}$ and considering the inverse images of such elements
in $\widetilde{BC}_{r, t}$, we see that the image of  $\textbf{m}$ can be expressed as a linear combination of regular monomials of $BC_{\ell, r, t}$, as required.
\end{proof}

\section{A basis of $BC_{2, r, t}$ with special parameters}
Let $\mfg=\mathfrak{q}(n)$ be the queer Lie superalgebra of rank $n$ over $\C$, which has a basis $e_{ij}=E_{i,j}+E_{-i,-j}$ (even element), $f_{i,j}=E_{i,-j}+E_{-i,j}$ (odd element) for
$i,j\in I^+=\{1,2,...,n\},$ where $E_{i,j}$ is the $2n\times2n$ matrix with entry $1$ at $(i,j)$ position and zero otherwise for $i,j\in I=I^+\cup I^-$, and  $I^-=-I^+$.
Let  $V=\C^{n|n}=V_{\bar0}\oplus V_{\bar1}$ be the
natural $\mfg$-module (and the natural ${\mathfrak{gl}}_{n|n}$-module) with basis $\{v_i\,|\,i\in I\}$. 
Then $v_i$ has the parity $[v_i]=[i]\in\Z_2$, where $[i]=0$ and $[-i]= 1$ for $i\in I^+$.
Let $V^*$ be the linear dual space of $V$ with  dual basis $\{\bar v_i\,|\,i\in I\}$. Thus $V^*$ is a left $\mfg$-module with action
\begin{eqnarray}\label{action-dual}
E_{a,b}\bar v_i=-(-1)^{[a]([a]+[b])}\d_{i,a}\bar v_b\mbox{ for }a,b,i\in I.\end{eqnarray}
Let $\fh=\fh_{\bar 0}\oplus\fh_{\bar 1}$ be a Cartan subalgebra of $\mfg$ with even part $\fh_{\bar0}={\rm span}\{e_{i,i}\,\,|\,i\in I^+\}$ and
odd part $\fh_{\bar 1}={\rm span}\{f_{i,i}\,\,|\,i\in I^+\}$. Let $\fh^*_{\bar0}$ be the dual space of $\fh_{\bar0}$ with
 $\{\es_i\,|\,i\in I^+\}$ being the dual basis of $\{e_{i,i}\,\,|\,i\in I^+\}$. Then an element $\l\in\fh^*$ (called a {\it weight}) can be written as \begin{equation}\label{weight-}\l=\SUM{i\in I^+}{}\l_i\es_i=(\l_1,...,\l_n)\mbox{ \ with \ }\l_i\in\C.\end{equation}
Let $M$ be any $\mfg$-module. For any $r,t\!\in\!\Z^{\ge0}$, set $M^{r,t}=M\OTIMES V^{\otimes r}\OTIMES  (V^*)^{\otimes t}$. For convenience
we denote the ordered set \begin{equation}\label{ordered-set}J=\{0\}\cup J_1\cup J_2,\
\mbox{ where $J_1=\{1,...,r\}$, $J_2=\{\bar 1,...,\bar t\}$,} \end{equation}
 and
$0\prec1\prec\ldots\prec r\prec\bar1\prec\ldots\prec\bar t$.
We write $M^{r,t}$ as
\begin{equation}\label{M-st==}M^{r,t}=\OT{i\in J}V_i,\mbox{ \ where $V_0=M$, $V_i=V$ if $0\prec i\prec\bar1$, and $V_i=V^*$ if $i\succ r$,} \end{equation}
 (hereafter all tensor products will be taken according to the order in $J$), which is a left $U(\mfg)^{\otimes(r+t+1)}$-module (where $U(\mfg)$ is the universal enveloping algebra of $\mfg$), with the action given by
 $$\Big(\OT{i\in J} g_i\Big)\Big(\OT{i\in J} x_i\Big)=(-1)^{\sum\limits_{i\in J}{}[g_i]\sum\limits_{j\prec i}{}[x_j]}\OT{i\in J}(g_ix_i)\mbox{ for }g_i\in U(\mfg),\ x_i\in V_i.$$

For the purpose of proving a basis of level $2$ walled-Brauer Clifford superalgebra, we take
$n=2m$ to be even integer. We denote $I_1^+=\{1,...,m\},$ $I_2^+=m+I_1^+$. Thus $I^+=I_1^+\cup I_2^+$.
For $i\in I_1^+$, we denote $i_{\bullet}=i+m\in I_2^+$. For $i\in I_2^+$, we denote $i_{\circ}=i-m\in I_1^+$.
Let $M=L_\l$  be the finite dimensional simple $\mfg$-module of type $M$ with dominant highest weight \equa{highw}{\mbox{$\l=(p,p\!-\!1,...,p\!-\!n\!+\!1)$
 for some $ p\!\in\!\C$ such that $p\!\in\!\C\backslash\Z$ or $p\!\in\!\Z$ with $p\!>\!n$ or $p\!<\!0$.}}
 Then  $End_g(L_\lambda)$ is one dimensional.
Denote by $v_\l$ a fixed highest weight vector of $L_\l$ with even parity, and $(L_\l)_\l$ the highest weight space of $L_\l$, which is $2^m$-dimensional with a basis
\equa{-basis-u}{B_1=\{b^\th v_\l\,|\,b^\th\in B_0\}\mbox{ with }
B_0=\Big\{b^\th:=\mbox{$\prod\limits_{i\in I_1^+}$} f_{i,i}^{\th_i}\,\Big|\,\th=(\th_1,...,\th_m)\in\Z_2^m\Big\},}
where the products are taken in any fixed order $($changing the order only changes the vectors by a factor $\pm1)$.
%
%
%
%
%
%%%%%%%%%%%%%%%%%%%%
For $i\in I_2^+$, we have
\equa{Action-ffi0}{f_{i,i}v_\l=\sqrt{-\frac{p+i-1}{p+i_\circ-1}}f_{i_\circ,i_\circ}v_\l.}
Let $C$ be the PBW monomial basis of $U(\mfg^-\oplus\fh)$.
We say a basis element $a\in C$ has {\it length} $\ell(a):=k$ if $a$ contains $k$ factors; for instance, $\ell(b^{\th})=|\th|$.
For $i\in\Z^{\ge0}$, let $C_i=\{a\in C\,|\,\ell(a)=i\}.$ Set \equa{basis-u}{D=
\Big\{u^{\si}:=\mbox{$\prod\limits_{i\in I_1^+}$}f_{i_{\bullet},i}^{\si_i}\,\Big|
\,\si\!=\!(\si_1,...,\si_m)\!\in\!\Z_2^m\Big\}\subset C, \mbox{ and }D_i=D\cap C_i.}
Let $(L_\l)_i$ be the subspace of $L_\l$ spanned by $cv_\l$ for $c\in C$ with $\ell(c)\le i$. Set $(L_\l)_{-1}=0$. Note that elements of $\mfg^+$ acting on $L_\l$ send $(L_\l)_i$ to $(L_\l)_{i-1}$.

\begin{Lemma}\label{basi1lemm}
For $i\in\Z^{\ge0}$, the set $D_iv_\l$ is $\C$-linear independent under modulo $(L_\l)_{i-1}$.
\end{Lemma}
\begin{proof}
Assume $c:=\sum_{\si\in\Z_2^m:|\si|=i}a_{\si} u^{\si} v_\l\in (L_\l)_{i-1}$ for some $a_{\si}\in\C$ with at least one $a_{\si}\ne0$.
Take a $\tilde\si\in\Z_2^m$ such that $a_{\tilde\si}\ne0$. Assume $\tilde\si_\ell\ne0$ for some $\ell\in I_1^+$.
Applying $f_{\ell,\ell_{\bullet}}\in\mfg^+$ to $c$, by moving $f_{\ell,\ell_{\bullet}}$ to the right until it meets $v_\l$, using the commutation relation $[f_{\ell,\ell_{\bullet}},f_{j_{\bullet},j}]=\d_{\ell{\ssc\,}j}(e_{\ell\ell}+e_{\ell_{\bullet},\ell_{\bullet}})$ (which is a Cartan element commuting with $f_{i_{\bullet},i}$ for $i\ne\ell$), we can easily obtain
$f_{\ell,\ell_{\bullet}}c=\sum_{\si\in\Z_2^m:\si_{\ell}\ne0}a'_{\si}(2p-(\ell+\ell_{\bullet})) u^{\si-1_\ell}v_\l\in (L_\l)_{i-2}$, where $a'_{\si}=\pm a_{\si}$, $1_\ell=(\d_{1\ell},...,\d_{m\ell})\in\Z_2^m$. Note that $2p-(\ell+\ell_{\bullet})\ne0$ by \eqref{highw}. Now induction on $|\th|$ gives that $a_{\tilde\si}=0$, a contradiction with the assumption.
\end{proof}

For each $i=1,2,...,$ by Lemma \ref{basi1lemm}, we can choose a maximal subset $\hat C_i$ of $C_i$ satisfying the following conditions
(i.e., we extend $D_i$ to a basis $\hat C_i$ of $(L_\l)_i$ modulo $(L_\l)_{i-1}$, thus  $\#\hat C_i={\rm dim}(L_\l)_i/(L_\l)_{i-1}$):
\begin{itemize}
\item[(C1)] $\hat C_i\supset D_i$;
\item[(C2)]
$\{u v_\l\,|\,u\in\hat C_i\}$ is a  $\C$-linear independent subset of $L_\l$.
\end{itemize}
Then we have the following
 basis of $L_\l$, \equa{BWE}{\mbox{$B^e:=\{w v_\l \,|\,w\in \hat C\}$, where  $\hat C=\bigcup\limits_{i=0}^\infty\hat C_i$}.} We say the basis element $wv_\l$ has {\it length} $\ell(wv_\l):=\ell(w)$.
Then from our choice of $\hat C_i$, we immediately have the following.
\begin{Lemma}\label{Basis20}
Let $\alpha\in C$ be a monomial basis element of length $j$. Then $\alpha v_\l$ is a combination of basis elements in $B^e$ with length $\le j$.
\end{Lemma}

Take the following 
 basis of $M^{r,t}$,
\begin{equation}\label{basis-M-0}
B_M=\Big\{b_M= b\otimes\OT{i\in J_1}v_{k_i}\otimes\OT{i\in J_2}\bar v_{k_i}\,\Big|\,b\in B^e,\,
k_i\in I\Big\}.
\end{equation}
Introduce the following elements, \begin{eqnarray}\label{def-Omega}&\!\!\!\!\!\!\!\!\!\!\!\!\!\!\!\!\!\!\!\!\!\!\!\!\!\!\!&
\bar e_{ij}=E_{ij}-E_{-i,-j},\ \ \  \bar f_{ij}=E_{-i,j}-E_{i,-j}\ \ \in\ \ {\mathfrak {gl}}_{n|n},\nonumber\\
&\!\!\!\!\!\!\!\!\!\!\!\!\!\!\!\!\!\!\!\!\!\!\!\!\!\!\!\!\!\!\!&\Omega_0\!=\!\mbox{$\sum\limits_{i,j\in I}$}(-1)^{[j]}E_{ij}\OTIMES E_{ji}\in {\mathfrak {gl}}_{n|n}^{\otimes 2},\ \ \
\Omega_1\!=\!\mbox{$\sum\limits_{i,j\in I^+}$}e_{ij}\OTIMES \bar e_{ji}\!-\!\mbox{$\sum\limits_{i,j\in I^+}$}f_{ij}\OTIMES \bar f_{ji}\in\mfg\OTIMES {\mathfrak {gl}}_{n|n}.\end{eqnarray}
For $a,b\in J$ with $a\prec b$, we define
$\pi_{ab}:U(\mfg)\otimes U({\mathfrak {gl}}_{n|n})\to U(\mfg)\otimes U({\mathfrak {gl}}_{n|n})^{\otimes(r+t)}$ by \begin{equation}\label{pi-ab}
\pi_{ab}(x\OTIMES y)=1\OTIMES\cdots\OTIMES1\OTIMES x\OTIMES 1\OTIMES\cdots\OTIMES1\OTIMES y\OTIMES1\OTIMES\cdots\OTIMES1,\end{equation}
where $x$ and $y$ are in the $a$-th and $b$-th tensors respectively. Similarly we have $\pi_a:U(\mfg)\to U(\mfg)^{\otimes(r+t+1)}$
which sends  $x$ to the $a$-th tensor.
\begin{Defn}\label{casm}
We can use \eqref{pi-ab} to define the following elements of the endomorphism algebra ${\rm End}_{U(\mfg)}(M^{r,t})^{\rm op}$,
\begin{eqnarray}\label{operator--1}&\!\!\!\!\!\!\!\!\!\!\!\!\!\!\!\!\!\!\!\!\!\!\!\!\!\!&
s_i=\pi_{i,i+1}(\Omega_0)|_{M^{r,t}}\ (1\!\le\! i\!<\!r),\ \ \ \
\bar s_j=\pi_{\bar j,\overline{j+1}}(\Omega_0)|_{M^{r,t}}\ (1\!\le\! j\!<\!t),\nonumber\\&\!\!\!\!\!\!\!\!\!\!\!\!\!\!\!\!\!\!\!\!\!\!\!\!\!\!\!\!\!\!\!\!\!&
x'_i=-\pi_{0i}(\Omega_1)|_{M^{r,t}}\ (1\!\le\! i\!\le\! r),\ \ \ \ \bar x'_j=-\pi_{0\bar j}(\Omega_1)|_{M^{r,t}}\ (1\!\le\! j\!\le\! t),
  \nonumber\\&\!\!\!\!\!\!\!\!\!\!\!\!\!\!\!\!\!\!\!\!\!\!\!\!\!\!\!\!\!\!\!\!\!& e_i\!=\!-\pi_{i\bar i}(\Omega_0)|_{M^{r,t}}\,(1\!\le\! i\!\le\!\min\{r,t\}),\ \
   c_i\!=\!\pi_i(c)\, (1\!\le\! i\!\le\! r),\ \ \bar c_i\!=\!\pi_{\bar i}(\bar c)\,(1\!\le\! i\!\le\! t),
\end{eqnarray}
where $c:V\to V$ (resp., $\bar c:V^*\to V^*$) is the automorphism such that $c(v_{\pm i})=\pm v_{\mp i}$ (resp., $\bar c(\bar v_{\pm i})=\bar v_{\mp i}$).  Set $x_1=x'_1,\,\bar x_1=\bar x'_1$.
\end{Defn}

Observe that $c^2=-1$ and $\bar c^2=1$, and  $c,\bar c$ correspond to  maps $c,\bar c:I\to I$ such that
\equa{c-map}{c(\pm i)=\bar c(\pm i)=\mp i\mbox{ for }i\in I^+,
\mbox{ \ and $c(v_i)=[i]v_{c(i)}$, \ \ $\bar c(\bar v_i)=\bar v_{\bar c(i)}$ for $i\in I$.}}

\begin{Lemma}\label{SMSM}
\begin{enumerate}\item[\rm(a)]
The minimal polynomial of $x_1$ with respect to  $M^{r,t}$ is $f(x)=x^2-p(p+1)$.
\item[\rm(b)]
The minimal polynomial of $\bar x_1$ with respect to  $M^{r,t}$ is $g(\bar x_1)=x^2-(p-n+1)(p-n)$.
\item[\rm(c)] We have $e_1x_1e_1=-n(2p-n+1)e_1$ with respect to $M^{r,t}$.
\end{enumerate}
\end{Lemma}
\begin{proof}
(a) We may assume $r=1,\,t=0$. Note that the only possible highest weight in the finite dimensional $\mfg$-module $L_\l\otimes V$ is $\mu=\l+\epsilon_1$, which is a typical dominant weight. Thus $L_\l\otimes V$ must be complete reducible, and thus  a direct sum of finite copies of $L_{\mu}$. Observe that
the set $\{u^\th\otimes v_{\pm1}\,|\,u^\th\in B_0\}$, with $2^{m+1}$ elements, is a maximal set of $\C$-linear independent highest weight vectors of weight $\mu$. Since $L_\mu$ occupies $2^m$ $\C$-linear independent highest weight vectors, we  see that $L_\l\otimes V=L_\mu^{\oplus2}$, which  as a $\mfg$-module is generated by
$v^\pm_\mu:=v_\l\otimes v_{\pm1}$. One can easily verify that $v_\mu^\pm x_1=\mp(p+1)v_\mu^\pm\pm f_{11}v_\mu^\mp $. Thus $v_\mu^+,f_{11}v_\mu^{-}$ (resp.,
$v_\mu^-,f_{11}v_\mu^{+}$) span a 2-dimensional $x_1$-invariant subspace of $L_\l\otimes V$, and the minimal polynomial of $x_1$ in this subspace is $f(x)=x_1^2-p(p+1)$. Since $x_1$ commutes with the $\mfg$-action and $L_\l\otimes V$ is generated by $v_\mu^\pm$, we see $f(x_1)$ is also the minimal polynomial of $x_1$ in $M^{r,t}$.

(b) We can assume $r=0,\,t=1.$ Similar to the arguments in (a), we have $L_\l\otimes V^*=L_\nu^{\oplus2}$ with highest weight $\nu=\l-\epsilon_n$ (which is again a typical dominant weight) and two highest weight vectors $v_\nu^\pm:=v_\l\otimes\bar v_{\pm n}$. In addition, $v_\nu^\pm\bar x_1=\pm(p-n)v_\nu^\pm+f_{nn} v_\nu^\mp$. Thus the minimal polynomial of $\bar x_1$ is $g(\bar x_1)=\bar x_1^2-(p-n+1)(p-n)$.

(c) We can assume $r=t=1.$ Then for $a,b\in I$, we have
$$\begin{array}{lll}(v_\l\otimes v_a\otimes \bar v_b)e_1x_1e_1=(-1)^{[a]}\d_{ab}\sum\limits_{i\in I}(v_\l\otimes v_i\otimes\bar v_i)x_1e_1\\[4pt]
\phantom{===}=(-1)^{[a]}\d_{ab}\sum\limits_{i\in I}(-1)^{[i]}\big((p\!+\!1\!-\!i)v_\l\otimes v_i\otimes\bar v_i\big)e_1\!=\!-n(2p\!-\!n\!+\!1)(v_\l\otimes v_a\otimes \bar v_b)e_1.\end{array}$$
Since $L_\l\otimes V\otimes V^*$ is generated by $v_\l\otimes v_a\otimes \bar v_b$ for $a,b\in I$, and $e_1,x_1$ commute with the $\mfg$-action, we obtain (c).
 \end{proof}
\begin{Lemma}\label{p4} For $k\in I^+$, we have
\begin{eqnarray*}
&\!\!\!\!\!\!\!\!\!&
(v_\lambda\otimes v_{\pm k})x_1=\mp\lambda_k v_\lambda \otimes v_{\pm k}\pm (-1)^{[ v_\lambda]}f_{k,k}v_{\lambda}\otimes v_{\mp k}\mp \mbox{$\sum\limits_{j<k}$}e_{k,j}v_\lambda \otimes v_{\pm j}\mp\(-1)^{[v_\lambda]} \mbox{$\sum\limits_{j<k}$}f_{k,j}v_\lambda \otimes v_{\mp j},\\
&\!\!\!\!\!\!\!\!\!\!\!\!\!\!\!\!\!\!\!\!\!\!\!\!&
(v_\lambda\otimes{\bar v_{\pm k}}) \bar x_1=\pm\lambda_k v_\lambda \otimes {\bar v_{\pm k}}\mp (-1)^{[ v_\lambda]}f_{k,k}v_{\lambda}\otimes \bar  v_{\mp k}  \pm \mbox{$\sum\limits_{i>k}$}e_{i,k}v_\lambda \otimes \bar v_{\pm i} \mp\(-1)^{[ v_\lambda]}\mbox{$\sum\limits_{i>k}$}f_{i,k}v_\lambda \otimes {\bar v_{\mp i}}.
\end{eqnarray*}
\end{Lemma}
\begin{proof}
The result follows from the definitions of $x_1$ and $\bar x_1$.
\end{proof}

For any $a\in I$, we set $a^+=|a|\in I^+$. Then \eqref{c-map} gives
\equa{CCSCS}{(c(a))^+=(\bar c(a))^+=a^+\mbox{ for }a\in I.}
Now we assume $BC_{2,r, t}$ is the level two walled-Brauer Clifford superalgebra such that $x_1,\bar x_1$ satisfy the degree 2 polynomials in Lemma \ref{SMSM}, and
 parameters satisfy
\equa{omega-i}{\omega_0\!=\!0,\
\omega_1\!=\!-2m(2p-2m+1),\
\omega_i\!=\!p(p+1)\omega_{i-2}
\mbox{ for $i\!\ge\!2$,}
}
where $ p\!\in\!\C\backslash\Z$, or $p\in\Z$ with $p>2m$ or $p<0$,
and $m\in\Z^{>0}$ satisfies $m\ge  2(r+t)$. By Lemma~4.4 and Definition~3.1(1)(3), we have
$e_1 f(x_1)=e_1g(\bar x_1)$.  We take $n=2m$. Take the weight $\l$ as in \eqref{highw}, then we can define the space $ M^{r,t}$ as in \eqref{M-st==}.

\begin{Prop}\label{hom} There is an algebra  homomorphism $\varphi: BC_{2,r, t} \rightarrow End_{\mfg}(M^{r,t})^{op}$ such that $\varphi$ sends the generators $e_1$, $x_1$, $\bar x_1$, $s_i$'s, $\bar s_j$'s, $c_m$'s, $\bar c_n$'s  to the same symbols defined in the Definition~\ref{casm}.\end{Prop}

\begin{proof}By Lemma~4.4(a)-(b),
we need to show the images of the generators satisfy the relations in Definition~\ref{awbsa}.
First it is easy to see that (2.1)--(2.3) and Definition~2.1 (1)--(10) are satisfied (cf.\cite{JK} or \cite [Theorem~1.4]{BGJKW}).
 By \cite[Thoerem~7.4.1]{HKS}, (2.7)--(2.8)  are satisfied.
 Moreover, (2), (4), (8)--(11) in Definition~\ref{awbsa} follows from the Definition~\ref{casm}.

Let $\Omega_{0,i}=\pi_{0,i}(\Omega_1) $ for $i\in J_1\cup J_2$,
$S_{i,j}=\pi_{i,j}(\Omega_0) $ for $i,j\in J_1\cup J_2$, and
$C_i=\pi_i(C)$ for $i\in J_1\cup J_2$, where $C=\sum_{i\in I^{+}} \bar f_{i,i}$.
Then $c_i= C_i|_{M^{r,t}}$ for $i\in J_1$ and   $\bar c_i= C_{\bar i}|_{M^{r,t}}$ for $\bar i\in J_2$. It follows from  the proof of \cite[Thoerem~7.4.1]{HKS}  that
\begin{itemize}
\item [(a)] $\Omega_1(1\otimes C )=-(1\otimes C )\Omega_1$,
\item [(b)] $S_{i,i+1}\Omega_{0,i}S_{i,i+1}=\Omega_{0,i+1}$,
\item [(c)] $(1^{\otimes {i-1}}\otimes C \otimes 1) S_{i,i+1}=S_{i,i+1}(1^{\otimes i }\otimes C)$,
\item [(d)] $\Omega_{0,i}\Omega_{0,j}-\Omega_{0,j}\Omega_{0,i}=(\Omega_{0,j}-\Omega_{0,i})S_{i,j}+(\Omega_{0,j}+\Omega_{0,i})
    C_iC_jS_{i,j}$.
 \end{itemize}
 Assuming that $i=1$ and $j=\bar 1$ in (d), we have \begin{equation}\label{sss1} \Omega_{0,1}(\Omega_{0,\bar 1}+S_{1,\bar 1}-C_1C_{\bar 1}S_{1,\bar 1})=\Omega_{0,\bar 1}\Omega_{0,1}+\Omega_{0,\bar 1}S_{1,\bar 1}+\Omega_{0,\bar 1}C_1C_{\bar 1}S_{1,\bar 1}.\end{equation}
 By (a)-(c), we get $\Omega_{0, \bar 1} C_1C_{\bar 1} S_{1, \bar 1}=-C_1S_{1, \bar 1} C_1\Omega_{0,  1}$, $\Omega_{0, \bar 1}S_{1,\bar 1}=S_{1,\bar 1} \Omega_{0, 1}$, $C_1 C_{\bar 1} S_{1, \bar 1}=C_1S_{1, \bar 1} C_1$, respectively.
  So, the right hand side of \eqref{sss1} is equal to
  \begin{equation} \label{sss2}\Omega_{0,\bar 1}\Omega_{0,1}+S_{1,\bar 1}\Omega_{0,1}-C_1S_{1,\bar 1}C_1\Omega_{0,1}=(\Omega_{0,\bar 1}+S_{1,\bar 1}-C_1S_{1,\bar 1}C_1)\Omega_{0,1}.\end{equation}
Now,  Definition~\ref{awbsa}(3) is satisfied by \eqref{sss1}-\eqref{sss2}. Consider the following  $\mfg$-homomorphism:
\begin{itemize}
\item[]
$\alpha:V \otimes V^{*} \rightarrow \mathbb C$
 such that$u\otimes \phi\mapsto (-1)^{[{\phi}] [{u}]} \phi(u)$;
\item[]
$\beta:\mathbb C\rightarrow V\otimes V^{*} $
such that $1\mapsto\sum_{i\in I}v_i\otimes \bar v_i$.
\end{itemize}
Then $e_1=-\beta\circ\alpha$.
Since $ \alpha\circ {x_1}^a\circ\beta \in End_{\mfg}(L_\lambda)$, and $L(\lambda)$ is of type $M$, we have  $\alpha\circ {x_1}^a\circ\beta=-{\omega_a}$, for some $ {\omega_a}\in\mathbb C$.
Then $e_1x_1^ae_1=(-\beta\circ\alpha)\circ x_1^a\circ(-\beta\circ\alpha)=-{\omega_a}\beta\circ\alpha=\omega_a e_1$.
Similarly, $e_1\bar x_1^a e_1=\bar \omega_a e_1$ for some ${\bar \omega_a}\in\mathbb C$. So, Definition~\ref{awbsa}
(5), (7) are satisfied.
Since $c_1^2=-1,\,\bar c_1^2=1,\,c_1x_1=-x_1c_1$,  we can get that $e_1c_1^2x_1^{2a}e_1=-e_1x_1^{2a}e_1$.
Hence $\omega_{2a}=-\omega_{2a}$ and $\,\omega_{2a}=0$, and Definition~\ref{awbsa}  (6) is satisfied.
Finally,  Definition~\ref{awbsa}  (1) follows from  Lemma \ref{p4}.
\end{proof}

By Proposition~\ref{hom},  $M^{r,t}$ is a right $BC_{2,r, t}$-module.
For any $\alpha,\b\in\Z_2^r$,
$\bar\alpha,\bar\beta\in\Z_2^t$, we %set $|\alpha|=\sum_i\alpha_i$, and
 define the following elements of  $BC_{2,r, t}$:
\begin{equation}\label{x-al-bx-b}
c^\alpha=\mbox{$\prod\limits_{i=1}^rc_i^{\alpha_i},\ \  x'^\beta=\prod\limits_{j=1}^r x'^{\beta_j}_j,$}\ \
\bar c^{\bar\alpha}=\mbox{$\prod\limits_{i=1}^t\bar c_i^{\bar\alpha_i},\  \ \bar x'^{\bar \beta}=\prod\limits_{j=1}^t\bar x'^{\bar\beta_j}_j,$}
\end{equation}where the product in $x'^\beta$ is written in the order $x_r'^{\beta_r}\cdots x_1'^{\beta_1}$ (thus $x_1'$ acts first on $M^{r,t}$)
and the like for $\bar x'^{\bar\beta}$.

\begin{Theorem}\label{theo-2222}The monomials
\begin{equation}\label{monom}
{\textit{\textbf{m}}}:=d_1^{-1}c^\alpha x'^\beta e^f \bar x'^{\bar\beta}\bar c^{\bar\alpha}wd_2,\end{equation}
with $\alpha,\beta\in\Z_2^r,\,\bar\alpha,\bar\beta\in\Z_2^t$ and $d_1,e^f,w,d_2$  as in
\eqref{bcabasis},
 are $\mathbb C$-linearly independent
elements of  $ BC_{2,r, t}$.
\end{Theorem}
\def\tif#1{{\textit{\textbf{#1{$\ssc\,$}}}}}\begin{proof}
Suppose there is a nonzero $\mathbb C$-combination $\tif c:=\sum_{\tif m}r_{\tif m}\tif m$ of
monomials \eqref{monom} being zero. We fix a monomial
$\tilde{\tif m}:=\tilde d^{-1}_1c^{\tilde \alpha} x'^{\tilde \beta} e^{\tilde f} \bar x'^{\bar{\tilde \beta}}\bar c^{\bar{\tilde \alpha}}\tilde w\tilde d_2$ in $\tif c$ with nonzero coefficient $r_{\tilde{\tif m}}\ne0$ which satisfies the following conditions:\begin{itemize}\item[(i)]
$|\tilde\beta|+|\bar{\tilde\beta}|$ is maximal;
\item[(ii)] $\tilde f$ is minimal    among all monomials satisfying (i).
\end{itemize}
We take the basis element $v= v_\l\OTIMES\OTIMES_{i\in J_1}v_{k_i}\OTIMES\OTIMES_{i\in J_2}\bar v_{k_i}\in B_{ M}$ (cf.~\eqref{basis-M-0}) such that
(note that here is the place where we require condition that $2(r+t)\le m$)
\begin{itemize}\item[(1)] for $1\le i\le r$, $k_i=i$ if $\tilde\b_i=0$ and $k_i=-(m+i)$ if $\tilde\b_i=1$;
\item[(2)] for $1\le i\le \tilde f$, $k_{\bar i}=k_i^+$; 
\item[(3)]
for $\tilde f<i\le t$, $k_{\bar i}=r+i_{\bullet}$ if $\bar{\tilde \b}_i=0$ and $k_{\bar i}=-(r+i)$ if $\bar{\tilde \b}_i=1$.
\end{itemize}
Take
\begin{eqnarray}
\label{b'===}&\!\!\!\!\!\!\!\!\!\!\!\!\!\!\!\!\!\!\!\!&
z:=(v) c^{\tilde\alpha}\tilde d_1{\tif c}\tilde d_2^{-1}\tilde w^{-1}\bar c^{\bar{\tilde\alpha}}\in M^{r,t},\\
\label{1b'===}&\!\!\!\!\!\!\!\!\!\!\!\!\!\!\!\!\!\!\!\!&
 \tilde u=\mbox{$\prod\limits_{i=1}^r$}f^{\tilde\b_i}_{i_{\bullet},i}
\mbox{$\prod\limits_{i=1}^t$}f^{\bar{\tilde\b}_i}_{r+i_{\bullet},r+i}v_\l\in B^e,
\\
\label{2b'===}&\!\!\!\!\!\!\!\!\!\!\!\!\!\!\!\!\!\!\!\!&
 \tilde v=\tilde u\OTIMES\OT{i\in J_1}v_{i}\OTIMES\OT{i\in J_2}\bar v_{r+i_{\bullet}}\in B_{ M},
 \end{eqnarray}
where the product in \eqref{1b'===} is in the same order as in \eqref{basis-u}.
We want to prove that when written $z$ as a combination of basis $B_{ M}$ in \eqref{basis-M-0}, the coefficient $\chi^{z}_{\tilde v}$ of  $\tilde v$ is nonzero.
Thus assume a monomial $\tif m$ in \eqref{monom} appears in the expression of $\tif c$ with $r_{\tif m}\ne0$ and consider the following element,
\begin{eqnarray}\label{u-11111}
z_1&\!\!\!:=\!\!\!&(v)c^{\tilde\alpha}\tilde d_1{\tif m}\tilde d_2^{-1}\tilde w^{-1}\bar c^{\bar{\tilde\alpha}}=
(v)c^{\tilde\alpha}\tilde d_1d_1^{-1}c^\alpha x'^\b e^f\bar x'^{\bar\b}\bar c^{\bar\alpha}wd_2\tilde d_2^{-1}\tilde w^{-1}\bar c^{\bar{\tilde\alpha}}\nonumber\\
&\!\!\!=\!\!\!&
\Big(\,v_\l\otimes\OT{i\in J_1}v_{ c^{\g_i+\tilde\alpha_i}(k_{(i)\tilde d_1d_1^{-1}})}\otimes\OT{i\in J_2}\bar v_{k_{(i)\tilde d_1d_1^{-1}}}\Big)x'^\b e^f\bar x'^{\bar\b}\bar c^{\bar\alpha}wd_2\tilde d_2^{-1}\tilde w^{-1}\bar c^{\bar{\tilde\alpha}},
\end{eqnarray}where $\g_i=\alpha_{(i)\tilde d_1d_1^{-1}}$,
and where the last equation is understood as ``equal up to a sign'' (cf.~\eqref{c-map}), which follows by noting that elements in $\Sigma_r\times\bar{\Sigma}_t$ have natural right actions on $J_1\cup J_2$ by permutations and $c$ acts on $I$ by \eqref{c-map}.
Write $z_1$ as a $\mathbb C$-combination of  basis $B_{ M}$. If $\tilde v$ appears as a term with a nonzero coefficient in the combination, then we say that $z_1$ {produces} $\tilde v$.

Note that $\tilde u$ has length $|\b|+|\bar\b|$. By Definition \ref{casm} and from our choice of $B^e$ in \eqref{BWE}, % (cf.~conditions (C1)--(C3)),
we see that  factors of $\tilde u$ can be only
contributed by the actions of $x'_i$'s and $\bar x'_i$'s, and each $x'_i$ or $\bar x'_i$ can at most contribute one length of $\tilde u$ by observing the following: if the first factor of a term in $\pi_{0i}(\Omega_{1})$ for $i\in J_1\cup J_2$ acting on the first factor of an element in $B_M$ changes the first factor to a basis element in $B^e$ then this $\pi_{0i}(\Omega_{1})$ may contribute one length, otherwise the first factor is changed to a combination of basis elements with length not increasing by Lemma \ref{Basis20}.
We see that $z_1$
cannot produce a basis element with degree higher than
$|\b|+|\bar\b|$.
Thus $\tilde v$ cannot be produced  if $|\b|+|\bar\b|<|\tilde\b|+|\bar{\tilde\b}|$.
So  by condition (i), we can assume \equa{b+beta=}{|\b|+|\bar\b|=|\tilde\b|+|\bar{\tilde\b}|.} Then
$f{\!}\ge{\!}\tilde f$ by condition~(ii).

For any basis element $b_M$ written as in \eqref{basis-M-0}, 
we say $k_i$ the {\it $i$-th label} of  $b_M$ for $i\in J_1\cup J_2$.
Note from  \eqref{1b'===} that
all factors of $\tilde u$ have the following form
 \equa{f-ij=}{\mbox{$f_{i_{\bullet},i}$ with $i\in I_1^+$}.}
Thus  when $x'_i=-\pi_{0i}(\Omega_1)$ for $1\le i\le r$ is applied to
the element inside the bracket, it can only change its $i$-th label, say $\ell$, to $\pm\ell$, $\pm(\ell-m)$.
Since $\tilde d_1d_1^{-1}$ only permutes labels and $c^{\g_i+\tilde\alpha_i}$ only changes labels up to a sign,
in order for a term in \eqref{u-11111} to contribute to $\chi^{z_1}_{\tilde v}$,
we need at least $f$ pairs $(i,\bar j)\in J_1\times J_2$ such that the $i$-th label $k_i$ and $j$-th  label $k_{\bar j}$
satisfy the condition $k_i^+=k_{\bar j}^+$ or $k_i^+=k_{\bar j}^++m$. From our choice of the vector $v$, we must have $f\le\tilde f$. Thus we can suppose $\tilde f=f$ by the fact that $f\ge \tilde f$.

Set $J_{f}=\{i,\bar i\,|\,1\le i\le \tilde f=f\}\subset J_1\cup J_2$ (cf.~\eqref{ordered-set}).
If $d_1\ne \tilde d_1$,
then by definition \eqref{drt}, 
 we have \begin{equation}\label{j---in}\tilde j :=(j)\tilde d_1d_1^{-1}\notin J_{f}\mbox{ \ for some $j\in J_{f}$}.\end{equation} Say $\tilde j \!\in\! J_1$ (the proof is similar if $\tilde j \!\in\! J_2$), then $f\!<\!\tilde j \!\le\! r$.
Condition (1) shows that $k^+_{\tilde j }= \tilde j$.
Then conditions (2) and (3) show that there is no $\bar\ell\in J_2$ with $k^+_{\tilde j }= k^+_{\bar\ell }$ or  $k^+_{\tilde j }= k^+_{\bar\ell }+m$. Since all factors of $\tilde u$ have the form
 \eqref{f-ij=}, we see that $z_1$ cannot produce the basis element $\tilde v$.
Thus we can suppose $\tilde d_1=d_1$. Then $c^{\g_i+\tilde\alpha_i}(k_{(i)\tilde d_1d_1^{-1}})=c^{\alpha_i-\tilde\alpha_i}(k_i)$ (note that $c^2=1$ acting on $I$). If $\alpha_i\ne\tilde\alpha_i$ for some $1\le i\le f$, then $c^{\alpha_i-\tilde\alpha_i}(k_i)=-k_i$ and after applying $x'^{\beta}$ to the element inside the bracket in \eqref{u-11111}, we obtain an element which satisfies the condition that either its $i$-th label is not $i$ (in this case after we apply $e^f$ we obtain the zero element) or else its zero-th factor cannot contain the factor
$\prod_{i=1}^rf^{\tilde\b_i}_{i_{\bullet},i}$. In any case we cannot obtain the element $\tilde v$. Thus we can assume $\alpha_i=\tilde\alpha_i$ for  $1\le i\le f$. Similarly, we can assume $\alpha_i=\tilde\alpha_i$ for  $f<i\le r$, i.e.,
$\alpha=\tilde\alpha$.

By conditions (1) and (2), we see that if $\beta_i\ne\tilde\b_i$ for some $i$ with $1\le i\le f$, or $\beta_i=1>\tilde\b_i$ for some $i\in J_1$, then again
$z_1$ cannot produce the basis element $\tilde v$.~Thus we suppose: $\beta_i\!=\!\tilde\b_i$ if $1\!\le\! i\!\le\! f$, and
$\beta_i\le\tilde\b_i$ for~$i\!\in\! J_1$. 
If $\tilde\b_i=1$ but $\b_i=0$ for some $i\in J_1$, then by \eqref{f-ij=}, $z_1$ can only produce some basis elements which have
 at least a tensor factor, say $v_\ell$, with $\ell\in I^-$, and thus $\tilde v$ cannot be produced.
Hence we can suppose  $\b=\tilde\b$. Dually, we can suppose~$\bar\beta\!=\!\bar{\tilde\beta}$.

Rewrite $wd_2\tilde d_2^{-1}\tilde w^{-1}$ as $wd_2\tilde d_2^{-1}\tilde w^{-1}=d_{20}\tilde d_{20}^{-1}w'$, where $d_{20}=wd_2w^{-1}$,
$\tilde d_{20}=w\tilde d_2w^{-1}$ and $w'=w\tilde w^{-1}$. Note that 
$w'\in \mathfrak S_{r-f} \times \bar{\mathfrak S}_{t-f}$, which only permutes elements of $(J_1\cup J_2)\bs J_{f}$.
We see that if $ d_{20}\ne \tilde d_{20}$, then as in \eqref{j---in}, there exists some $j\in J_{f}$ with
$\tilde j :=(j)d_{20}\tilde d_{20}^{-1}w'\notin J_{f }$, thus $\tilde u_{ M}$ cannot be produced. So assume $ d_{20}=\tilde d_{20}$.
Similarly we can suppose $w'=1$. Then the same arguments after \eqref{j---in} show that we can assume $\bar\alpha=\bar{\tilde\alpha}$ (cf.~\eqref{j---in}).

The above has in fact proved that if the coefficient $\chi^{z_1}_{\tilde v}$ is nonzero then  $z_1$ in \eqref{u-11111} must satisfy
$(d_1,\alpha,\beta,f,{\bar\beta},{\bar\alpha},w,d_2)=
(\tilde d,{\tilde \alpha},{\tilde \beta},{\tilde f},{\bar{\tilde \beta}},{\bar{\tilde \alpha}},\tilde w,\tilde d_2)$, i.e.,
$z_1=(v)\tilde{\tif m}$. In this case, one can easily verify that $\chi^{z_1}_{\tilde v}\!=\!\pm1$. This proves that $z$ defined in \eqref{b'===} is nonzero, a contradiction. The theorem is proven.
\end{proof}

\begin{Cor}\label{level-2} $BC_{2, r, t}$ has a $\mathbb C$-basis which consists of all   regular monomials of it.
\end{Cor}
\begin{proof} We have the result immediately from Corollary~\ref{level-l-span} and Theorem~\ref{theo-2222}.\end{proof}

\section{Homomorphisms between $\widetilde{BC}_{r, t}$  and $BC_{2, r+k, t+k}$ }
In this section, we generalize Theorem~\ref{level-1} so as to  establish infinite many homomorphisms from  $ BC^{\rm aff}_{r, t}$ to  $BC_{2, r+k, t+k}$ for all positive integers $k$, where $BC_{2, r+k, t+k}$ are level two walled Brauer-Clifford superalgebras which appear in the higher level mixed Schur-Weyl-Sergeev duality  in section~4.  As an application, we prove that $ BC^{\rm aff}_{r, t}$ has $R$-basis  which consists of  all regular monomials in the Definition~\ref{regm}.
Recall $x_i', \bar x_j'$  in \eqref{relsmurp}.
\begin{Lemma}\label{usefaff} For all admissible $i, j$, we have the following results in $BC_{r, t}^{\rm aff}$:
$s_j\bar x'_i=\bar x'_i s_j$, $\bar s_j  x'_i=x'_i\bar s_j$,  $x'_i\bar c_j=\bar c_j x'_i$ and $\bar x'_i c_j=c_j\bar x'_i$.
\end{Lemma}
\begin{proof} Easy exercises.
\end{proof}

\begin{Lemma}\label{usefaffco} Recall  $y_i$ and $\bar y_j$ in Definition~$\ref{RSu-GRSS}$. The following results hold  in   $ BC_{r, t}^{\rm aff} $ for all admissible  $i,j$:
\begin{multicols}{2}
\begin{enumerate}
\item [(1)] $x'_i \bar y_i=\bar y_i x'_i$,
\item [(2)] $\bar x'_i y_i=y_i \bar x'_i$,
\item [(3)] $x'_{i+1}  y_i= y_i x'_{i+1}$,
\item [(4)] $\bar x'_{i+1} \bar  y_i= \bar y_i \bar x'_{i+1}$,
\item [(5)] $(\bar x'_{i+1}+\bar y_{i+1})  x_j'= x_j'(\bar x'_{i+1}+\bar y_{i+1})$ if  $j\le i$,
\item [(6)] $(x'_{i+1}+y_{i+1}) x_j'=x_j'(x'_{i+1}+y_{i+1})$, if  $j\le i$,
\item [(7)] $(\bar x'_{i+1}+\bar y_{i+1}) \bar  x_j'= \bar x_j'(\bar x'_{i+1}+\bar y_{i+1})$ if  $j\le i$,
\item [(8)]  $(x'_{i+1}+y_{i+1})\bar  x_j'=\bar x_j'(x'_{i+1}+y_{i+1})$, if  $j\le i$.
\end{enumerate}\end{multicols}
\end{Lemma}
\begin{proof}
By symmetry, it is enough to  prove (1), (3), (5), (6).

(1) If $j\le i-1$, then $x_i' e_{j, i}=e_{j, i}x_i'$, $x'_i\bar c_j=\bar c_j x'_i$, and $\bar s_j  x'_i=x'_i\bar s_j$   by Lemmas~\ref{relxpri}(1) and \ref{usefaff}. So,
   $$x'_i \bar y_i=x_i'\Big(\mbox{$\sum\limits_{j=1}^{i-1}$} (e_{j, i}-\bar e_{j, i})-\bar L_i\Big)=\bar y_i x'_i.$$
One can check (3) via Definition~\ref{RSu-GRSS} similarly.

(5)
By Lemmas~\ref{hecrel2}(1),~\ref{relxpri}(1)--(2), $x'_j(\bar x'_{i+1}+e_{j, i+1}-\bar e_{j, i+1})=(\bar x'_{i+1}+e_{j, i+1}-\bar e_{j, i+1})x'_j$ and
$x'_j e_{s, i+1 }=e_{s, i+1 }x_j'$ and $x'_j \bar e_{s, i+1 }=\bar e_{s, i+1 }x_j'$ whenever $j\neq s$.  Since
$x'_j \bar L_{i+1}=\bar L_{i+1} x_j'$, we have
$$\begin{aligned}  (\bar x'_{i+1}+\bar y_{i+1})x_j' &=\Big(\bar x'_{i+1}+e_{j, i+1}-\bar e_{j, i+1}+\mbox{$\sum\limits_{1\le s\le i, s\neq j}$}(e_{s, i+1 }-\bar e_{s, i+1})-\bar L_{i+1}\Big)x_j'\\
&= x_j'\Big(\bar x'_{i+1}+e_{j, i+1}-\bar e_{j, i+1}+\mbox{$\sum\limits_{1\le s\le i, s\neq j}$}(e_{s, i+1 }-\bar e_{s, i+1})-\bar L_{i+1}\Big)\\ & =x_j' (\bar x'_{i+1}+\bar y_{i+1}).\\
\end{aligned}$$

(6)  By \eqref{asera}, \eqref{comm-hc} and  Lemma~\ref{relxpri}(1),
 we have
$$\begin{aligned} x_1(x_{i+1}'+y_{i+1})& =x_1\Big(x_{i+1}+\mbox{$\sum\limits_{j=1}^{i}$} (e_{i+1, j}+\bar e_{i+1, j})\Big)= \Big(x_{i+1}+\mbox{$\sum\limits_{j=1}^{i}$} (e_{i+1, j}+\bar e_{i+1, j})\Big)x_1\\ &= (x_{i+1}'+y_{i+1})x_1.\\ \end{aligned}$$
Applying $(1, j)$ on both sides of the above equation yields (6). \end{proof}

For the simplification of notation, we define \begin{equation} \label{zi} z_i=x_i'+y_i \text{ and $\bar z_j=\bar x_j'+\bar y_j$}\end{equation}
for all admissible $i$ and $j$.
\begin{Lemma}\label{usefcycl}The following results hold in   $ BC_{r, t}^{\rm aff} $ for all admissible  $i, j$:
\begin{multicols}{2}
\begin{itemize}
\item [(1)] $s_j z_i =z_i s_j$, $\bar s_j \bar z_i=\bar z_i \bar s_j$,  if $j\neq i-1, i$,
\item  [(2)] $s_j \bar z_i=\bar z_i s_j$, $\bar s_j z_i =z_i \bar s_j$,  if $j\neq i-1$,
\item [(3)] $z_i  c_i=-c_iz_i $, $\bar z_i \bar c_i=-\bar c_i \bar z_i$,
\item [(4)] $z_i c_j=c_j z_i $, $\bar z_i \bar c_j=\bar c_j \bar z_i$, if $i\neq j$,
\item [(5)] $z_i\bar c_j=\bar c_j z_i$, $\bar z_i c_j=c_j\bar z_i$, if $i\leq j$,
\item [(6)] $z_i(e_i+\bar z_i-\bar e_i)=(e_i+\bar  z_i-\bar e_i)z_i $,
\item [(7)] $e_{i} \bar z_i=-e_i( x_i+\bar L_i)$, $e_{i} z_i=e_i(x_i+\bar L_i)$,
\item [(8)] $e_i s_i z_is_i=s_i z_i  s_i e_i$,  $e_j \bar s_j \bar z_j\bar s_j=\bar s_j \bar z_j\bar s_j e_j$,
\item [(9)] $z_i\tilde z_i  =\tilde z_iz_i$,
\item [(10)] $\bar z_i\tilde{\bar z_i} =\tilde{\bar z_i}\bar z_i$,
\item [(11)] $e_i z_i e_i=e_i x'_i e_i=\omega_1 e_i $,
\item [(12)] $ e_i z_i^k c_i e_i=0$,  $\forall k\in \mathbb N$,
\item [(13)] $e_i z_i^{2n} e_i=0 $, $e_i (\bar z_i)^{2n} e_i=0$  $\forall n\in \mathbb N$,
\end{itemize}\end{multicols} where
$\tilde z_i=(s_i z_i s_i -(1-c_ic_{i+1})s_i)$ and $\tilde{\bar z_i}=(\bar s_i \bar z_i \bar s_i -(1+\bar c_i\bar c_{i+1})\bar s_i)$.

\end{Lemma}
\begin{proof} (1)--(5) follows from Lemma~\ref{usef}(1)--(5), Lemma~\ref{usefaff} and
\eqref{asera} and \eqref{aseradual2}.
 (6) follows from Lemmas~\ref{usef}(7),~\ref{hecrel2}(1),~\ref{usefaffco}(1)--(2).
 (7) follows from Lemmas~\ref{usef}(8),~\ref{hecrel2}(3)--(4). (8) follows from Lemmas~\ref{usef}(9),~\ref{relxpri}(1).
 Multiplying $(2, i+1)(1, i)$ on both sides of  $x_1x_2=x_2x_1$ (see \eqref{asera})
 yields \begin{equation} \label{bii1} x_i'(x_{i+1}'-(1-c_ic_{i+1})s_i)=(x_{i+1}'-(1-c_ic_{i+1})s_i)x_i'.\end{equation}
 Now, (9) follows from \eqref{bii1},  Lemmas~\ref{usef}(10),~\ref{usefaffco}(3)--(4). We leave (10) to the reader since it  can be verified, similarly.
 (11) follows from Lemmas~\ref{usef}(13),~\ref{relxpri}(3).
  Via (6), one can prove  (12) by arguments similar to those for  Lemma~\ref{usef}(11). Finally, one can verify   (13) by arguments similar to those
   for Lemma~\ref{usef}(12).
\end{proof}

From here to the end of Theorem~\ref{alghomcyco}, we assume the ground ring $R$ is $\mathbb C$. Also,  $BC_{2, r, t}$  is one of those which appear in  the higher  level mixed Schur-Weyl-Sergeev duality in section~4.

\begin{Lemma}\label{step2}  The $e_k BC_{2, k, k}$ is  the left $BC_{2, k-1, k-1}$-module  spanned by
$e_k c_k^{\sigma_1}x_k'^{\sigma_2} s_{k, j} \bar s_{k, l}$ for all $ \sigma_1, \sigma_2\in \Z_2$ and $ 1\le j, l\le k$,
where $BC_{2, k-1, k-1}$ is the subalgebra of $BC_{2, k, k}$ generated by $c_1, \bar c_1$,$ x_1, \bar x_1$, $ s_1, \ldots, s_{k-2}$ and  $\bar s_1, \ldots, \bar s_{k-2}$.
\end{Lemma}
\begin{proof} This  result, which  is a counterpart of Lemma~\ref{spannwbc}, can be proved similarly. \end{proof}

\begin{Prop}\label{esecycl}\begin{enumerate} \item [(1)]
$e_k BC_{2, k, k} e_k=e_{k} BC_{2, k-1, k-1}$.\item[(2)] Recall $z_k$ and $\bar z_k$ in \eqref{zi}.  There is a unique $\xi_{a, k}$ $($resp., $\bar \xi_{a, k}${}$)$ in $BC_{2,k-1, k-1}$ such that $e_k z_k^a e_k=\xi_{a, k} e_k$  (resp., $e_k \bar z_k^a e_k=\bar \xi_{a, k} e_k)$. Moreover, $\xi_{2n,k}=\bar \xi_{2n,k}=0$, and
$\xi_{1, k}=\omega_1$. \end{enumerate}\end{Prop}
\begin{proof} (1) follows from  Lemma~\ref{step2} (see the proof of Proposition~\ref{ese1}) and (2) follows from  (1) and  Corollary~\ref{level-2} for $BC_{2, k, k}$ and Lemma~\ref{usefcycl}(11), (13).\end{proof}

 Lemmas~\ref{assumcycl} and \ref{assum1cycl} can be proven by arguments similar to those for Lemma~\ref{y1}--\ref{y2}.

\begin{Lemma}\label{assumcycl} For any $n\in \mathbb N$,  $e_i\bar z_i^{2n+1}=\sum_{j=0}^n a^{(i)}_{2n+1, j} e_iz_i^{2j+1}$ for some    $a^{(i)}_{2n+1,j}\in R[ \xi_{3,i}, \ldots, \xi_{2n-1,i}]$  such that
\begin{enumerate}\item [(1)] $a^{(i)}_{2n+1, n}=-1$,
\item [(2)] $a^{(i)}_{2n+1, j}=a^{(i)}_{2n-1, j-1}$ for all $1\le j\le n-1$,
\item [(3)] $a^{(i)}_{2n+1, 0}=\sum_{j=0}^{n-1} a^{(i)}_{2n-1, j}\xi_{2j+1,i}$.\end{enumerate}
\end{Lemma}

\begin{Lemma}\label{assum1cycl} For any positive integer $n$, $e_i\bar z_i^{2n}=\sum_{j=0}^n a^{(i)}_{2n, j} e_i z_i^{2j}$ for some $a^{(i)}_{2n,j}\in R[ \xi_{3,i}, \ldots, \xi_{2n-1,i}]$  such that
\begin{enumerate}\item  [(1)] $a^{(i)}_{2n, n}=1$,
\item [(2)] $a^{(i)}_{2n, j}=a^{(i)}_{2n-1, j-1}$ for all $1\le j\le n-1$,
\item [(3)] $a^{(i)}_{2n, 0}=\sum_{j=0}^{n-1} a^{(i)}_{2n-2, j}\xi_{2j+1,i}$.  \end{enumerate}
\end{Lemma}

 We can assume $k \ge 2$ (resp., $n\ge 1$)  in the Lemma~\ref{free1cycl}  since  $z_k=x_k'+y_k$ and $\bar z_k=\bar x_k'+\bar y_k$ (resp., $\xi_{1, k}=\omega_1$ and $\bar \xi_{1, k}=-\omega_1$ by Lemma~\ref{usefcycl}(11)).

\begin{Lemma}\label{free1cycl} We have $\bar \xi_{2n+1, k}\in  R[\xi_{3, k}, \ldots, \xi_{2n+1, k}]$ if  $k \in\Z^{\ge2}$ and $n\in \mathbb Z^{\ge 1}$. Furthermore,  both $\xi_{2n+1, k}$
and $\bar\xi_{2n+1, k}$ are central in $BC_{2,k-1, k-1}$.\end{Lemma}

\begin{proof} The first statement follows from Lemma~\ref{assumcycl}.
We have $s_jz_k=z_ks_j$ and $c_jz_k=z_kc_j$ for all $j\le k-1$ by Lemma~\ref{usefcycl}(1)--(4). Since $y_k=\sum_{i=1}^{k-1} (e_{k, i}+\bar e_{k, i})-L_k$£¬   $z_k=x_k'+y_k=x_k+\sum_{i=1}^{k-1} (e_{k, i}+\bar e_{k, i})$. By \eqref{asera} and Lemma~\ref{relxpri}(1) and (2), $x_1z_k=z_kx_1$ for $k\ge 2$.
Obviously, $e_1$ commutes with $x_k'$ and $e_{k, i}, \bar e_{k, i}, c_k, (k, i)$ whenever $i\neq 1$ and $k\ge 2$. Since $e_1 e_{k, 1}=e_1 (k, 1)$, we have
$e_1z_k=z_ke_1$.

 We have proved that  $h$ commutes with $e_k, z_k$ for
  any $h\in \{e_1, s_i,c_j,x_1 \mid  1\le i\le k-2\ ,1\le j\le k-1\} $. So   $e_k (h \xi_{a, k})=e_k(\xi_{a, k} h)$. By  Corollary~\ref{level-2} and Proposition~\ref{esecycl}, $h \xi_{2n+1, k}=\xi_{2n+1, k} h$.
  Finally, we need to check $e_k (h \xi_{a, k})=e_k(\xi_{a, k} h)$ for any $h\in \{\bar s_i,\bar c_j,\bar x_1 \mid  1\le i\le k-2\ ,1\le j\le k-1\}$. In this case, we use Lemma~\ref{assumcycl} so as to  use $\bar z_k $ instead of $z_k$
in $e_k z_k^{2n+1}  e_k$. Therefore,  $h \xi_{2n+1, k}=\xi_{2n+1, k} h$, as required.
\end{proof}

\begin{Lemma} For $k,a\in\Z^{\ge 1}$, we have
 \begin{equation}\label{sykcycl}
 s_k z_{k+1}^a =h_k^a s_k-\mbox{$\sum\limits_{b=0}^{a-1}$} h_k^{a-1-b}z_{k+1}^b  +\mbox{$\sum\limits_{b=0}^{a-1}$} (-1)^{a-b} c_kc_{k+1} h_k^{a-b-1}
 z_{k+1}^b,
 \end{equation} where $h_k=z_k+e_k+\bar e_k$.
\end{Lemma}
\begin{proof}The result can be easily checked  by  induction on $a$.
  \end{proof}

Recall  $z_i, \bar z_j$ in \eqref{zi} for all admissible $i$ and $j$.
For all $1\le j\le k-1$, define
\begin{equation}\label{z-bar-zcycl} p_{j, k}=s_{j, k-1} (z_{k-1} +e_{k-1}+\bar e_{k-1})s_{k-1, j},  \ \ \text{and  $\bar p_{j, k}=\bar s_{j, k-1} (\bar z_{k-1}+e_{k-1}-\bar e_{k-1})\bar s_{k-1, j}$.}\end{equation}
Note that $\xi_{0,k}=0$, and $e_k h=0$ for $h\in BC_{2,k-1, k-1}$ if and only if $h=0$. We will use this fact  freely in the proof of
the following lemma, where we use the terminology that
a monomial in  $p_{j, k+1}$'s  $\bar p_{j, k+1}$'s, and $ x_j', \bar x_j'$
is a {\it leading term} in an expression if it has
the highest degree by defining ${\rm deg\,}p_{i,j}={\rm deg\,}\bar p_{i,j}= {\rm deg\,}x'_j={\rm deg\,}\bar x'_j=1$.

\begin{Lemma}\label{omedcycl} For any positive integer $n$,   $\xi_{2n+1, k+1}$ can be written as an $R$-linear combination of
monomials in   $ p_{j, k+1}$, $\bar p_{j, k+1}$, $x'_j$ and  $\bar x'_j$ for $1\le j\le k$ such that the leading terms
of  $\xi_{2n+1, k+1}$ are
 $-2\sum_{j=1}^{k}(p_{j, k+1}^{2n}+p_{j, k+1}^{2n-1} x_j'-\bar p_{ j, k+1}^{2n}-\bar p_{j, k+1}^{2n-1}\bar x_j')$.
\end{Lemma}

\begin{proof} We have  $x_1^2=p(p+1)$ and $\bar x_1^2=(p-n+1)(p-n)$  in $BC_{2, r, t}$ by  Lemma~\ref{SMSM}.
So,
\begin{equation}\label{ommmmcycl}{\sc\!}
\begin{aligned}
& \xi_{2n+1,k+1} e_{k+1}\overset{\rm Prop.~5.5} =  e_{k+1}z_{k+1}^{2n+1} e_{k+1}=e_{k+1}y_{k+1} z_{k+1}^{2n}e_{k+1}+e_{k+1}x'_{k+1}z_{k+1}^{2n}e_{k+1}\\ = & e_{k+1}y_{k+1}
z_{k+1}^{2n}e_{k+1}+e_{k+1}(x'_{k+1})^2 z_{k+1}^{2n-1}e_{k+1}+e_{k+1}x'_{k+1}y_{k+1}z_{k+1}^{2n-1}e_{k+1}\\ = & e_{k+1}y_{k+1}z_{k+1}^{2n}e_{k+1}+
p(p+1)e_{k+1} z_{k+1}^{2n-1}e_{k+1}-e_{k+1}\bar x'_{k+1}y_{k+1}z_{k+1}^{2n-1}e_{k+1} \text{ (by Lemma~\ref{hecrel2}(3))}\\= &e_{k+1}y_{k+1}z_{k+1}^{2n}e_{k+1}+p(p+1)e_{k+1} z_{k+1}^{2n-1}e_{k+1}+e_{k+1}\bar y_{k+1}\bar x'_{k+1}z_{k+1}^{2n-1}e_{k+1}  \text{ (Lemmas~\ref{usef}(8),~\ref{usefaffco}(2))}  \\=& e_{k+1}(-L_{k+1}+\bar L_{k+1})z_{k+1}^{2n}e_{k+1}+p(p+1)\xi_{2n-1, k+1} e_{k+1} +e_{k+1}(L_{k+1}-\bar L_{k+1} )\bar x'_{k+1}z_{k+1}^{2n-1}e_{k+1}.
\end{aligned}\!\!
\end{equation}
Recall  $h_k=z_k+e_k+\bar e_k$ in \eqref{sykcycl}.
Considering the right-hand side of \eqref{ommmmcycl} and expressing $L_{k+1}$ by \eqref{RSu-GRSS}, using $(j,k+1)=$ $s_{j,k}s_ks_{k,j}$  and the fact that $s_{j,k},s_{k,j}$ commute with $x'_{k+1},y_{k+1},e_{k+1}$, a term in the linear combination of  $e_{k+1} L_{k+1}z_{k+1}^{2n}e_{k+1}$ becomes
\begin{equation} \label{step44} s_{j, k} e_{k+1} s_{k} z_{k+1}^{2n} e_{k+1}
s_{ k, j}\! +
s_{j, k} e_{k+1}c_{k+1} s_{k} c_{k+1} z_{k+1}^{2n} e_{k+1} s_{ k, j}.
\end{equation}
Since $c_{k+1} x_{k+1}'=-x_{k+1}'c_{k+1}$ and $c_{k+1} y_{k+1}=-y_{k+1} c_{k+1}$, we have  $c_{k+1} z_{k+1}^{2n}=z_{k+1}^{2n} c_{k+1}$. Note that $e_{k+1}c_{k+1}=e_{k+1}\bar c_{k+1}$, and
$\bar c_{k+1}$ commutes with $s_k, x_{k+1}'$ and $y_{k+1}$. So
\begin{equation}\label{step45} \begin{aligned} & s_{j, k} e_{k+1}c_{k+1} s_{k} c_{k+1} z_{k+1}^{2n} e_{k+1} s_{ k, j}=s_{j, k} e_{k+1} \left(s_{k} z_{k+1}^{2n}\right) e_{k+1}
s_{ k, j}\! \\
\overset {  \eqref{sykcycl}}=& s_{j, k} e_{k+1}\left( h_k^{2n} s_k -\mbox{$\sum\limits_{b=0}^{2n-1}$} h_k^{2n-b-1} z_{k+1}^b
+(-1)^{b} \mbox{$\sum\limits_{b=0}^{2n-1}$} c_k c_{k+1}  h_k^{2n-b-1}  z_{k+1}^b \right)e_{k+1}s_{k, j} \\=&s_{j, k}  h_k^{2n}s_{k,j}e_{k+1}-s_{j, k}\mbox{$\sum\limits_{b=0}^{2n-1}$}  h_k^{2n-b-1} s_{k, j} \xi_{b,k+1}e_{k+1}\text{ (by Lemma~\ref{usefcycl}(12)).}\end{aligned}\end{equation}
By induction assumption, the leading terms
of  $\xi_{b, k+1}$ are of degree $b-1$. So, the leading term of $s_{j, k} e_{k+1} \left(s_{k} z_{k+1}^{2n}\right) e_{k+1}
s_{ k, j}$ is $p_{j, k+1}^{2n}$ and hence  $-e_{k+1} L_{k+1}z_{k+1}^{2n}e_{k+1}$ contributes to the leading terms $-2\sum_{j=1}^k p_{j,  k+1}^{2n}$.

We compute  $e_{k+1} L_{k+1}\bar x_{k+1}'z_{k+1}^{2n-1}e_{k+1}$. A term of it becomes $2s_{j, k} e_{k+1}s_k\bar x'_{k+1}z_{k+1}^{2n-1}e_{k+1}s_{k,j}$ since $$e_{k+1}c_{k+1}s_k  c_{k+1}\bar x'_{k+1}z_{k+1}^{2n-1}e_{k+1}=  e_{k+1}s_k\bar x'_{k+1}z_{k+1}^{2n-1}e_{k+1}.$$ Thus, it is enough to compute the leading terms of
$s_{j, k} e_{k+1}s_k\bar x'_{k+1}z_{k+1}^{2n-1}e_{k+1}s_{k,j}$. Since $x_k'$ commutes with $z_{k+1}$ and $e_{k+1}$,
and $e_{k+1} s_k\bar x_{k+1}'=e_{k+1} \bar x_{k+1}'s_k=-e_{k+1}  x_{k+1}'s_k=-e_{k+1} s_k x_{k}'$,  we have
$$s_{j, k} e_{k+1}s_k\bar x'_{k+1}z_{k+1}^{2n-1}e_{k+1}s_{k,j}=-s_{j, k} e_{k+1}s_k z_{k+1}^{2n-1}e_{k+1}s_{k,j} x'_{j}$$
By \eqref{step45}, $s_{j, k} e_{k+1}s_k\bar x'_{k+1}z_{k+1}^{2n-1}e_{k+1}s_{k,j}$ contributes to the leading term $\-p_{j, k+1}^{2n-1}x_j'$ whose degree is  $2n$.
Finally, we need to compute the leading terms of $e_{k+1} \bar L_{k+1}z_{k+1}^{2n}e_{k+1}$ and $-e_{k+1}\bar L_{k+1} \bar x'_{k+1}z_{k+1}^{2n-1}e_{k+1}$,
By Lemma~\ref{assum1cycl}, one can use $\bar z_{k+1}^{2n}$ to replace $z_{k+1}^{2n}$ in $e_{k+1} \bar L_{k+1}z_{k+1}^{2n}e_{k+1}$. Thus, $e_{k+1} \bar L_{k+1}z_{k+1}^{2n}e_{k+1}$ contributes to the leading terms $\bar p_{j, k+1}^{2n}$.
Similarly, a term of $e_{k+1}\bar L_{k+1}\bar x'_{k+1}z_{k+1}^{2n-1}e_{k+1}$ is of form $\bar x_j' e_{k+1} (\bar j, \overline {k+1}) z_{k+1}^{2n-1} e_{k+1}$ whose
leading term $\bar p_{j, k+1}^{2n-1}\bar x_j'$ is of degree $2n$.
The proof is completed.
\end{proof}

\begin{Lemma}\label{commomegacycl}For $a\!\in\!\Z^{\ge0},{\ssc\,}k\!\in\!\Z^{\ge 1}$,~both~$\xi_{a, k+1}$ and
$\bar \xi_{a, k+1}$ commute with $x'_{k+1}+y_{k+1}$ and $\bar x'_{k+1}+\bar y_{k+1}$.\end{Lemma}
 \begin{proof}By Proposition~\ref{esecycl}, we can assume that $a=2n+1$ and $n\geq 1$.  By Lemma~\ref{omedcycl}, $\xi_{2n+1, k}$ can be written as a linear combinations of monomials  in
 $p_{j, k+1}$, $\bar p_{j, k+1}$,  $x'_j$, and $\bar x'_j$ for $1\le j\le k$. From \eqref{z-bar-zcycl},
 \begin{equation}\label{pze} p_{j, k+1}=x_j'+z_{j, k+1}, \text{ and $\bar p_{j, k+1}=\bar x_j'+\bar z_{j, k+1}$,}\end{equation}
 where $z_{j,k+1} ,\bar z_{j,k+1}$ are defined in Lemma~\ref{dzjk}.
 So, it is enough to prove that $z_{j, k+1}$, $\bar z_{j, k+1}$,  $x'_j$ and  $\bar x'_j$  commute with $\bar x'_{k+1}+\bar y_{k+1}$ if  $1\le j\le k$,. By Lemma~\ref{usefaffco}(5)--(6), both $x'_j$ and $\bar x'_j$ commute with $\bar x'_{k+1}+\bar y_{k+1}$. Finally,  $z_{j, k+1}$, $\bar z_{j, k+1}$ commute with $\bar x'_{k+1}+\bar y_{k+1}$ since (1): both $z_{j, k+1}$ and $\bar z_{j,k+1}$ commute with $y_{k+1}$ and $\bar y_{k+1}$  (see the proof of Lemma~\ref{commomega}) and (2):
 $z_{j,k+1}$ and $\bar z_{j,k+1}$   are linear combinations of elements which  commutes with  both  $\bar x'_{k+1}$ and   $ { x'_{k+1}}$, by  Lemmas~\ref{dzjk},
 ~\ref{relxpri}(1)--(2), and ~\ref{usefaff} and \eqref{asera} and \eqref{aseradual2}. This proves that $\xi_{a, k+1}$ commutes with $x'_{k+1}+y_{k+1}$ and $\bar x'_{k+1}+\bar y_{k+1}$ and so is
 $\bar \xi_{a, k+1}$ by Lemma~\ref{assumcycl}.\end{proof}

\begin{Lemma}\label{commom1egacycl}For $a\!\in\!\Z^{\ge0},{\ssc\,}k\!\in\!\Z^{\ge 1}$, both $\xi_{2a+1, k+1}$ and
$\bar \xi_{2a+1, k+1}$ commute with $c_j$ and $\bar c_j$ if $j \geq k+1$.\end{Lemma}
\begin{proof} We have proved that both  $z_{j, k+1}$ and $\bar z_{j, k+1}$  commute with $c_j$ and $\bar c_j$ if $j \geq k+1$ in the proof of Lemma~\ref{commomega}.
 If $j\le k$, then  $ x_j'$ and $\bar x_j $ commute with $c_l$ and $\bar c_l$ for $l\ge k+1$, by Lemma~\ref{usefaff} and \eqref{asera} and \eqref{aseradual2}. Now, the result follows from
 \eqref{pze} and Lemma~\ref{omedcycl}, immediately. \end {proof}

In the following result, we assume the ground field is $\mathbb C$ since we use level two walled Brauer-Clifford superalgebras in section~4.
After we have proved the freeness of cyclotomic walled Brauer-Clifford sueralgebras in section~6, we know that the following result is available over an integral domain $R$.
\begin{Theorem}  \label{alghomcyco}
For any $k\in \mathbb Z^{>0}$, there is a superalgebra homomorphism
$\phi_k: BC_{r,t}^{\rm  aff}\rightarrow  BC_{2, r+k,t+k}$ sending
$\omega_{2n+1}, \bar \omega_{2n+1}, e_1,  s_i, \bar s_j, c_l, \bar c_m, x_1, \bar x_1$ to
$\xi_{2n+1, k+1}, \bar \xi_{2n+1, k+1}, e_{k+1}, s_{k+i}, \bar s_{k+j},  c_{k+l}, \bar c_{k+m}, z_{k+1}$, $ \bar z_{k+1}$ respectively,
for all admissible  $i, j, l, m, n$.
\end{Theorem}

\begin{proof} It is enough to verify the images of generators of $BC_{r, t}^{\rm aff}$ satisfy the defining relations for $BC_{r, t}^{\rm aff}$ in Definition~\ref{awbsa}.
If so, $\phi_k$ is an algebra homomorphism. Since $\phi_k$ sends even (resp., odd) generators to even (resp., odd) elements in  $BC_{2, r+k,t+k}$, $\phi_k$ is a superalgebra homomorphism.

 By Corollary~\ref{wsba-1}, $\phi_k$ satisfies
 \eqref{symm}--\eqref{aseradual}  and Definition~\ref{wsera}(1)--(10).  By Lemma~\ref{usefcycl}(1),(3),(9),(10), $\phi_k$ satisfies  \eqref{asera} and \eqref{aseradual2}. Applying anti-involution $\sigma$ on Lemma~\ref{usef}(7), we see that $\phi_k$ satisfies
the Definition~\ref{awbsa}(1). By  Lemma~\ref{usefcycl}(8), (6) (resp., (11)--(13)),  $\phi_k$  satisfies  Definition~\ref{awbsa}(2)--(4) (resp., (5)--(7)).
 Finally,  $\phi_k$ satisfies Definition~\ref{awbsa}(8)--(11) by
  Lemma~\ref{usefcycl}(2), (5).
\end{proof}

The following result is a counterpart of \cite[Theorem~4.14]{RSu}.

\begin{Theorem} \label{main1cyco}  Suppose  $R$ is a domain which contains $2^{-1}$ and $  \omega_1$. Then
 $ BC_{r,t}^{\rm aff}$ is free over $R$ spanned by
all regular monomials in Definition~$\ref{regm}$. In particular,  $BC_{r,t}^{\rm aff}$ is  of infinite rank.
\end{Theorem}
\begin{proof} By Proposition~\ref{spanned},  it is enough to prove that $M$, the set of all regular monomials of
$BC_{r, t}^{\rm aff},$ is linear independent over  $\mathbb Z[\omega_1, 2^{-1}]$, where
$\omega_1$ is an indeterminate. By fundamental theorem on algebras, it suffices to prove it for sufficient many $\omega_1$'s.
This can be done by choosing $\omega_1$ as in \eqref{omega-i}. So, it is enough to prove that $M$ is linear independent over  $\mathbb C$ for infinite many $\omega_1$'s
in \eqref{omega-i}.

 By Lemma~\ref{hecrel}(1)--(3) and Definition~\ref{regm}, we  assume that a regular monomial $\mathbf{m}$  of $BC_{r, t}^{\rm aff} $  is of form
 \begin{equation}\label{regmm} \mathbf{m}= c^{\alpha} x^{\beta} d_1^{-1} e^f wd_2   {\bar x}^{\gamma}  {\bar c}^{\delta}
\mbox{$\prod\limits_{n\in \mathbb Z^{>0}}$} \omega_{2n+1}^{a_{2n+1}},\end{equation} where $(\alpha, \delta)\in \mathbb Z_2^r\times \mathbb Z_2^t$  and $(\beta, \gamma)\in \mathbb N^r\times \mathbb N^t$ and  $d_1, d_2\in D_{r, t}^f$ and $w\in \Sigma_{r-f}\times \Sigma_{\overline{t-f}}$ and $0\le f\le \min\{r, t\}$.
So, it is equivalent to prove that the above regular monomials are linear independent. If it were false, then
 there is a finite subset $S\subset M$
such that  $\sum_{\mathbf{ m}\in  S} r_{\textbf{m}} {\textbf {m}}=0$  and  $r_{\mathbf {m}}\neq 0$ for all $\textbf {m}\in S$.
For each $\mathcal S$, we set
\begin{equation}\label{k-kmcyco}\begin{aligned} & \tilde k =\max\big\{|\beta|+\SUM{n}{} (2n+1) a_{2n+1}\mid
 c^{\alpha} x^{\beta} d_1^{-1} e^f wd_2   {\bar x}^{\gamma}  {\bar c}^{\delta}
\mbox{$\prod\limits_{n\in \mathbb Z^{>0}}$} \omega_{2n+1}^{a_{2n+1}}\in \mathcal S \big\},\\
&
\hat k=\max\big\{|\gamma|+\SUM{n}{} (2n+1) a_{2n+1}\mid  c^{\alpha} x^{\beta} d_1^{-1} e^f wd_2   {\bar x}^{\gamma}  {\bar c}^{\delta}
\mbox{$\prod\limits_{n\in \mathbb Z^{>0}}$} \omega_{2n+1}^{a_{2n+1}}\in \mathcal S \big\}.
\\ \end{aligned}
\end{equation}
If $\tilde k\geq \hat k$, we define $k=\tilde  k$ and \begin{equation}\label{k11} \begin{aligned}
& f_0=\min\big\{f\mid c^{\alpha} x^{\beta} d_1^{-1} e^f wd_2   {\bar x}^{\gamma}  {\bar c}^{\delta} \mbox{$\prod\limits_{n\in \mathbb Z^{>0}}$} \omega_{2n+1}^{a_{2n+1}}\in S, |\beta|+\SUM{n}{} (2n+1) a_{2n+1}=k\,\big\},\\
& k_1= \max\big\{|\gamma|\mid c^{\alpha} x^{\beta} d_1^{-1} e^{f_0} wd_2   {\bar x}^{\gamma}  {\bar c}^{\delta} \mbox{$\prod\limits_{n\in \mathbb Z^{>0}}$} \omega_{2n+1}^{a_{2n+1}}\in S, |\beta|+\SUM{n}{} (2n+1) a_{2n+1}= k\,\big\}.\\
\end{aligned}
 \end{equation}
If $\tilde k<\hat k$, we define $k=\hat k$ and
 \begin{equation}\label{k12}\begin{aligned}& f_0=\min\big\{f\mid c^{\alpha} x^{\beta} d_1^{-1} e^f wd_2   {\bar x}^{\gamma}  {\bar c}^{\delta} \mbox{$\prod\limits_{n\in \mathbb Z^{>0}}$} \omega_{2n+1}^{a_{2n+1}}\in S, |\gamma|+\SUM{n}{} (2n+1) a_{2n+1}=k\,\big\}, \\ & k_1= \max\big\{|\beta|\mid  c^{\alpha} x^{\beta} d_1^{-1} e^{f_0} wd_2   {\bar x}^{\gamma}  {\bar c}^{\delta} \mbox{$\prod\limits_{n\in \mathbb Z^{>0}}$} \omega_{2n+1}^{a_{2n+1}}\in S, |\gamma|+\SUM{n}{} (2n+1) a_{2n+1}= k\big\}.\\
\end{aligned}
  \end{equation}

Let  $\phi_k{\ssc\!}:\! BC_{r,t}^{\text{\rm aff}}{\ssc\!}\!\rightarrow {\ssc\!}\! BC_{2,r+k,t+k}(\omega_1)$ be the superalgebra homomorphism in Theorem~\ref{alghomcyco}.
Since $x_i =x_i'-L_i $ and $\bar x_i=\bar x_i'-\bar L_i$,
\begin{equation}\label{ximage} \phi_k(x_i)=\mbox{$\sum\limits_{j=1}^k$} (e_{k+i, j}+\bar e_{k+i, j})+x'_{k+i}-L_{k+i},\  \text{and}\ \phi_k(\bar x_i)=\mbox{$\sum\limits_{j=1}^k$}( e_{j, k+i}-\bar e_{j, k+i})+\bar x'_{k+i} -\bar L_{k+i}.\end{equation}
Using Lemma~\ref{omedcycl} to express $\xi_{2n+1,k+1}$ for $n\in \mathbb Z^{\ge 1}$, we see that
  some terms of $\phi_k(\mathbf m)$ are of forms
(we will see in the next paragraph
that other terms of $\phi_k(\mathbf m)$ will not contribute to our computations)
\begin{equation} \label{prod1cycl}  \mbox{$\prod_{i=1}^r { c}_{i+k}^{\alpha_i}\prod\limits_{i=1}^r$} (k\! +\!i, i_{1}) \cdots (k\!+\!i, i_{\beta_i}) \phi_k (d_1^{-1} e^f w d_2)
 \mbox{$\prod\limits_{j=1}^t$} (\overline{k\!+\!j}, \bar j_1) \cdots (\overline{k\!+\!j}, \bar j_{\gamma_j})\mbox{$\prod\limits_{j=1}^t$} { \bar c}_{j+k}^{\delta_j}  \mbox{$\prod\limits_{n\ge 1}$} \mathbf c_{2n+1},
\end{equation}
where $\mathbf c_{2n+1}$, which comes from $\xi_{2n+1, k+1}$,  ranges over products of  $a_{2n+1}$ disjoint cycles  in $\Sigma_k$ (or $\bar {\Sigma}_k)$ such that each cycle is of length $2n+1$.

By Theorem~\ref{theo-2222}, $BC_{r, t}$ is a subalgebra of $BC_{2, r, t}$ and so is the walled Brauer algebra, say $B_{r, t}(0)$ which is isomorphic to the subalgebra generated by
$e_1, s_1, \ldots, s_{r-1}$ and $\bar s_1, \ldots, \bar s_{t-1}$. Similarly, we have the subalgebra $B_{r+k, t+k}(0)$ of $BC_{2, r+k, t+k}$.
It is known that $B_{r, t}(0)$ can be defined by so called $(r, t)$-walled Brauer diagrams.
 Each of them  is a   diagram with $(r\!+\!t)$ vertices on the top and bottom rows, and vertices on both rows are labeled  from left to right
by $r, \ldots,2, 1, \bar 1, \bar 2, \ldots, \bar t$.
Every vertex $i\!\in\!\{1, 2, \ldots, r\}$ (resp., $\bar i\!\in $ $\{\bar 1, \bar 2, \ldots, \bar t\}$) on each
row must be connected to a unique vertex $\bar j$ (resp., $j$) on the same row or a unique vertex
$j$ (resp., $\bar j$) on the other row.
The pairs $[i, j]$ and $[\bar i, \bar j]$ are called
{\it vertical edges}, and the pairs    $[\bar i, j]$ and $[i, \bar j]$ are called
{\it horizontal edges}.
By definition, a   $\phi_k(d_1^{-1} e^f w d_2)$ in \eqref{prod1cycl} corresponds to a unique $(r+k, t+k)$-walled Brauer diagram such that
$[i, i]$ and $[\bar j, \bar j]$ are its  vertical edges for all $1\le i, j\le k$ (see e.g. \cite{RSu}).
We call the terms of the form (\ref{prod1cycl}) the {\it leading terms} if
 \begin{itemize}\item[(i)] $k=|\beta|+\sum_n (2n+1) a_{2n+1}$
 if $\tilde k\geq \hat k$ and  $k=|\gamma|+\sum_n(2n+1) a_{2n+1}$
  if $\tilde k<\hat k$. (cf.~\eqref{k-kmcyco}),
 \item[(ii)] the corresponding  $f$ in (\ref{prod1cycl}) is $f_0$ in \eqref{k12},
 \item [(iii)] $|\gamma|=k_1$ if $\tilde k\geq \hat k$ and $|\beta|=k_1$ if $\tilde k<\hat k$,
\item[(iv)]
in the first case of (i), the juxtapositions of the sequences $i_1, i_2, \ldots, i_{\beta_i}$ for $1\!\le \!i\!\le\! r$ and $\mathbf c_{2n+1}$, $n\ge 1$ run through all  permutations of the sequences in $1, 2, \ldots, k$ and  the sequences $\bar j_1, \bar j_2, \ldots,\bar j_{\gamma_j}$, $1\le j\le t$ run through all permutations of the sequence
$\bar 1, \bar 2, \ldots, \bar k_1$; while in the second case of (i), the  juxtapositions of the sequences $\bar j_1, \bar j_2, \ldots, \bar j_{\gamma_j}$ for $1\le j\le t$ and $\mathbf c_{2n+1}$, $n\ge 1$ run through all permutations of the sequences in $\bar 1, \bar 2, \ldots, \bar k$ and the sequence of $i_1, i_2, \ldots, i_{\beta_i}$, $1\le i\le r$ run through all permutations of sequence $1, 2, \ldots, k_1$.
 \end{itemize}

By Theorem~\ref{monom}, all $\tilde {\mathbf{m}}=c^{\tilde\alpha} x'^{\tilde \beta} \tilde d_1^{-1} e^{\tilde f} \tilde w\tilde d_2   {\bar x'}^{\tilde \gamma}  {\bar c}^{\tilde \delta} \in BC_{2, r+k, t+k}$ consist of a basis of $BC_{2, r+k, t+k}$ over $\mathbb C$, where
$\tilde \alpha, \tilde \beta\in \mathbb Z_2^{r+k}$,  $\tilde \gamma, \tilde \delta\in \mathbb Z_2^{t+k}$,  and  $\tilde d_1, \tilde d_2\in D_{r+k, t+k}^{\tilde f}$,  $\tilde w\in \Sigma_{r+k-\tilde f}\times \Sigma_{\overline{t+k-\tilde f}}$ and $0\le \tilde f\le \min\{r+k, t+k\}$.
Such monomials will be called normal monomials.
Moreover, $\tilde{\mathbf{m}}$ is called an admissible  monomial if \begin{enumerate} \item [(a)]
$\tilde \alpha_i=\tilde \delta_j=0$ for all $1\le i\le k$, $1\le j\le k_1$  if $\tilde k\ge \hat k$;  or $\tilde \alpha_i=\tilde \delta_j=0$ for all $1\le i\le k_1$ and $1\le j\le k$ if  $\tilde k<\hat k$,
 \item [(b)]
 the corresponding walled Brauer diagram of $\tilde d_1^{-1} e^{\tilde f } \tilde w \tilde d_2$ satisfies (1)--(5) as follows:
\begin{enumerate}
 \item[(1)] $\tilde f= f_0$,
 \item [(2)] no vertical edge of form $[i, i]$ and $[\bar j, \bar j]$, $1\le i\le k$, $1\le j\le k_1$ if $\tilde k \geq \hat k$,
 \item [(3)] no vertical edge of form $[i, i]$ and $[\bar j, \bar j]$, $1\le i\le k_1$, $1\le j\le k$ if $\tilde k<\hat k$,
 \item [(4)] no horizontal edge of form $[i, \bar j]$, $1\!\le\! i\! \le\! k$, at the bottom  row if $\tilde k \geq \hat k$,
 \item [(5)] no horizontal edge of form $[i, \bar j]$, $1\!\le\! j\! \le\! k$, at the top row  if $\tilde k<\hat k$,
 \end{enumerate}
 \item[(c)]  $\tilde \beta_i=\tilde \gamma_j=0$, for all  $1\le i\le r+k$ and $1\le j\le t+k$.
    \end{enumerate}
\noindent
In the following, we assume that $\tilde k<\hat k$ (the case $\tilde k\ge \hat k$ can be dealt with in a similar way).   A $\phi_k(\mathbf{m})$ contributes admissible monomials of $BC_{2, k+r, k+t}$ only when
    $\mathbf{m}\in \mathcal S$ is given in \eqref{regmm}
    such that   $k=|\gamma|+\sum_{n} (2n+1) a_{2n+1}$, $ f=f_0$ and   $k_1=|\beta|$. More explicitly, the  leading terms exactly appear in $\phi_k(\mathbf m)$ which are  admissible monomials of $BC_{2, k+r, k+t} $.
We claim that other terms in $\phi_k( \sum_{\mathbf  m\in \mathcal S} r_{\mathbf {m}} {\mathbf  m})$
  are obtained   from \eqref{prod1cycl}  by \begin{itemize}\item [(1)] using  the terms  $e_{k+i, i_j}, \bar e_{k+i,i_j}$ of $ \phi_k(x_i)$
to replace some $(k+i, i_{j})$,  \item [(2)]  using the terms  $e_{i, j}, \bar e_{i, j}$
 $1\le i, j\le k$  (resp.,  $\bar e_{j_i, k+j},  e_{j_i, k+j}$ of  $\phi_k(\bar x_j)$)  to replace  ($\bar i, \bar j)$ in $\mathbf c_{2n+1}$ (resp.,  $(\overline{k+j}, \bar j_i)$) if (1) does not occur,
\item[(3)] using the term $x_{k+i}'$ of  $\phi_k(x_i)$ to replace $(k+i, i_j)$ or
using the term $\bar x_{k+i}'$ of  $\phi_k(\bar x_i)$ to replace $(\overline{k+i}, \bar i_j)$, or using either $x_j'$ or $\bar x_j'$ to replace some $(\bar i, \bar j)$ in $\mathbf c_{2n+1}$, provided that neither (1) nor   (2)  occurs,
\item [(4)] using some $(\overline{k+j}, \bar s)$ or   $\bar c_s (\overline{k+j}, \bar s)\bar c_s$, $s>k$ to replace $(\overline{k+j}, \bar j_i)$; or using $(i, j), c_j(i, j)c_j$ to replace $(\bar i, \bar j)$ in $\mathbf c_{2n+1}$; or using $(k+i, s)$, $c_s(k+i, s) c_s$, $s>k_1$, to replace $(k+i, i_j)$, provided that (1)--(3) do not occur,
\item [(5)] using $\bar c_{j_i} (\overline{k+j}, \bar j_i) \bar c_{j_i}$ to replace $(\overline{k+j}, \bar j_i)$, or using $\bar c_{j} (\bar i, \bar j)\bar c_j$ to replace $(\bar i, \bar j)$ in $\mathbf c_{2n+1}$ or using $c_{i_j} (k+i, i_j) c_{i_j}$ to replace  $(k+i, i_j)$, provided that (1)-(4) do not occur.
      \end{itemize}
\noindent
In the case (1), we use defining relations for $BC_{2, r+k, t+k}$ to rewrite the corresponding monomial as a linear combinations of normal  monomials. Each of these normal  monomials  corresponds to a unique walled Brauer diagram, say $D$,  in which there is a horizontal edge  $[i, \bar j]$ at the top row of $D$ such that $1\le j\le k$.  Such a monomial does not satisfy (b)(5).  Similarly, in case (2) (resp.,
  (resp.,  (3),  (4),  (5)), the corresponding monomials of $BC_{2, k+r, k+t}$ can be written as a linear combinations of normal  monomials which do not satisfy  (b)(1) or (b)(3) (resp., (c) or (b)(3),  (b)(3),   (a)).
This verifies our claim.

We assume that $\mathbf{m}_1, \mathbf m_2, \ldots, \mathbf m_p $ are all monomials in $S$ which contribute leading terms. Write
\begin{equation}\label{mii} \mathbf m_i=c^{\alpha(\mathbf m_i)} x^{\beta(\mathbf m_i)} d_1(\mathbf m_i)^{-1} e^{f(\mathbf m_i)} w(\mathbf(m_i)d_2(\mathbf m_i)   {\bar x}^{\gamma(\mathbf m_i)}  {\bar c}^{\delta(\mathbf m_i)}
\mbox{$\prod\limits_{n\in \mathbb Z^{>0}}$} \omega_{2n+1}^{a_{2n+1}(\mathbf m_i)}.\end{equation}
Then $k=|\gamma(\mathbf m_i)|+\sum_n (2n+1) a_{2n+1}(\mathbf m_i)$, $f(\mathbf m_i)=f_0$ and $k_1=|\beta(\mathbf m_i)|$. Let $A_i$ be the set of all leading terms
contributed  by $\phi_k(\mathbf m_i)$. These leading terms are admissible monomials of  $BC_{2, k+r, k+t}$. We have proved that other terms of $\sum_{\mathbf m\in S} r_{\mathbf m} \phi_k(\mathbf m)$ will not contribute admissible monomials of $BC_{2, k+r, k+t}$. So,
\begin{equation} \label{setneq}\mbox{$\sum\limits_{i=1}^p$}  \tilde r_{\mathbf m_i} \mbox{$\sum\limits_{\mathbf n\in A_i}$} \mathbf n=0, \end{equation}
where $  \tilde r_{\mathbf m_i}$ is the $a r_{\mathbf m}$ and $a$ is a power of $\pm 2$, which comes from the coefficients of leading terms of $\xi_{2n+1, k+1}$.
In the following, we explain that \eqref{setneq} does not hold. If so, then $\mathcal S$ is linear independent and the result will follow.

Suppose $g_i=(k+i, 1)(k+i, 2)\cdots (k+i, k_1)$ and $\bar g_j=(\overline {k+j}, \bar 1)(\overline{k+j}, \bar 2)\cdots (\overline{k+j}, \bar k)$ for
$1\le i\le r$ and $1\le j\le t$. Note that each element in  $A_s$ contains  factors $(k+i, i_1)\cdots (k+i, i_{\beta_i})$ contributed by $\phi_k(x_i)^{\beta_i}$,
such that $i_1, \cdots, i_{\beta_i}$, $1\le i\le r$,  is a permutation of some elements in $1, 2, \ldots, k_1$. So, $(k+i, k_1)\cdots (k+i, k_1-\beta_i+1)$ is one of such factors and hence
$g_i (k+i, k_1) (k+i, k_1-1)\cdots (k+i, k_1-\beta_i+1)$ fixes $k_1, \ldots, k_1-\beta_i+1$.
 So,  there is an element in  $g_i (A_s)$ whose walled Brauer diagram contains  vertical edges $[j, j]$, where $j$ ranges   $\beta(\mathbf m_s)_i$ numbers in  $\{1, 2,\ldots, k_1\}$.     If
  $g_i(A_{s'})$ contains an  element such that the corresponding walled Brauer diagram contains  vertical edges $[j, j]$ for $\beta(\mathbf m_s)_i$ numbers in $\{1, 2, \ldots, k_1\}$,
  then $\beta(\mathbf m_{s'})_i=\beta(\mathbf m_{s})_i$. So, we can assume $\beta(\mathbf m_1)=\beta(\mathbf m_j)$, $2\le j\le p$. Similarly, we use $\bar g_j$ instead of $g_i$ to obtain
   $\gamma(\mathbf m_1)=\gamma(\mathbf m_j)$, $2\le j\le p$.
Note that $\mathbf c_{2n+1}$ is the product of $a_{2n+1}$ disjoint cycles with length $2n+1$. So, different $\prod_{n} \omega_{2n+1}^{a_{2n+1}}$ gives product of disjoint cycles with different lengthes and thus, we can assume $a_{2n+1}(\mathbf m_i)$ is independent of $\mathbf m_i$. Since any leading term is of form in \eqref{prod1cycl}, by Theorem~\ref{level-2} or Theorem~\ref{wbhsa321},   we can also assume that $\alpha(\mathbf m_i), \delta(\mathbf m_i)$ are independent of $\mathbf m_i$. Since $c_i^2=-1$ and $\bar c_j^2=1$, we can assume $\alpha(\mathbf m_i)=0^r\in \Z_2^r$ and $\delta(\mathbf m_i)=0^t\in \Z_2^t$. Using \eqref{prod1cycl}, we see that there exists   a leading term in $A_i\cap A_j$ if and only if $\mathbf m_i=\mathbf m_j$. So, $\tilde r_{\mathbf m_i}=0$ for all $1\le i\le p$, a contradiction.
So, $S$ is linear independent over $\mathbb C$ and hence   over $\mathbb Z[2^{-1}, \omega_1]$. In general, using arguments on base change yields the result over   an arbitrary integral   domain $R$ containing ${\mathbf {\omega}}_1$ and $2^{-1}$.
\end{proof}

The following result follows from Theorem~\ref{main1cyco}, immediately.

\begin{Theorem} \label{main1cyco321}  Suppose  $R$ is a domain containing $2^{-1}$ and $\omega_{2n+1}$, for all $n\in \mathbb N$.  Then
 $ \widetilde{BC}_{r,t}$ is free over $R$ spanned by
all of its  regular monomials. In particular,  $\widetilde{BC}_{r,t}$ is  of infinite rank.
\end{Theorem}

\section{A basis of the  cyclotomic Brauer-Clifford   superalgebra}
In this section, we assume that $R$ is a domain  containing  $ 2^{-1}$ and parameters  $\{\omega_{2n+1}\in R\mid n\in \mathbb N\}$.  The  {\textit {affine walled Brauer-Clifford superalgebra}}
$ \widetilde {BC}_{r, t}$ with respect to the defining parameters $\omega_{2n+1}$'s  can be also defined in a simpler way  as follows.
As a free $R$-superspace,
\begin{equation}\label{Affbcycl}{ \widetilde{BC}_{r, t} =R[\mathbf {x}_r]\otimes BC_{r, t}\otimes R[\bar {\mathbf {x}}_t],}
 \end{equation}the  tensor product of the walled Brauer-Clifford superalgebra $ BC_{r, t}$ with
 two polynomial algebras $R[\mathbf{x}_r]:=R[x_1, x_2, \ldots, x_r]$   and  $R[\bar{ \mathbf{x}}_t]:=R[\bar x_1, \bar x_2, \ldots, \bar x_t]$.
The multiplication of $\widetilde{BC}_{r, t}$ is defined such that
   $R[\mathbf{x}_r]\otimes 1\otimes  1$, $1\otimes 1\otimes R[\bar{ \mathbf{x}}_t]$,
 $1\otimes BC_{r, t}\otimes 1$, $R[\mathbf{x}_r]\otimes  HC_r\otimes 1$ and
$1\otimes {\bar{HC}}_{t}\otimes R[\bar{ \mathbf{x}}_t]$
are subalgebras isomorphic to
$R[\mathbf{x}_r]$, $R[\bar{ \mathbf{x}}_t]$, $BC_{r, t}$, $HC_r^{\rm{aff}}$, and
$\overline{HC}{\ssc\,}_t^{\rm{aff}}$ respectively,
 and (for simplicity, without confusion we identify elements
 $x_i\otimes1\otimes1$, $1\otimes s_i\otimes 1$, $1\otimes e_i\otimes 1$, $1\otimes \bar s_i\otimes 1$, $1\otimes  c_i\otimes 1$,$1\otimes \bar c_i\otimes 1$,$1\otimes1\otimes\bar x_i$ in \eqref{Affbcycl} with
 $x_i,s_i,e_i,\bar s_i,c_i,\bar c_i,\bar x_i$)
\begin{eqnarray}
\label{it1}&\!\!\!\!\!\!\!\!\!\!\!\!\!\!\!\!&
e_1(x_1+\bar x_1)=(x_1+\bar x_1) e_1=0,\ \ e_1 s_1x_1s_1=s_1x_1s_1e_1,\ \  e_1 \bar s_1\bar x_1\bar s_1=\bar s_1\bar x_1\bar s_1e_1,\\
\label{it2}&\!\!\!\!\!\!\!\!\!\!\!\!\!\!\!\!&
s_i\bar x_1=\bar x_1 s_i,\ \ \bar s_i x_1= x_1 \bar s_i,\   i>1, \  x_1 (e_1-\bar e_1+\bar x_1)= (e_1-\bar e_1+\bar x_1)x_1,\\
\label{it3}&\!\!\!\!\!\!\!\!\!\!\!\!\!\!\!\!&
e_1x_1^ke_1=\omega_k e_1,\ \  e_1\bar x_1^ke_1=\bar \omega_k e_1,
%\label{it4}&\!\!\!\!\!\!\!\!\!\!\!\!\!\!\!\!&e_1c_1=e_1\bar c_1,\  \ e_1c_1e_1=0,
\end{eqnarray}
where $\bar \omega_{2k+1}$ determined by $\omega_{1}, \ldots, \omega_{2k+1}$ as in  Corollary~\ref{assump1}. Further, $\omega_{2n}=\bar\omega_{2n}=0$.

We hope to classify finite dimensional simple $\widetilde {BC}_{r, t}$-modules over an algebraically closed field $F$ with $\text{Char} F\neq 2$. This leads us to introduce cyclotomic Brauer-Clifford superalgebras as follows.  Let $f(x)$ be the minimal polynomial of $x_1$ with respect to a finite dimensional $\widetilde {BC}_{r, t}$-module $M$. Then \begin{equation}\label{minimalf} f(x)=x^k\mbox{$ \prod\limits_{i=1}^n$} (x-u_i), \end{equation}
where $u_1, \ldots, u_n$  are nonzero in $F$. Let $\langle f(x_1)\rangle $ be the two-sided ideal of  $\widetilde {BC}_{r, t}$ generated by $f(x_1)$. Since   $M$ is simple,   $\langle f(x_1)\rangle\neq \widetilde {BC}_{r, t} $.

\begin{Lemma}\label{inducdeg} We have $c_1f(x_1)= \epsilon f(x_1)c_1$, where $\epsilon\in \{-1, 1\}$ and  $f(x)$ is given in \eqref{minimalf}.
\end{Lemma}
\begin{proof} We prove the result by induction on $\text{deg\,} f(x)$. If $\text{deg\,} f(x)=1$, then $f(x_1)=x_1-u$.
 When $u\neq 0$, we have  $\langle x_1-u\rangle=\widetilde{BC}_{r, t}$, a contradiction, since  $2u=c_1 (x_1-u) c_1-(x_1-u) \in \langle x_1-u\rangle$. So, $f(x_1)=x_1$ and
 $c_1f(x_1)=-f(x_1)c_1$.
In general, if $ f(x_1) c_1\in \{ c_1 f(x_1), -c_1 f(x_1)\}$, there is nothing to prove. Otherwise, by Definition~\ref{awbsa}(2), $f(-x_1)=-c_1f(x_1)c_1$ and   $ f(-x_1)\neq f(x_1)$. In this case,
we choose $h(x_1)\in \{f(-x_1)+f(x_1),  f(-x_1)-f(x_1)\}$ such that $\text{deg\,} h(x)< \text{deg\,} f(x)$.
 Define   $d(x)= {\rm g.c.d\,} (h(x), f(x))$, the greatest common divisor of $h(x)$ and $f(x)$. Then $ \langle d(x_1)\rangle=  \langle f(x_1)\rangle$ and hence the irreducible module $M$ is killed by $d(x_1)$. This is a contradiction since $f(x_1)$ is the minimal polynomial of $x_1$ with respect to $M$.
 \end{proof}

  Recall that $f(x_1)$ in \eqref{minimalf}. Since $c_1 f(x_1)c_1=-\epsilon f(x_1)$, and $c_1 f(x_1) c_1=\pm x_1^k \prod_{i=1}^n (x_1+u_i)$,
$f(x)$ is the  minimal polynomial of $x_1$ with respect to $M$ if and only if $x_1^k \prod_{i=1}^n (x_1+u_i) $ is the  minimal polynomial of $x_1$ with respect to $M$. In other words, $u_i$ and $-u_i$ appear simultaneously if $u_i\neq 0$. Thus, we can assume
\begin{equation}\label{polyoff}
f(x_1)= x_1^k\mbox{$\prod\limits  _{i=1}^m$} (x_1^2-u_i^2),
\end{equation}
 where $0\neq  u_i\in F$, $1\le i\le m$. Moreover, by
   Lemmas~\ref{assum}--\ref{assum1}, there is a monic polynomial $g(\bar x_1)$ with degree $l=k+2m$ such that
\begin{equation}\label{feg}
e_1f(x_1)=(-1)^k e_1g(\bar x_1).
\end{equation}

\begin{Lemma}\label{gcg} Let $g(\bar x_1)$ be given such that \eqref{feg} is satisfied.  Then
$\bar c_1 g(\bar x_1)=\epsilon g(\bar x_1)\bar c_1$.
\end{Lemma}
\begin{proof} Since $l=k+2m$,  we have
$$\begin{aligned} (-1)^k e_1 g(\bar x_1) \bar c_1 & =e_1f(x_1)\bar c_1= e_1\bar c_1 f(x_1)=e_1 c_1 f(x_1)=\epsilon e_1f(x_1) c_1\\ & =(-1)^k \epsilon e_1  g(\bar x_1)c_1=(-1)^k \epsilon e_1 c_1 g(\bar x_1)=(-1)^k \epsilon e_1\bar c_1  g(\bar x_1).\end{aligned}$$
By Theorem~\ref{main1cyco321}, $\bar c_1 g(\bar x_1)=\epsilon g(\bar x_1)\bar c_1$.
  \end{proof}

In the remaining part of this paper, we assume that \begin{equation}\label{polyofg}
g(\bar x_1)= \bar x_1^{k_1}\mbox{$\prod \limits_{j=1}^{m_1}$}(\bar x_1^2-\bar u_j^2),
\end{equation}  such that  $k_1+2m_1=k+2m$ and $0\neq \bar u_j\in F$, $1\le j\le m_1$. This is reasonable by Lemma~\ref{gcg}. Since the finite dimensional simple $\widetilde{BC}_{r, t}$-module $M$ is killed by $f(x_1)$, by \eqref{feg},  it is killed by  $e_1g(\bar x_1)$. We want to consider simple  $\widetilde{BC}_{r, t}$-modules $M$ such that $e_1$ acts on $M$ nontrivially, it is necessary  to assume that $M$ is killed by $g(\bar x_1)$. That is why we introduce cyclotomic walled Brauer-Clifford superalgebras as in the Definition~\ref{cwbcsa1}.

From here to the end of this section, we assume both $\widetilde{BC}_{r, t}$ and $BC_{l, r, t}$ are defined over a domain  $R$ containing $ 2^{-1}$ and parameters $\omega_{2n+1}$ for all $n\in \mathbb N$.

\begin{Lemma}\label{condi-k1cycl} Write $f(x_1)=x_1^{k+2m}+\sum_{i=1}^{2m} a_i x_1^{k+2m-i}$, where $f(x_1)$ is  given in \eqref{polyoff}.  Then  $e_1$ is an $R$-torsion
element of  $  BC_{k+2m, r, t}$  unless \begin{equation}\label{ccccc}
{ \omega_\ell=-(a_1\omega_{\ell-1}+\ldots +a_{2m}\omega_{\ell-2m})}\mbox{ \ for all \ }\ell\ge k+2m.
 \end{equation}
\end{Lemma}
\begin{proof} Let  $b_\ell=\omega_\ell+a_1\omega_{\ell-1}+\ldots+ a_{2m}\omega_{\ell-2m}\in R$. By \eqref{it3}, $b_\ell e_1=e_1 f(x_1)x_1^{\ell-2m-k}e_1$ in $ \widetilde{BC}_{r, t}$ and $b_\ell e_1=0$  in  $BC_{k+2m,r,t}$.
Thus, $e_1$ is  an $R$-torsion element if  $b_\ell\neq 0$ for some $\ell\ge k+2m$.
\end{proof}

\begin{Defn}\label{condi-kcycl}  The superalgebras $ \widetilde{BC}_{r, t}$ and   $BC_{k+2m,r,t}$ are called  \textit{admissible} (with respect to $f(x_1)$)  if \eqref{ccccc} holds.
\end{Defn}

\begin{Lemma}\label{zero-1cycl} For $1\leq i\leq r$, $1\leq j\leq t$, define   $f_i= f(x_i')$ and
$g_{i}= g (\bar x_i')$,  where   $f(x_1)$ and $g(\bar x_1)$ satisfy   \eqref{polyoff}--\eqref{polyofg}.
Then the following equations hold for all admissible  $i, j$:  \begin{multicols}{2}\begin{enumerate}
 \item[(1)]$c_jf_i=\epsilon f_ic_j$,
 \item[(2)] $\bar c_jg_i=\epsilon g_i\bar c_j$,
     \item[(3)]$\bar c_jf_i=f_i\bar c_j$,
     \item[(4)] $ c_jg_i= g_i c_j$,
     \item[(5)]$\bar s_jf_i=f_i\bar s_j$,
     \item[(6)] $ s_jg_i= g_i s_j$,
     \item[(7)]$f_if_j=f_jf_i $ in ${\rm gr} (\widetilde{BC} _{r,t})$,
     \item[(8)] $g_ig_j=g_jg_i$ in ${\rm gr} (\widetilde{BC}_{r,t})$.
     \end{enumerate} \end{multicols}
 \end{Lemma}
\begin{proof} These equations can be easily verified  by using Lemmas~\ref{inducdeg}--\ref{gcg} and Definition~\ref{awbsa}. \end{proof}

 Note that  the affine Hecke-Clifford superalgebra $HC_r^{\rm aff}$ (resp., $\overline { HC}{}{\ssc\,}_t^{\rm aff}$)  is  isomorphic to the sub-superalgebra of $\widetilde{BC}_{r, t}$ generated by $x_1, s_1,\ldots, s_{r-1}$ and $c_1$ (resp., $\bar x_1, \bar s_1,\ldots, \bar s_{t-1}$ and $\bar c_1$).
 \begin{Lemma}\label{zero-12}  For any  $a\in \mathbb Z^{>0}$, we have
  \begin{itemize} \item [(1)] $(x_i')^a f(x'_\ell)-f(x'_\ell)x'^a_i\in  \sum_{b<a}\sum_{h, h_1=1}^{\max\{i, \ell\}} f(x'_h)(x'_{h_1})^b HC_r$,
  \item [(2)]  $g(\bar x_\ell') (\bar x_i')^a-(\bar x_i')^a g(\bar x_\ell') \in  \sum_{b<a}\sum_{h, h_1=1}^{\max\{i, \ell\}}\overline{HC}_t (\bar x'_{h_1})^b g(\bar x'_h) $.
  \end{itemize}
 \end{Lemma}
\begin{proof} We have
  $x_1x_2=x_2x_1$, where  $x_2=x_2'-s_1-c_2s_1c_2$ (see  \eqref{asera}).
  By Lemma~\ref{zero-1cycl}(1), \begin{equation}\label{key-1}
x_2' f(x_1)=f(x_1)(x_2'-s_1-c_2s_1c_2)+f(x_2') s_1+\epsilon f(x_2')c_2s_1c_2. \end{equation}
Considering ${s_{i, 2} x_2' f(x_1) s_{2, i}}$ yields the result  when  $a=1$ and $\ell=1$.
If $\ell>1$, then
   $$x_i' f(x_\ell')=x_i's_{\ell-1} f(x_{\ell-1}') s_{\ell-1}= s_{\ell-1}x_{(i)s_{\ell-1}}' f(x_{\ell-1}') s_{\ell-1}.$$ So, the result follows from inductive assumption on $\ell-1$. This is (1)  when $a=1$. The general case follows from arguments on  induction on $a$. Finally, (2) can be verified, similarly.
\end{proof}

\def\bii{{\vec{i}}}
\def\bij{{\vec{j}}}

 \begin{Prop}~\label{2ijcycl} Define  $J_L=\sum_{i=1}^t \widetilde{BC}_{r, t}\, g_i$ and   $J_R=\sum_{i=1}^r f_i\, \widetilde{BC}_{r, t}$. Let $I$ be the two-sided ideal of $\widetilde{BC}_{r, t}$  generated by
$f(x_1)$ and $ g(\bar x_1)$.  We have
  \begin{enumerate}\item [(1)] $J_R$ is a left $HC_r^{\rm aff}\otimes \overline{HC}_t$-module,
  \item [(2)] $J_L$ is a right $HC_r\otimes \overline{HC}{\ssc\,}_{ t}^{\rm aff}$-module,
  \item [(3)]    $I=J_L+J_R$  if  $\widetilde{BC}_{r, t}$ is admissible.
\end{enumerate}
\end{Prop}

 \begin{proof} (1)--(2)  It is easy to see that  $J_R$ is stable under the left action of $HC_r\otimes \overline{HC}_t$. By Lemma~\ref{zero-12}(1), it is stable under the left action of ${HC}_r^{\rm aff}$. One can check (2) via Lemma~\ref{zero-12}(2), similarly.

 (3) Obviously,
 $J_L+J_R\subseteq I$. So, (3) follows if we can   prove $I\subseteq J_L+J_R$.
 Since $f(x_1), g(\bar x_1)\in J_L+J_R$, it suffices to verify that $ J_L+J_R$ is a two-sided ideal of  $\widetilde{BC}_{r, t}$.
We claim \begin{equation} \label{stepp1} h J_R\subset J_L+J_R,\end{equation}
for any generator $h$  of  $\widetilde{BC}_{r, t}$.  If so,   $h (J_L+J_R)\subset J_L+J_R$ and hence $J_L+J_R$ is a left ideal.

 In fact, by (1), it is enough to verify \eqref{stepp1} when $h\in \{\bar x_1, e_1\}$. If we have $e_1 J_R\subset J_L+J_R$, then $c_1e_1c_1 J_R\subset J_L+J_R$.
 Since $(\bar x_1+e_1-\bar e_1) f(x_1)=f(x_1) (\bar x_1+e_1-\bar e_1)\in J_R$, we have   $\bar x_1 f(x_1)\in J_L+J_R$.   Multiplying $(1, i)$ on both sides of $\bar x_1 f(x_1)$
 yields $\bar x_1 f_i\in J_L+J_R$. So, we need to verify $e_1 J_R\subset J_L+J_R$.
 By \eqref{it1},
 $e_1 f_i =f_i e_1$ for $i\ge 2$. So, $e_1 J_L\subset  J_L+J_R$ if $e_1f(x_1) \widetilde{BC}_{r, t}\subset J_L+J_R$.  This will be verified by checking
 \begin{equation}\label{stepp2}
 e_1 f(x_1)\textbf{m}\in J_L+J_R,\end{equation} for each regular monomial $\textbf{m}$ of $\widetilde{BC}_{r, t}$ in Definition~\ref{cregm}. Using arguments on graded structure of $\widetilde{BC}_{r, t}$,  we can write
 $\textbf{m}=c^{\alpha} x^{\beta} e_{i_1, j_1} \cdots  e_{i_f, j_f} w \bar x^{\gamma} {\bar c}^{\delta}$
 for some $(\alpha, \delta)\in \Z_2^r\times\Z_2^t$ and $(\beta, \gamma)\in \mathbb N^r\times \mathbb N^t$, and $w\in \Sigma_r\times \Sigma_{\bar t}$ and $1\le i_1, \ldots, i_f\le r$ and
 $1\le j_1, \ldots, j_f\le t$ such that $\{i_k, j_k\}\cap \{i_l, j_l\}=\emptyset$ if $k\neq l$.
 In the following, we write $e_{\bii, \bij}=e_{i_1, j_1} \cdots  e_{i_f, j_f}$.
  We prove \eqref{stepp2} by induction on $|\b|$.

\textit{Case~1: $|\b|=0$.}

If $f=0$, then $e_1f(x_1)c^\alpha w \bar x^\gamma \bar c^{\delta}=(-1)^k e_1g(\bar x_1)c^\alpha w \bar x^\gamma \bar c^{\delta}\subseteq J_L $. The last inclusion follows from (2). Suppose   $1\le f\le \min\{r, t\} $. Since $\widetilde {BC}_{r,t} $ is admissible, $e_1 f(x_1)e_1=0$. On the other hand,  we have  $e_1x_1^k c_1 e_1=0$ for all $k$. So,
 $e_1f(x_1)  \textbf{m}=0$ if $e_1$ is a factor
of  $e_{\bii,\bij}$.  If $e_1$ is not a factor of $e_{\bii,\bij}$, there are three cases we need to discuss.
\begin{itemize} \item   If $e_{ p, 1}$ is a factor of  $e_{\bii,\bij}$, and $p\neq 1$, then we assume that $i_1=p$ and $j_1=1$  since any two factors of $e_{\bii, \bij}$ commute each other. We have
$  e_1 f(x_1)c^{\alpha}  e_{ p, 1}=\prod_{i=2}^r c_i^{\alpha_i} e_1 f(x_1)e_{ p, 1} c_1^{\alpha_1} $. Since $$ e_1 f(x_1) e_{p, 1}=s_{p, 2} e_1 f(x_1) s_1 e_1 s_{1, p}=s_{p, 2} e_1s_1f(x_2') e_1 s_{1, p}=s_{p, 2} f(x_2')e_1 s_{1, p}\in J_R,$$ we have  $  e_1 f(x_1)c^{\alpha}  e_{1, p}\in J_R$ by
 (1). So,  $e_1f(x_1)\textbf{m}\in J_R$.

\item  If $e_{1, p}$ is a factor of $e_{\bii, \bij}$ and $p\neq 1$, we assume $i_1=1$ and $j_1=p$.  We have    $$\begin{array}{lll}
e_1 f(x_1) e_{1, p}\!\!\!\!&=(-1)^k \bar s_{p, 2} e_1 g(\bar x_1) \bar s_1 e_{1} \bar s_{1, p}=(-1)^k \bar s_{p, 2} e_1\bar s_1 g(\bar x_2' )  e_{1} \bar s_{1, p}\\[4pt]
\!\!\!\!&=(-1)^k \bar s_{p, 2} e_1 g(\bar {x}_2') \bar s_{1,p}=(-1)^k \bar s_{p, 2} e_1 \bar s_{1, p} g(\bar x_1)\in J_L.\end{array}$$
So, $  e_1 f(x_1)c^{\alpha}  e_{ 1, p}=\prod_{i=2}^r c_i^{\alpha_i} e_1 f(x_1) c_1^{\alpha_1} e_{1, p}=\prod_{i=2}^r c_i^{\alpha_i}\bar c_p^{\alpha_1} e_1 f(x_1)  e_{1, p}
\in J_L$.  By (2) and the equation $g(\bar x_1)\prod_{k=2}^f e_{i_k, j_k} =\prod_{k=2}^f e_{i_k, j_k} g(\bar x_1)$,  $e_1f(x_1)\textbf{m}\in J_L$.

\item  Finally, suppose   $\{ i_l, j_l\}\cap \{1\}=\emptyset$ for all possible $l$, then  $e_1f(x_1)\textbf{m}\in J_L$  follows from (2) and the following fact
$$e_1 {\textbf{\textit{f}}}(x_1)\mbox{$ \prod\limits_{ l=1}^f e_{i_f, j_f}=\prod\limits_{ l=1}^f e_{i_f, j_f}e_1 {\textbf{\textit{f}}}(x_1) =(-1)^k \prod\limits_{ l=1}^f $}e_{i_f, j_f} e_1 {\textbf{\textit{g}}}(\bar x_1)\in J_L.$$
\end{itemize}

 \textit{Case~2: $|\beta|>0$.}

 Suppose  $\beta_i\neq 0$ for some $2\le i\le r$.
 We have $x_i\textbf{m}'=\epsilon \textbf{m}$ by Lemma~\ref{hecrel}(1)--(2), where $\textbf{m}'$ is  obtained from $\textbf{m}$ by removing one $x_i$ and $\epsilon=\pm 1$. So   $e_1f(x_1)\textbf{m}\in J_L+J_R$ if   $e_1f(x_1)x_i \textbf{m}'\in J_L+J_R$.
  Since  $f(x_1)x_i=x_if(x_1)$, it suffices to  prove $e_1 L_i f(x_1)\textbf{m}'\in J_L+J_R$ and   $e_1 (x_i-L_i) f(x_1)\textbf{m}'\in J_L+J_R$.

  In the first case,
  since $e_1(j, i)=(j, i) e_1$ if ${j\neq 1}$ and $e_1c_i=c_i e_1$ and $c_if(x_1)=f(x_1)c_i$, by inductive assumption on  $|\beta|$ and (1)--(2), it is enough to prove $e_1 (1, i) f(x_1) c_i^l \textbf{m}'\in J_R$ for $l\in \Z_2$. In fact, it is the case since
    $$e_1 (1, i) f(x_1) c_i^l\textbf{m}'=e_1 f(x_i') (1, i) c_i^l   \textbf{m}'=f(x_i') e_1 (1, i)  c_i^l \textbf{m}' \in J_R.$$
    In the second case, since $e_1 (x_i-L_i) f(x_1)\textbf{m}'=(x_i-L_i) e_1f(x_1)\textbf{m}'$, by induction on $\b$, $e_1f(x_1)\textbf{m}'\in J_L+J_R$.
    By (1), we have  $(x_i-L_i) e_1f(x_1)\textbf{m}'\in J_L+J_R$. So,   $ e_1f(x_1)\textbf{m}\in J_L+J_R$.

If $\b_i=0$, $2\le i\le r$, then  $x^\b=x_1^{\b_1}$ with $\b_1>0$.  In this case, $\textbf{m}=c^\alpha x_1^{\b_1} e_{\bii, \bij} w \bar x^\gamma \bar c^\delta$.
We want to prove  $v=e_1f(x_1) \textbf{m}\in J_L+J_R$.
If $j_\ell\neq 1$, $1\le \ell\le f$, then by  inductive assumption,
$$\begin{aligned}
v&=e_1 f(x_1) c^\alpha x_1^{\b_1}e_{\bii,\bij} w\bar x^\gamma \bar c^\delta
= (-1)^k  e_1g(\bar x_1) c^\alpha x_1^{\b_1}  e_{\bii,\bij} w\bar x^\gamma \bar c^\delta
 \equiv (-1)^k  e_1 c^\alpha  x_1^{\b_1}g(\bar x_1)e_{\bii,\bij} w\bar x^\gamma \bar c^\delta \\ &
= (-1)^k e_1c^\alpha x_1^{\b_1}e_{\bii,\bij}{\sc\,}g(\bar x_{1}) w{\bar x}^\gamma \bar c^\delta \nonumber \in J_L w {\bar x}^\gamma \bar c^\delta\subset J_L+J_R,\\
\end{aligned}$$
where the ``\,$\equiv$\,'' is  modulo $J_L+J_R$.
Finally, if $j_\ell=1$ for some $\ell$,  without loss of any generality, we assume  $j_1=1$.  If $i_1=1$,
we have $e_1f(x_1) c^\alpha x_1^\beta e_1=0$ no matter whether $\alpha_1=1$ or $0$. In the first case, this result follows from the equation $e_1x_1^k c_1 e_1=0$ for all $k\in \mathbb N$. In the second case, this result follows from the fact that $\widetilde{BC}_{r, t}$ is admissible.  It remains to deal with the cases when $i_l\neq 1$ for all $l$.
Define $ i'=(i_2,\ldots, i_f)$ and $ j'=(j_2,\ldots ,j_f)$.
Then
$$\begin{aligned} v&= e_1 f(x_1)c^\alpha   x_1^{\b_1}e_{i_1, 1}e_{\bii',\bij'} w\bar x^\gamma \bar c^\delta
=\mbox{$\prod\limits_{i=2}^r$} c_i^{\alpha_i} e_1f (x_1) e_{i_1,1} c_1^{\alpha_1}  x_1^{\b_1}e_{\bii',\bij'} w\bar x^\gamma \bar c^\delta\\
&=\mbox{$\prod\limits_{i=2}^r$} c_i^{\alpha_i} e_1 e_{i_1, 1} f(x_1) c_1^{\alpha_1} x_1^{\b_1}e_{\bii',\bij'} w\bar x^\gamma \bar c^\delta
=\mbox{$\prod\limits_{i=2}^r$} c_i^{\alpha_i}  e_1(1,i_1)f(x_1)c_1^{\alpha_1} x_1^{\beta_1}e_{\bii',\bij'} w\bar x^\gamma \bar c^\delta\\
&= \mbox{$\prod\limits_{i=2}^r$}  {c_i^{\alpha_i}  e_{1}f (x_{i_1}') (1, i) c_1^{\alpha_1} x_1^{\beta_1}e_{\bii',\bij'} w\bar x^\gamma \bar c^\delta}
\in \mbox{$\prod\limits_{i=2}^r$} c_i^{\alpha_i} J_R\subset J_R, \text{ by (1).}
\\
\end{aligned}$$
This completes the proof of \eqref{stepp2} and hence $h J_R\!\subset \!J_L\!+\!J_R$. One can similarly check $J_L h\!\subset \!J_L\!+\!J_R$.  \end{proof}

For $(\alpha, \beta)\in \mathbb N^r\times \mathbb N^t$, let
${\textbf{\textit{f}}}(x')^\alpha=f_1^{\alpha_1}\cdots f_r^{\alpha_r}$ and
${\textbf{\textit{g}}}(\bar x')^\beta=g_1^{\beta_1}\cdots g_t^{\beta_t}$. Recall that $l=k+2m$.

   \begin{Lemma}\label{mathcalncycl} The affine walled Brauer-Clifford superalgebra $\widetilde{BC}_{r,t}$ is a free $R$-module with $\mathcal N$ as its $R$-basis, where
 \begin{eqnarray} \label{regs1}&\!\!\!\!\!\!\!\!&
  \mathcal N= \stackrel{\sc \min\{ r, t\}}{\underset{\sc f=0}{\dis\mbox{\Large$\cup$}}}
  \{ {\textbf{\textit{f}}}(x')^\alpha  {c}^{\tilde \gamma} {x}^{\gamma}  d_1^{-1} e^f w d_2 \bar x^{\delta}{\bar c}^{\tilde \delta} {\textbf{\textit{g}}}(\bar x')^\b
 \mid (\alpha, \beta)\in\mathbb N^r \times  \mathbb N^t, (\gamma, \delta ,\tilde \gamma,\tilde \delta)\in\mathbb Z_l^r \times  \mathbb Z_l^t\times \Z_2^r \times  \Z_2^t, \nonumber\\[-10pt]
 &\!\!\!\!\!\!\!\!&
\phantom{ \mathcal N= \stackrel{\sc \min\{m, n\}}{\underset{\sc f=0}{\dis\mbox{\Large$\cup$}}}
  \{}
 d_1, d_2\in D_{r, t}^f,
  w\in \Sigma_{r-f}\times \bar{\Sigma}_{{t-f}}\}.\end{eqnarray}\end{Lemma}
\begin{proof} The result follows from  Theorem~\ref{main1cyco321} since the transition matrix between $\mathcal N$ and the basis in Theorem~\ref{main1cyco321}
is upper-unitriangular.
\end{proof}

\begin{Lemma}\label{I}Let $I$ be the two-sided ideal of  $\widetilde{BC}_{r,t}$ generated by
$f(x_1)$ and $g(\bar x_1)$ satisfying \eqref{polyoff}--\eqref{polyofg}. If  $ \widetilde{BC}_{r,t}$ is admissible, then $S$ is an $R$-basis of $I$, where
\begin{equation}\label{sbasiscycl} S=\{ {\textbf{\textit{f}}}(x')^\alpha {c}^{\tilde \gamma} {x}^{\gamma}  d_1^{-1} e^f w d_2 \bar x^{\delta}{\bar c}^{\tilde \delta}{\textbf{\textit{g}}} (\bar x')^\b\in\mathcal N\mid \alpha_i+\b_j\neq 0 \text{ for some $i, j$}\}.\end{equation}
\end{Lemma}
\begin{proof}
Let $M$ be the $R$-module spanned by $ S$. Obviously, $M\subseteq I$. If   $J_L\subseteq M$, and $J_R\subseteq M$, by Lemma~\ref{2ijcycl}(3),  $M=I$, proving the result.
By symmetry, we verify  $J_R\subseteq M$. By Lemma~\ref{mathcalncycl}, we need to verify  $f(x_i') \textbf{m}\in M$ for any basis element $\textbf{m}$ in \eqref{regs1}.
In fact,  we have
$$f(x_i')f(x_\ell')\in f(x_\ell') f(x_i')+\mbox{$\sum\limits_{j=1}^{\ell-1}$} f(x_j') \widetilde{BC}_{r, t}, \text{ by Lemma~\ref{zero-12}}. $$
Using induction on degrees, we have $f(x_i') \textbf{m}\in M$. Finally, one can check $J_L\subset M$, similarly.
\end{proof}

\begin{Theorem}\label{level-k-walledcycl} The cyclotomic walled Brauer-Clifford superalgebra  $ BC_{k+2m, r, t}$ is free over $R $ with rank  $2^{r+t}(k+2m)^{r+t}(r+t)!$ if and only if
 $BC_{k+2m, r, t}$ is admissible.
\end{Theorem}

\begin{proof}By Corollary~\ref{level-l-span}, $ BC_{k+2m, r, t}$ is spanned by all of its regular monomials.
If $BC_{k+2m, r,t}$ is not admissible,  $e_1$ is an $R$-torsion element by  Corollary~\ref{assump1}. Since  $e_1\in M$, either $BC_{k+2m, r, t}$ is not free or the rank of $BC_{k+2m, r, t}$ is strictly less than  $2^{r+t}(k+2m)^{r+t}(r+t)!$, the number of all regular monomials of $ BC_{k+2m, r, t}$.
 If  $BC_{k+2m, r,t}$ is  admissible, by  Lemmas~\ref{mathcalncycl}--\ref{I},
 all regular monomials of  $BC_{k+2m, r,t}$ are $R$-linear independent. So $BC_{k+2m, r, t}$ is free over $R$ with rank  $2^{r+t}(k+2m)^{r+t}(r+t)!$. \end{proof}

\section{Isomorphisms between affine (resp.,  cyclotomic)  Brauer-Clifford algebras and Comes-Kujawa's affine (resp., cyclotomic) algebras}

Let $\mathscr {AOBC}$ be the degenerate affine oriented Brauer-Clifford supercategory defined in \cite[Definition~3.7]{CK}.
All notations on  $\mathscr {AOBC}$ we use here are the same as those in \cite{CK}.
For any positive integers  $r,t$, let $\mathscr {BC}_{r,t}^{\rm aff}:= \End_{\mathscr {AOBC}}(\uparrow^r\downarrow^t)$.
The aim of this section is to establish an isomorphism between certain quotient of $\mathscr {BC}_{r,t}^{\rm aff}$ and our $\widetilde{BC}_{r, t}$. First, we recall some of results in   \cite[Section~3.2]{CK}. We use the  notations in \cite{CK} to denote the morphisms in  $\mathscr {AOBC}$.

Given a nonnegative integer $r$ and any morphism $f$,  write
\begin{equation*}
    \begin{tikzpicture}[baseline = 12pt,scale=0.5,color=\clr,inner sep=0pt, minimum width=11pt]
        \draw[-,thick] (0,0) to (0,2);
        \draw (0,1) node[circle,draw,thick,fill=white]{$f$};
    \end{tikzpicture}^r
    =~
    \underbrace{
        \begin{tikzpicture}[baseline = 12pt,scale=0.5,color=\clr,inner sep=0pt, minimum width=11pt]
            \draw[-,thick] (0,0) to (0,2);
            \draw (0,1) node[circle,draw,thick,fill=white]{$f$};
        \end{tikzpicture}
        \cdots
        \begin{tikzpicture}[baseline = 12pt,scale=0.5,color=\clr,inner sep=0pt, minimum width=11pt]
            \draw[-,thick] (0,0) to (0,2);
            \draw (0,1) node[circle,draw,thick,fill=white]{$f$};
        \end{tikzpicture}
    }_{r\text{ times}}
    ~.
\end{equation*}
In the monoidal supercategory,  the super-interchange law is represented by
\begin{equation}\label{super-interchange}
    \begin{tikzpicture}[baseline = 19pt,scale=0.5,color=\clr,inner sep=0pt, minimum width=11pt]
        \draw[-,thick] (0,0) to (0,3);
        \draw[-,thick] (2,0) to (2,3);
        \draw (0,2) node[circle,draw,thick,fill=white]{$f$};
        \draw (2,1) node[circle,draw,thick,fill=white]{$g$};
    \end{tikzpicture}
    ~=~
    \begin{tikzpicture}[baseline = 19pt,scale=0.5,color=\clr,inner sep=0pt, minimum width=11pt]
        \draw[-,thick] (0,0) to (0,3);
        \draw[-,thick] (2,0) to (2,3);
        \draw (0,1.5) node[circle,draw,thick,fill=white]{$f$};
        \draw (2,1.5) node[circle,draw,thick,fill=white]{$g$};
    \end{tikzpicture}
    ~=(-1)^{\p{f}\p{g}}~
    \begin{tikzpicture}[baseline = 19pt,scale=0.5,color=\clr,inner sep=0pt, minimum width=11pt]
        \draw[-,thick] (0,0) to (0,3);
        \draw[-,thick] (2,0) to (2,3);
        \draw (0,1) node[circle,draw,thick,fill=white]{$f$};
        \draw (2,2) node[circle,draw,thick,fill=white]{$g$};
    \end{tikzpicture}
    ~.
\end{equation}

\begin{Lemma}\label{CKu} $\mathscr{BC}_{r,t}^{\rm aff}$ is an associative  superalgebras  generated by  even generators $e_1$, $s_i (1\leq i\leq r-1)$, $\bar s_j (1\leq j\leq t-1)$
 $y_i (1\leq i\leq r)$ and $ \bar y_j (1\leq j\leq t)$, and odd generators $c_i (1\leq i\leq r)$ and $\bar c_i (1\leq i\leq t)$, and even central elements
  $\bar\delta_k, (k\in\mathbb Z_{\geq0})$ defined as follows:
\begin{equation}\label{gens for AOBC}
\begin{aligned}&
   e_1=
   %\begin{tikzpicture}[baseline = 5pt, scale=0.75, color=\clr]
%        \end{tikzpicture}^{r-1}~
         \begin{tikzpicture}[baseline = 5pt, scale=0.75, color=\clr]
             \draw[->,thick] (1,0) to[out=up, in=down] (1,1);     \draw (1.5,1.0) node{\footnotesize{$r-1$}};
              \draw[<-,thick] (0,1) to[out=down,in=left] (1,0.6) to[out=right,in=down] (2.2,1);
          \draw[<-,thick] (2.2,0) to[out=up,in=right] (1,0.4) to[out=left,in=up] (0,0);\end{tikzpicture} ~
             \begin{tikzpicture}[baseline = 5pt, scale=0.75, color=\clr]
            \draw[<-,thick] (2.25,0) to[out=up, in=down] (2.25,1);
        \end{tikzpicture}^{t-1}
    ~,\quad  s_i=\begin{tikzpicture}[baseline = 5pt, scale=0.75, color=\clr]
            \draw[->,thick] (0,0) to[out=up, in=down] (0,1);
        \end{tikzpicture}^{i-1}~
    \begin{tikzpicture}[baseline = 5pt, scale=0.75, color=\clr]
                \draw[->,thick] (1,0) to[out=up,in=down] (0.5,1);
        \draw[->,thick] (0.5,0) to[out=up,in=down] (1,1);
                  \end{tikzpicture}~
                   \begin{tikzpicture}[baseline = 5pt, scale=0.75, color=\clr]
            \draw[->,thick] (2.25,0) to[out=up, in=down] (2.25,1);
        \end{tikzpicture}^{r-i-1}~
                  \begin{tikzpicture}[baseline = 5pt, scale=0.75, color=\clr]
            \draw[<-,thick] (2.25,0) to[out=up, in=down] (2.25,1);
        \end{tikzpicture}^{t}
    ~,\quad
    \bar s_i=\begin{tikzpicture}[baseline = 5pt, scale=0.75, color=\clr]
            \draw[->,thick] (0,0) to[out=up, in=down] (0,1);
        \end{tikzpicture}^{r}~
        \begin{tikzpicture}[baseline = 5pt, scale=0.75, color=\clr]
            \draw[<-,thick] (2.25,0) to[out=up, in=down] (2.25,1);
        \end{tikzpicture}^{i-1}~
    \begin{tikzpicture}[baseline = 5pt, scale=0.75, color=\clr]
                \draw[->,thick](0.5,1) to[out=down,in=up](1,0);
        \draw[->,thick] (1,1)to[out=down,in=up](0.5,0);
                  \end{tikzpicture}~
                  \begin{tikzpicture}[baseline = 5pt, scale=0.75, color=\clr]
            \draw[<-,thick] (2.25,0) to[out=up, in=down] (2.25,1);
        \end{tikzpicture}^{t-i-1}~,\\
          &  y_i= \begin{tikzpicture}[baseline = 5pt, scale=0.75, color=\clr]
            \draw[->,thick] (0,0) to[out=up, in=down] (0,1);
        \end{tikzpicture}^{i-1}~
                \begin{tikzpicture}[baseline = 5pt, scale=0.75, color=\clr]
                       \draw[->,thick] (0.5,0) to (0.5,1);
               \draw (0.5,0.5) \bdot;
    \end{tikzpicture}~
    \begin{tikzpicture}[baseline = 5pt, scale=0.75, color=\clr]
            \draw[->,thick] (0,0) to[out=up, in=down] (0,1);
        \end{tikzpicture}^{r-i}~
    \begin{tikzpicture}[baseline = 5pt, scale=0.75, color=\clr]
            \draw[<-,thick] (2.25,0) to[out=up, in=down] (2.25,1);
        \end{tikzpicture}^{t}
    ~,\quad
    c_i=\begin{tikzpicture}[baseline = 5pt, scale=0.75, color=\clr]
            \draw[->,thick] (0,0) to[out=up, in=down] (0,1);
        \end{tikzpicture}^{i-1}~
                \begin{tikzpicture}[baseline = 5pt, scale=0.75, color=\clr]
                       \draw[->,thick] (0.5,0) to (0.5,1);
               \draw (0.5,0.5) \wdot;
    \end{tikzpicture}~
    \begin{tikzpicture}[baseline = 5pt, scale=0.75, color=\clr]
            \draw[->,thick] (0,0) to[out=up, in=down] (0,1);
        \end{tikzpicture}^{r-i}~
    \begin{tikzpicture}[baseline = 5pt, scale=0.75, color=\clr]
            \draw[<-,thick] (2.25,0) to[out=up, in=down] (2.25,1);
        \end{tikzpicture}^{t},\\
   &\bar y_i= \begin{tikzpicture}[baseline = 5pt, scale=0.75, color=\clr]
            \draw[->,thick] (0,0) to[out=up, in=down] (0,1);
        \end{tikzpicture}^{r}~
    \begin{tikzpicture}[baseline = 5pt, scale=0.75, color=\clr]
            \draw[<-,thick] (0,0) to[out=up, in=down] (0,1);
        \end{tikzpicture}^{i-1}~
           \begin{tikzpicture}[baseline = 5pt, scale=0.75, color=\clr]
                       \draw[->,thick] (0.5,1) to (0.5,0);
               \draw (0.5,0.5) \bdot;
    \end{tikzpicture}~
    \begin{tikzpicture}[baseline = 5pt, scale=0.75, color=\clr]
            \draw[<-,thick] (2.25,0) to[out=up, in=down] (2.25,1);
        \end{tikzpicture}^{t-i},
       ~\quad
      \bar c_i= \begin{tikzpicture}[baseline = 5pt, scale=0.75, color=\clr]
            \draw[->,thick] (0,0) to[out=up, in=down] (0,1);
        \end{tikzpicture}^{r}~
    \begin{tikzpicture}[baseline = 5pt, scale=0.75, color=\clr]
            \draw[<-,thick] (0,0) to[out=up, in=down] (0,1);
        \end{tikzpicture}^{i-1}~
           \begin{tikzpicture}[baseline = 5pt, scale=0.75, color=\clr]
                       \draw[->,thick] (0.5,1) to (0.5,0);
               \draw (0.5,0.5) \wdot;
    \end{tikzpicture}~
    \begin{tikzpicture}[baseline = 5pt, scale=0.75, color=\clr]
            \draw[<-,thick] (2.25,0) to[out=up, in=down] (2.25,1);
        \end{tikzpicture}^{t-i},\\
    & \bar\delta_k=\begin{tikzpicture}[baseline = 5pt, scale=0.5, color=\clr]
        \draw[<-,thick] (0.6,1) to (0.5,1) to[out=left,in=up] (0,0.5)
                        to[out=down,in=left] (0.5,0)
                        to[out=right,in=down] (1,0.5)
                        to[out=up,in=right] (0.5,1);
        \draw (1,0.5) \bdot;
        \draw (1.4,0.5) node{\footnotesize{$k$}};
    \end{tikzpicture}.~ \quad
 \end{aligned}
\end{equation}
These generators satisfy the following local relations:
\begin{equation}\label{ese}
e_1s_1\bar s_1e_1s_1=e_1s_1\bar s_1e_1\bar s_1,~  s_1e_1s_1\bar s_1e_1=\bar s_1 e_1s_1\bar s_1e_1,   e_1 s_1 e_1 =e_1\bar s_1 e_1=e_1,
\end{equation}
\begin{equation}\label{AOBC  rels 1}
    \begin{tikzpicture}[baseline = 10pt, scale=0.5, color=\clr]
        \draw[-,thick] (0,0) to[out=up, in=down] (1,1);
        \draw[->,thick] (1,1) to[out=up, in=down] (0,2);
        \draw[-,thick] (1,0) to[out=up, in=down] (0,1);
        \draw[->,thick] (0,1) to[out=up, in=down] (1,2);
            \end{tikzpicture}
    ~=~
    \begin{tikzpicture}[baseline = 10pt, scale=0.5, color=\clr]
        \draw[-,thick] (0,0) to (0,1);
        \draw[->,thick] (0,1) to (0,2);
        \draw[-,thick] (1,0) to (1,1);
        \draw[->,thick] (1,1) to (1,2);
    \end{tikzpicture}
    ,\qquad
    \begin{tikzpicture}[baseline = 10pt, scale=0.5, color=\clr]
        \draw[->,thick] (0,0) to[out=up, in=down] (2,2);
        \draw[->,thick] (2,0) to[out=up, in=down] (0,2);
        \draw[->,thick] (1,0) to[out=up, in=down] (0,1) to[out=up, in=down] (1,2);
            \end{tikzpicture}
    ~=~
    \begin{tikzpicture}[baseline = 10pt, scale=0.5, color=\clr]
        \draw[->,thick] (0,0) to[out=up, in=down] (2,2);
        \draw[->,thick] (2,0) to[out=up, in=down] (0,2);
        \draw[->,thick] (1,0) to[out=up, in=down] (2,1) to[out=up, in=down] (1,2);
           \end{tikzpicture},
\end{equation}
\begin{equation}\label{AOBC rels 2}
    \begin{tikzpicture}[baseline = 10pt, scale=0.5, color=\clr]
        \draw[->,thick] (1,1) to[out=down, in=up] (0,0);
        \draw[-,thick] (1,1) to[out=up, in=down] (0,2);
        \draw[->,thick] (0,1) to[out=down, in=up] (1,0);
        \draw[-,thick] (1,2) to[out=down, in=up] (0,1);
            \end{tikzpicture}
    ~=~
    \begin{tikzpicture}[baseline = 10pt, scale=0.5, color=\clr]
        \draw[<-,thick] (0,0) to (0,1);
        \draw[-,thick] (0,1) to (0,2);
        \draw[<-,thick] (1,0) to (1,1);
        \draw[-,thick] (1,1) to (1,2);
            \end{tikzpicture}
    ,\qquad
    \begin{tikzpicture}[baseline = 10pt, scale=0.5, color=\clr]
        \draw[<-,thick] (0,0) to[out=up, in=down] (2,2);
        \draw[<-,thick] (2,0) to[out=up, in=down] (0,2);
        \draw[<-,thick] (1,0) to[out=up, in=down] (0,1) to[out=up, in=down] (1,2);
            \end{tikzpicture}
    ~=~
    \begin{tikzpicture}[baseline = 10pt, scale=0.5, color=\clr]
        \draw[<-,thick] (0,0) to[out=up, in=down] (2,2);
        \draw[<-,thick] (2,0) to[out=up, in=down] (0,2);
        \draw[<-,thick] (1,0) to[out=up, in=down] (2,1) to[out=up, in=down] (1,2);
            \end{tikzpicture},
\end{equation}
\begin{equation}\label{AOBC scat rels 3}
    \begin{tikzpicture}[baseline = 7.5pt, scale=0.5, color=\clr]
        \draw[->,thick] (0,0) to[out=up, in=down] (0,1.5);
        \draw (0,1) \wdot;
        \draw (0,0.5) \wdot;
            \end{tikzpicture}
    ~=~
    \begin{tikzpicture}[baseline = 7.5pt, scale=0.5, color=\clr]
        \draw[->,thick] (0,0) to[out=up, in=down] (0,1.5);
            \end{tikzpicture}
    ,\qquad
    \begin{tikzpicture}[baseline = 7.5pt, scale=0.5, color=\clr]
        \draw[->,thick] (0,0) to[out=up, in=down] (1,1.5);
        \draw[->,thick] (1,0) to[out=up, in=down] (0,1.5);
        \draw (0.2,0.5) \wdot;
            \end{tikzpicture}
    ~=~
    \begin{tikzpicture}[baseline = 7.5pt, scale=0.5, color=\clr]
        \draw[->,thick] (0,0) to[out=up, in=down] (1,1.5);
        \draw[->,thick] (1,0) to[out=up, in=down] (0,1.5);
        \draw (0.8,1) \wdot;
            \end{tikzpicture}
    ~,
\end{equation}
\begin{equation}\label{AOBC scat rels 3-1}
    \begin{tikzpicture}[baseline = 7.5pt, scale=0.5, color=\clr]
        \draw[<-,thick] (0,0) to[out=up, in=down] (0,1.5);
        \draw (0,1) \wdot;
        \draw (0,0.5) \wdot;
            \end{tikzpicture}
    ~=-~
    \begin{tikzpicture}[baseline = 7.5pt, scale=0.5, color=\clr]
        \draw[<-,thick] (0,0) to[out=up, in=down] (0,1.5);
            \end{tikzpicture}
    ,\qquad
    \begin{tikzpicture}[baseline = 7.5pt, scale=0.5, color=\clr]
        \draw[<-,thick] (0,0) to[out=up, in=down] (1,1.5);
        \draw[<-,thick] (1,0) to[out=up, in=down] (0,1.5);
        \draw (0.2,0.5) \wdot;
            \end{tikzpicture}
    ~=~
    \begin{tikzpicture}[baseline = 7.5pt, scale=0.5, color=\clr]
        \draw[<-,thick] (0,0) to[out=up, in=down] (1,1.5);
        \draw[<-,thick] (1,0) to[out=up, in=down] (0,1.5);
        \draw (0.8,1) \wdot;
            \end{tikzpicture}
    ~,
\end{equation}
\begin{equation}\label{AOBC  rels 4}
   \begin{tikzpicture}[baseline = 7.5pt, scale=0.5, color=\clr]
        \draw[<-,thick] (0.6,1.25) to (0.5,1.25) to[out=left,in=up]
                        (0,0.75) to[out=down,in=left]
                        (0.5,0.25) to[out=right,in=down]
                        (1,0.75) to[out=up,in=right]
                        (0.5,1.25);
                 \draw (1,0.75) \bdot;\draw (1.5,0.75) node{\footnotesize{$2k$}};
           \end{tikzpicture}
    ~=0,\quad e_1 y_1^ke_1=\bar\delta_k e_1,\quad  e_1c_1 y_1^ke_1=0, \quad
    \begin{tikzpicture}[baseline = 7.5pt, scale=0.5, color=\clr]
        \draw[->,thick] (0,0) to[out=up, in=down] (0,1.5);
        \draw (0,1) \bdot;
        \draw (0,0.5) \wdot;
            \end{tikzpicture}
    ~=-~
    \begin{tikzpicture}[baseline = 7.5pt, scale=0.5, color=\clr]
        \draw[->,thick] (0,0) to[out=up, in=down] (0,1.5);
        \draw (0,1) \wdot;
        \draw (0,0.5) \bdot;
                \end{tikzpicture},~\quad
                \begin{tikzpicture}[baseline = 7.5pt, scale=0.5, color=\clr]
        \draw[<-,thick] (0,0) to[out=up, in=down] (0,1.5);
        \draw (0,1) \bdot;
        \draw (0,0.5) \wdot;
            \end{tikzpicture}
    ~=-~
    \begin{tikzpicture}[baseline = 7.5pt, scale=0.5, color=\clr]
        \draw[<-,thick] (0,0) to[out=up, in=down] (0,1.5);
        \draw (0,1) \wdot;
        \draw (0,0.5) \bdot;
                \end{tikzpicture}
\end{equation}

\begin{equation}\label{AOBC  rels 6}
    \begin{tikzpicture}[baseline = 7.5pt, scale=0.5, color=\clr]
        \draw[->,thick] (0,0) to[out=up, in=down] (1,1.5);
        \draw[->,thick] (1,0) to[out=up, in=down] (0,1.5);
        \draw (0.2,1) \bdot;
            \end{tikzpicture}
    ~-~
    \begin{tikzpicture}[baseline = 7.5pt, scale=0.5, color=\clr]
        \draw[->,thick] (0,0) to[out=up, in=down] (1,1.5);
        \draw[->,thick] (1,0) to[out=up, in=down] (0,1.5);
        \draw (0.8,0.5) \bdot;
                    \end{tikzpicture}
    ~=~
    \begin{tikzpicture}[baseline = 7.5pt, scale=0.5, color=\clr]
        \draw[->,thick] (0,0) to[out=up, in=down] (0,1.5);
        \draw[->,thick] (1,0) to[out=up, in=down] (1,1.5);
            \end{tikzpicture}
    ~-~
    \begin{tikzpicture}[baseline = 7.5pt, scale=0.5, color=\clr]
        \draw[->,thick] (0,0) to[out=up, in=down] (0,1.5);
        \draw[->,thick] (1,0) to[out=up, in=down] (1,1.5);
        \draw (0,0.7) \wdot;
        \draw (1,0.7) \wdot;
           \end{tikzpicture}
    ~,
\end{equation}

\begin{equation}\label{AOBC  rels 61}
    \begin{tikzpicture}[baseline = 7.5pt, scale=0.5, color=\clr]
        \draw[<-,thick] (0,0) to[out=up, in=down] (1,1.5);
        \draw[<-,thick] (1,0) to[out=up, in=down] (0,1.5);
        \draw (0.2,1) \bdot;
            \end{tikzpicture}
    ~-~
    \begin{tikzpicture}[baseline = 7.5pt, scale=0.5, color=\clr]
        \draw[<-,thick] (0,0) to[out=up, in=down] (1,1.5);
        \draw[<-,thick] (1,0) to[out=up, in=down] (0,1.5);
        \draw (0.8,0.5) \bdot;
                    \end{tikzpicture}
    ~=-~
    \begin{tikzpicture}[baseline = 7.5pt, scale=0.5, color=\clr]
        \draw[<-,thick] (0,0) to[out=up, in=down] (0,1.5);
        \draw[<-,thick] (1,0) to[out=up, in=down] (1,1.5);
            \end{tikzpicture}
    ~-~
    \begin{tikzpicture}[baseline = 7.5pt, scale=0.5, color=\clr]
        \draw[<-,thick] (0,0) to[out=up, in=down] (0,1.5);
        \draw[<-,thick] (1,0) to[out=up, in=down] (1,1.5);
        \draw (0,0.7) \wdot;
        \draw (1,0.7) \wdot;
           \end{tikzpicture}
    ~,
\end{equation}
\begin{equation}\label{AOBC  rels 7}
\begin{tikzpicture}[baseline = 5pt, scale=0.75, color=\clr]
                 \draw[<-,thick] (0.75,1) to[out=down,in=left] (1,0.6) to[out=right,in=down] (1.25,1);
          \draw[<-,thick] (1.25,0) to[out=up,in=right] (1,0.4) to[out=left,in=up] (0.75,0);
          \draw (0.8,0.8) \bdot; \end{tikzpicture}
          ~=~
          \begin{tikzpicture}[baseline = 5pt, scale=0.75, color=\clr]
                 \draw[<-,thick] (0.75,1) to[out=down,in=left] (1,0.6) to[out=right,in=down] (1.25,1);
          \draw[<-,thick] (1.25,0) to[out=up,in=right] (1,0.4) to[out=left,in=up] (0.75,0);
          \draw (1.2,0.8) \bdot; \end{tikzpicture}, ~
\begin{tikzpicture}[baseline = 5pt, scale=0.75, color=\clr]
                 \draw[<-,thick] (0.75,1) to[out=down,in=left] (1,0.6) to[out=right,in=down] (1.25,1);
          \draw[<-,thick] (1.25,0) to[out=up,in=right] (1,0.4) to[out=left,in=up] (0.75,0);
          \draw (0.8,0.2) \bdot; \end{tikzpicture}
          ~=~
          \begin{tikzpicture}[baseline = 5pt, scale=0.75, color=\clr]
                 \draw[<-,thick] (0.75,1) to[out=down,in=left] (1,0.6) to[out=right,in=down] (1.25,1);
          \draw[<-,thick] (1.25,0) to[out=up,in=right] (1,0.4) to[out=left,in=up] (0.75,0);
          \draw (1.2,0.2) \bdot; \end{tikzpicture},~
    \begin{tikzpicture}[baseline = 5pt, scale=0.75, color=\clr]
                 \draw[<-,thick] (0.75,1) to[out=down,in=left] (1,0.6) to[out=right,in=down] (1.25,1);
          \draw[<-,thick] (1.25,0) to[out=up,in=right] (1,0.4) to[out=left,in=up] (0.75,0);
          \draw (0.8,0.8) \wdot; \end{tikzpicture}
          ~=~
          \begin{tikzpicture}[baseline = 5pt, scale=0.75, color=\clr]
                 \draw[<-,thick] (0.75,1) to[out=down,in=left] (1,0.6) to[out=right,in=down] (1.25,1);
          \draw[<-,thick] (1.25,0) to[out=up,in=right] (1,0.4) to[out=left,in=up] (0.75,0);
          \draw (1.2,0.8) \wdot; \end{tikzpicture}, ~
\begin{tikzpicture}[baseline = 5pt, scale=0.75, color=\clr]
                 \draw[<-,thick] (0.75,1) to[out=down,in=left] (1,0.6) to[out=right,in=down] (1.25,1);
          \draw[<-,thick] (1.25,0) to[out=up,in=right] (1,0.4) to[out=left,in=up] (0.75,0);
          \draw (0.8,0.2) \wdot; \end{tikzpicture}
          ~=~
          \begin{tikzpicture}[baseline = 5pt, scale=0.75, color=\clr]
                 \draw[<-,thick] (0.75,1) to[out=down,in=left] (1,0.6) to[out=right,in=down] (1.25,1);
          \draw[<-,thick] (1.25,0) to[out=up,in=right] (1,0.4) to[out=left,in=up] (0.75,0);
          \draw (1.2,0.2) \wdot; \end{tikzpicture},~      \end{equation}

\begin{equation}\label{super commuting relations}
    \begin{tikzpicture}[baseline = 19pt,scale=0.5,color=\clr,inner sep=0pt, minimum width=11pt]
        \draw[-,thick] (0,0) to (0,3);
        \draw[-,thick] (1,0) to (1,3);
        \draw[-,thick] (2,0) to (2,3);
        \draw[-,thick] (3,0) to (3,3);
        \draw[-,thick] (4,0) to (4,3);
        \draw (1,2) node[circle,draw,thick,fill=white]{$g$};
        \draw (3,1) node[circle,draw,thick,fill=white]{$h$};
        \draw (0,-0.3) node{$\mathbf a$};
        \draw (2,-0.3) node{$\mathbf b$};
        \draw (4,-0.3) node{$\mathbf c$};
    \end{tikzpicture}
    ~=(-1)^{\p{g}\p{h}}
    \begin{tikzpicture}[baseline = 19pt,scale=0.5,color=\clr,inner sep=0pt, minimum width=11pt]
        \draw[-,thick] (0,0) to (0,3);
        \draw[-,thick] (1,0) to (1,3);
        \draw[-,thick] (2,0) to (2,3);
        \draw[-,thick] (3,0) to (3,3);
        \draw[-,thick] (4,0) to (4,3);
        \draw (1,1) node[circle,draw,thick,fill=white]{$g$};
        \draw (3,2) node[circle,draw,thick,fill=white]{$h$};
        \draw (0,-0.3) node{$\mathbf a$};
        \draw (2,-0.3) node{$\mathbf  b$};
        \draw (4,-0.3) node{$\mathbf c$};
    \end{tikzpicture}
    ~,
\end{equation}
for all $\mathbf a,\mathbf b,\mathbf c\in\wrd$ and $g,h\in\left\{\caup , \swap, \swapr, \xdot, \xdotr, \cldot, \cldotr \right\}$ such that the diagrams in \eqref{super commuting relations} are all in $\mathscr {BC}_{r,t}^{\rm aff}$.

\end{Lemma}
\begin{proof}
Relations \eqref{ese}--\eqref{AOBC scat rels 3-1} follow from \cite[Corollary~3.6]{CK}. In other words,   the generators $ e_1$, $s_1,\ldots, s_{r-1}$, $\bar s_1, \ldots, \bar s_{t-1}$, $c_1, \ldots, c_r$ and $\bar c_1, \ldots, \bar c_t$ satisfy the defining relations for ${BC}_{r,t}$ in Definition~\ref{wsera}.  Relations \eqref{AOBC  rels 6}--\eqref{AOBC  rels 7} follows from \cite[Propositions~3.3,~3.10]{CK}. Relations in \eqref{AOBC  rels 4} (resp., \eqref{AOBC scat rels 3}) is  \cite[3.29]{CK} (resp., \cite[3.27]{CK}).
 Finally, \eqref{super commuting relations} follows from \cite[7.8]{CK}. Comes and Kujawa have pointed that  the relations (7.2)--(7.8) are defining relations
 for $\mathscr {AOBC}$. So, the above relations consist of a complete set of relations for $\mathscr {BC}_{r,t}^{\rm aff}$.
\end{proof}

Recall that $\widetilde{BC}_{r,t}$ is  the   affine walled Brauer-Clifford superalgebra in Definition~\ref{affinewbcsa}.
Their defining parameters are $\omega_i, \bar\omega_j\in R$ such that $\omega_{2n}=\bar \omega_{2n}=0$ for all $n\in \mathbb N$.
Let $J$ be the two-sided
ideal of $\mathscr{BC}_{r,t}^{\rm aff}$ generated by $\bar\delta_k-(-1)^k\omega_k$ and $\bar \delta_k+\sum_{l=1}^{k-2} \bar \delta_l \delta_{k-l-1}-\bar \omega_k$,
where  $ \delta_i$ (cf. \cite[Remark~.11]{CK})
are determined by
\begin{equation}\label{delta123}
\delta_k-\bar \delta_k= \sum_{0< i< k/2}\delta_{2i-1}\bar \delta_{k-2i}.
\end{equation}
 Define
\begin{equation}\label{affck} \widetilde{\mathscr {BC}}_{r, t}=\mathscr{BC}_{r,t}^{\rm aff}/J.\end{equation}

\begin{Theorem}\label{isomofoabc}  $\widetilde{\mathscr{BC}}_{r,t} \cong \widetilde{BC}_{r,t} $ as $R$-superalgebras.
\end{Theorem}

\begin{proof}
Let $\varphi: \widetilde{\mathscr{BC}}_{r,t} \rightarrow  \widetilde{BC}_{r,t} $ be the $R$-linear map such that
$\varphi$ sends $e_1$, $s_i$'s, $\bar s_j$'s (resp., $c_k$'s and $\bar c_l$'s)  to the same symbols (resp., $\sqrt{-1} c_k$'s and $\sqrt{-1} \bar c_l$'s) in $\widetilde{BC}_{r,t} $.
We remark that $c_i^2=1$ and $\bar c_j^2=-1$ in $\widetilde{\mathscr{BC}}_{r,t} $ whereas $c_i^2=-1$ and $\bar c_j^2=1$ in $\widetilde{{BC}}_{r,t} $.
Moreover,
\begin{equation}\label{defofvar}
\begin{aligned}
&\varphi(y_1 )=-x_1,~ \varphi(y_i)=s_{i-1}\varphi(y_{i-1}) s_{i-1}+(1+c_{i-1}c_{i})s_{i-1},\\
&\varphi(\bar y_1 )=\bar x_1 -\mbox{$\sum\limits_{i=1}^r$} e_{i,1} +\mbox{$\sum\limits_{i=1}^r$} c_i e_{i,1}c_i,~ \varphi(\bar y_i )= \bar s_{i-1}\varphi(y_{i-1}) \bar s_{i-1}-(1-\bar c_{i-1}\bar c_{i})\bar s_{i-1},\\
\end{aligned}\end{equation}
where $e_{i,1}=s_{i,1} e_1 s_{1,i}$.
We claim   that $\varphi$ is an algebra homomorphism. If so, since $ \varphi$ sends even (resp., odd) generators of  $\widetilde{\mathscr{BC}}_{r,t}$ to
 even (resp., odd) generators of  $\widetilde{{BC}}_{r,t}$, $\varphi$ has to be an even homomorphism. This shows that $\varphi$ is a superalgebra homomorphism.

We verify that the images of generators of  $  \widetilde{\mathscr{BC}}_{r,t}$ satisfy the corresponding defining relations in Lemma~\ref{CKu}.
First note that $\varphi$ preserves  relations \eqref{ese}--\eqref{AOBC scat rels 3-1} and the last two relations in \eqref{AOBC  rels 7} since the generators $ e_1$, $s_i$'s, $\bar s_j$'s, $c_i$'s and $c_j$'s satisfy the relations for $\widetilde{BC}_{r,t}$ (cf. \cite[Corollary~3.6]{CK}).
 It follows from the  commutative relations in Definition~\ref{awbsa}  that $\varphi$ preserves \eqref{super commuting relations} and the last two relations in \eqref{AOBC  rels 4}. Since $e_1c_1x_1^ke_1=0$ and $ e_1x_1^{2k}e_1=0$,
 $\varphi$ preserves  the first  and third relations in \eqref{AOBC  rels 4}.  Since  $ e_1 x_1^{k}e_1= w_k e_1$ and $\bar\delta_k=(-1)^k\omega_k$ in $R$, $\varphi$ preserves the second  relation in \eqref{AOBC  rels 4}.
By \eqref{defofvar}, $\varphi$ preserves  \eqref{AOBC  rels 6}--\eqref{AOBC  rels 61}.
By \cite[(3.24)]{CK},
\begin{equation}\label{bary1p}
\bar y_1'=~\begin{tikzpicture}[baseline = 10pt, scale=0.5, color=\clr]
        \draw[->,thick] (1,0) to[out=up, in=down] (1,2);     \draw (1.1,2.0) node{\footnotesize{$r$}};
        \draw[<-,thick] (2,0) to[out=up, in=down] (0,1) to[out=up, in=down] (2,2);
          \draw (0,1) \bdot;  \end{tikzpicture}~
            \begin{tikzpicture}[baseline = 10pt, scale=0.5, color=\clr]
            \draw[<-,thick] (2.25,0) to[out=up, in=down] (2.25,2);
        \end{tikzpicture}^{t-1}
\end{equation}
where $\bar y_1'=\bar y_1+\sum_{i=1}^r e_{i,1}+\sum_{i=1}^rc_ie_{i,1}c_i$. Hence, $e_1 y_1=e_1\bar y_1'$ by \cite[(3.24)]{CK} and \eqref{AOBC  rels 7}.
So, $\varphi(e_1)\varphi (y_1)=\varphi(e_1)\varphi(\bar y_1')$ by Definition~\ref{awbsa}(4). Similarly, we have $\varphi (y_1)\varphi(e_1)=\varphi(\bar y_1')\varphi(e_1)$,
proving that $\varphi$ preserves
the first two relations in \eqref{AOBC  rels 7}.

 By definition of $\varphi$, any generator of $\widetilde{BC}_{r, t}$ has an inverse image in  $\widetilde{\mathscr {BC}}_{r, t}$. This shows that $\varphi$  is an epimorphism.
Now, we consider both  $\widetilde{BC}_{r, t}$  and   $\widetilde{\mathscr {BC}}_{r, t}$
over $F$, the quotient field.  Since  $\widetilde{\mathscr {BC}}_{r, t}$ has a basis which consists of all  equivalence classes of dotted normal ordered oriented Brauer-Clifford diagrams, and each representative of an equivalent class is mapped to a unique regular monomial in   $\text{gr}\widetilde{BC}_{r, t}$. So, $\phi$ is an isomorphism.
 Finally, we remark that one can use the freeness of  $\widetilde{BC}_{r, t}$ to prove the freeness of  $\widetilde{\mathscr {BC}}_{r, t}$ and vice versa via the isomorphism $\varphi$ or $\varphi^{-1}$.
 \end{proof}

Fix $a,b \in\mathbb Z^{\geq0}$ and $u_i\in R$ for each $1\leq i\leq a$. For a given polynomial $f(u)=u^b\prod_{i=1}^a (u^2-u_i)$,
Comes and Kujuawa introduced the cyclotomic quotoent $ \mathscr {OBC}^f$ which is the supercategory defined as the quotient of $\mathscr {AOBC}$ by the left tensor ideal generated by $f(\xdot ) $. In the following, we consider the category ( which is also denoted by $ \mathscr{OBC}^f$ )  obtained from the original $ \mathscr{OBC}^f$   in \cite{CK} by specializing $\bar \delta_i$ at $(-1)^i w_i$.
By \cite[Remark~3.11,~Proposition~3.12]{CK},
\begin{equation}\label{delfor1}
(1-\mbox{$\sum\limits_{i\geq 1}$}\delta_{i-1}u^{-i})(1+\mbox{$\sum\limits_{j\geq 1}$}\bar \delta_{j-1}u^{-j})=1.
\end{equation}
Motivated by \cite[Remark~1.6]{BCNR}, for another given monic polynomial $g(x)$ of degree $\ell:=b+2a$, set
\begin{equation}\label{delfor2}
1+\mbox{$\sum\limits_{j\geq 1}$}\bar \delta_{j-1}u^{-j}:= \frac{g(u)}{f(u)}.
\end{equation}
Write $f(u)=u^\ell+a_{\ell-1} u^{\ell-1}+ \cdots + a_1 u+a_0$.  Then
\begin{equation}\label{fofy}
y:=(d \downarrow)\circ t \circ f(\xdot)\circ t^{-1} \circ ( \downarrow c)= y_\ell + a_{\ell-1}y_{\ell-1} +\cdots+ a_1 y_1 +a_0 \in \End_{\mathscr {AOBC}}(\downarrow),
\end{equation}
where $ d=\lcap$, $c=\lcup$ and $ t=\rswap$ and $y_k=(d \downarrow)\circ t \circ (\xdot)^k \circ t^{-1} \circ ( \downarrow c)$, $1\leq k\leq \ell$.
By \cite[(3.24)]{CK},
$$y_k=(\xdotr)^k- \sum_{i=0}^{k-1} \bar \delta_i (\xdotr)^{k-i+1}.$$
By \eqref{delfor2}, one can check that $y=g(\xdotr)$. So, the right tensor ideal generated by $f(\xdot)$ contains $g(\xdotr)$. Similarly, we also
have that the right tensor  ideal generated by  $g(\xdotr)$ also contains $f(\xdot)$.
Let $\widetilde{ \mathscr{BC}}_{r,t}^f:= \widetilde {\mathscr{BC}}_{r,t}/J$, where $ J$ is the two sided ideal of $\widetilde {\mathscr{BC}}_{r,t} $ generated by $f(y_1)$ and $g(\bar y_1')$, and
$\bar y_1'=\bar y_1+\sum_{i=1}^r e_{i,1}+\sum_{i=1}^rc_ie_{i,1}c_i$.
Then $\widetilde{\mathscr{BC}}_{r,t}^f \cong \End_{\mathscr {OBC}^f }(\uparrow^r \downarrow^t)$ by \cite[(7.9)]{CK}.
By the definition of  $\varphi$ in the proof of Theorem~\ref{isomofoabc}, we have $\varphi(f(y_1))=\tilde f(x_1)$, where $\tilde f(u)=(-1)^\ell f(-u)$.

The fact that $ \varphi(g(\bar y'_1))=g(\bar x_1)$ follows from \cite[(3.24)]{CK}.
Finally, $e_1\tilde f(x_1)=(-1)^be_1 g(\bar x_1)$ follows from \eqref{fofy}, \eqref{AOBC  rels 7} and the fact that $\varphi$ is an isomorphism in Theorem ~\ref{isomofoabc}.
This leads to the following result.

\begin{Theorem}\label{isomofoabcf} Let $BC_{\ell, r, t}:= \widetilde {BC}_{r,t}/I$, where $\widetilde {BC}_{r,t}$ is given in Theorem~$\ref{isomofoabc}$ and $I$ is the  two-sided ideal generated by  $\tilde f(x_1)$ and $g(\bar x_1)$. Then  $ BC_{\ell, r, t}$ is admissible and
 $ BC_{\ell, r,t}\cong \widetilde{\mathscr {BC}}_{r,t}^f$. In particular,   $\widetilde{\mathscr {BC}}_{r,t}^f$ is free over $R$ with rank $\ell^{r+t} 2^{r+t} (r+t)!$, where $\ell$ is the degree of $\tilde f(x_1)$.
\end{Theorem}
\begin{proof} The isomorphism $\varphi$ in the proof of Theorem~7.2 leads to an $R$-epimorphism   $\bar \varphi: \widetilde{\mathscr {BC}}_{r,t}^f\rightarrow BC_{\ell, r, t}$.
Comes and Kujawa have proved that $\widetilde{\mathscr{BC}}_{r,t}^f$ is generated by all the equivalence classes of normally ordered dotted oriented Brauer-Clifford diagrams without bubbles of type $\mathbf a\rightarrow \mathbf a$ with fewer than $\ell$ dots on each strands, where $\mathbf a = \uparrow^r \downarrow^t$.
 Moreover, $\varphi$ sends all these generators to the corresponding regular  { monomials in} $ \text{gr} (BC_{\ell,r,t})$. One can check that  $BC_{\ell, r,t}$ is admissible by \eqref{delfor2}. So,  by Theorem~\ref{level-k-walledcycl},  $ \varphi$ is injective and hence $ BC_{\ell,r,t}\cong \widetilde{\mathscr {BC}}_{r,t}^f$. In particular,  $\widetilde{\mathscr {BC}}_{r,t}^f$ is free over $R$ with rank $\ell^{r+t} 2^{r+t} (r+t)!$.
\end{proof}

\end{document}